\documentclass[a4paper,12pt,reqno]{amsart}

\usepackage{a4wide}
\usepackage{latexsym,amssymb}
\usepackage{upgreek}
\usepackage[hidelinks]{hyperref}
\usepackage{amsmath,amsthm,amsxtra,amsbsy}
\usepackage{amsfonts,eucal}
\usepackage{mathrsfs}
\usepackage{xcolor}
\usepackage{enumitem}
\usepackage{bbm} 
\usepackage{longtable}
\usepackage{soul}
\usepackage{stmaryrd} 
\numberwithin{equation}{section}

\usepackage{fancyhdr}
\usepackage{datetime}
 \def\calB{{\mathcal B}} 
 \def\calE{{\mathcal E}} \def\calF{{\mathcal F}}

\def\calM{{\mathcal M}}  
  \def\calR{{\mathcal R}}
\def\calS{{\mathcal S}}  
 \def\calW{{\mathcal W}} 
 
 \def\rme{{\mathrm e}}

\def\rmA{{\mathrm A}} \def\rmB{{\mathrm B}} \def\rmC{{\mathrm C}}
\def\rmD{{\mathrm D}}  \def\rmF{{\mathrm F}}
\def\rmG{{\mathrm G}}

 \def\fw{{\mathrm W}}

\def\sfS{{\mathsf S}}
\def\sfp{{\mathsf p}}
\def\sfx{{\mathsf x}}
\def\sfy{{\mathsf y}}
\def\sfs{{\mathsf s}} 
\def\sfR{{\mathsf R}}
\def\sfE{{\mathsf E}}

\def\fw{\mathrm F}
\newcommand\tj{\bj^\flat} 

\newtheorem{theorem}{Theorem}[section]
\newtheorem{prop}[theorem]{Proposition}
\newtheorem{problem}[theorem]{Problem}
\newtheorem{cor}[theorem]{Corollary}
\newtheorem{lemma}[theorem]{Lemma}
\newtheorem{Flemma}[theorem]{Formal Lemma}
\theoremstyle{definition}
\newtheorem{definition}[theorem]{Definition}

\newtheorem{xremark}[theorem]{Remark}
\newenvironment{remark}%
  {\pushQED{\qed}\begin{xremark}}
  {\popQED\end{xremark}}
\newtheorem{xample}[theorem]{Example}
\newenvironment{example}%
  {\pushQED{\qed}\begin{xample}}
  {\popQED\end{xample}}

\usepackage{framed}
\usepackage{ifthen,enumitem}
\makeatletter
\newcommand{\showfont}{encoding: \f@encoding{},  
  family: \f@family{},
  series: \f@series{},
  shape: \f@shape{},
  size: \f@size{}
}
\newcommand{\iffont}[3]{\ifthenelse{\equal{\f@shape}{#1}}{#2}{#3}}
\newcounter{Assumptions}
\newenvironment{Assumptions}[1]
  {\renewcommand\theAssumptions{{\bf(#1)}}%
   \protected@edef\@currentlabel{{\bf(#1)}}%
   \begin{framed}\noindent\textbf{Assumption~\theAssumptions.}}
  {\end{framed}}
\makeatother

\usepackage{mathtools}

\DeclareFontFamily{U}{matha}{\hyphenchar\font45}
\DeclareFontShape{U}{matha}{m}{n}{ <-6> matha5 <6-7> matha6 <7-8>
matha7 <8-9> matha8 <9-10> matha9 <10-12> matha10 <12-> matha12 }{}
\DeclareSymbolFont{matha}{U}{matha}{m}{n}
\DeclareMathSymbol{\abxrightharpoonup}{\mathrel}{matha}{"E1}

\newcommand{\OLI}{\color{red}}
\definecolor{ddcyan}{rgb}{0,0.1,0.9}
\definecolor{ddmagenta}{rgb}{0.7,0,0.9}
\definecolor{purple}{rgb}{0.5,0,0.9}

\newcommand{\GGG}{\color{ddcyan}}
\newcommand{\GGGO}{\nc}
\newcommand{\nc}{\normalcolor}

\newcommand{\EEE}{\color{black}}

\let\GGG\relax
\let\GGGO\relax

\let\OLI\relax

\newcommand{\RICKYNEW}{\color{purple}}
\let\RICKYNEW\relax

\newcommand{\ep}{\varepsilon}

\let\eps\ep
\newcommand{\R}{\mathbb{R}}
\newcommand{\N}{\mathbb{N}}

\newcommand{\teta}{{\boldsymbol\vartheta}}

\newcommand{\foraa}{\text{for a.e.\ }}
\def\One{\mathbbm{1}}
\def\dd{\mathrm{d}}
\DeclareMathSymbol{\mtimes}{\mathord}{symbols}{"0A}
\newcommand{\argmin}{\mathop{\rm argmin}}

\newcommand{\edg}{E}

\newcommand{\CB}{\mathrm{C}_{\mathrm{b}}}
\newcommand{\dnabla}{\overline\nabla}
\newcommand{\ona}{\dnabla}
\newcommand{\odiv}[1]{\dnabla \cdot(#1)}
\renewcommand{\odiv}[1]{\mathop{\overline{\mathrm{div}}}#1}
\newcommand{\odivn}{\overline{\mathrm{div}}}
\newcommand{\DVT}[3]{\mathscr{W}(#1,#2,#3)}
\newcommand{\CER}[2]{\mathcal{A}{(#1,#2)}}
\newcommand{\CEIR}[1]{\mathcal{A}{(#1)}}
\newcommand{\ADM}[4]{\mathscr{A}{(#1,#2;#3,#4)}}
\newcommand{\DVTn}{\mathscr{W}}
\newcommand{\VarW}[3]{\mathbb W(#1;[#2,#3])}
\newcommand{\VarWn}{\mathbb W}

\newcommand{\GMM}[2]{\mathrm{GMM}(#1;#2)}
\newcommand{\GMMT}[4]{\mathrm{GMM}(#1;(#2, #3),#4)}
\newcommand{\nuovorel}{\mathscr{S}^-}

\newcommand{\Cb}{\mathrm{C}_{\mathrm{b}}}
\newcommand{\Bb}{\mathrm{B}_{\mathrm b}}

\newcommand{\Fish}{\mathscr{D}}

\newcommand{\weakto}{\rightharpoonup}

\newcommand{\weaksigmatoabs}{\stackrel{\sigma}{\rightharpoonup}}
\newcommand{\hj}{\bj'}

\newcommand\cw{w^\flat} 

\newcommand\symmap{\mathsf s_\#}
\newcommand\symmapn{\mathsf s}

\newcommand{\CE}[2]{\mathcal{CE}(#1,#2)}
\newcommand{\CEI}[1]{\mathcal{CE}(#1)}
\newcommand{\CEP}[4]{\mathcal{CE}(#1,#2; #3,#4)}
\newcommand{\CEIP}[3]{\mathcal{CE}(#1;#2,#3)}

\newcommand{\piecewiseConstant}[2]{\overline{#1}_{\kern-1pt#2}}
\newcommand{\pwC}{\piecewiseConstant}
\newcommand{\underpiecewiseConstant}[2]{\underline{#1}_{\kern-1pt#2}}
\newcommand{\upwC}{\underpiecewiseConstant}
\newcommand{\sft}{\mathsf{t}}
\newcommand{\ttau}{t_{\tau}}
\newcommand{\Utau}{\rho_\tau}
\newcommand{\pwM}[2]{\widetilde{#1}_{\kern-1pt#2}}
\def\gen(#1,#2){\calS_{#1}(#2)}
\newcommand\genn[1]{\calS_{#1}}
\newcommand{\nuovo}{\mathscr{S}}
\let\ol\overline

\newcommand{\tetapi}{\boldsymbol \teta}
\newcommand{\pinfty}{{+\infty}}
\newcommand{\mres}{\kern1pt\mathbin{\vrule height 1.6ex depth 0pt width
    0.13ex\vrule height 0.13ex depth 0pt width 1.3ex}}
\newcommand{\topref}[2]{\stackrel{\eqref{#1}}#2}
\newcommand{\Dom}{\calM^+}
\newcommand{\half}{\relax} 
\newcommand{\thalf}{\relax} 
\newcommand{\frA}{\mathfrak A}
\newcommand{\frB}{\mathfrak B}
\newcommand{\frc}{\mathfrak c}
\newcommand{\frl}{\mathfrak l}
\newcommand{\frs}{\mathfrak s}
\newcommand{\frf}{\mathfrak f}
\newcommand{\frm}{\mathfrak m}
\newcommand{\kernel}[2]{\boldsymbol {#1}_{#2}}
\def\calS{\mathscr E}
\def\calF{\mathscr F}
\def\calR{\mathscr R}
\def\Aalpha{\upalpha}
\newcommand{\Lebone}{\lambda}
\newcommand{\bnu}{\boldsymbol\upnu}
\newcommand{\bj}{{\boldsymbol j}}
\newcommand{\bsigma}{{\boldsymbol \varsigma}}

\newcommand{\rfield}{\mathrm F_0}

\newcommand{\OrmD}{\rmD}
\newcommand{\Gop}{\boldsymbol G}
\newcommand{\Fop}{\boldsymbol F}
\newcommand{\restr}[1]{\lower3pt\hbox{$|_{#1}$}}
\newcommand{\nchi}{{\raise.3ex\hbox{$\chi$}}}
\begin{document}

\title{Jump processes as Generalized Gradient Flows}

\author{Mark A.\ Peletier}
\address{M.\ A.\ Peletier, Department of Mathematics and Computer Science and Institute for Complex Molecular Systems, TU Eindhoven, 5600 MB Eindhoven, The Netherlands} 
\email{M.A.Peletier\,@\,tue.nl}

\author{Riccarda Rossi}
\address{R.\ Rossi, DIMI, Universit\`a degli studi di Brescia. Via Branze 38, I--25133 Brescia -- Italy}
\email{riccarda.rossi\,@\,unibs.it}

\author{Giuseppe Savar\'e}
\address{G.\ Savar\'e, Dipartimento di Matematica ``F.\ Casorati'', Universit\`a degli studi di Pavia. Via Ferrata 27, I--27100 Pavia -- Italy}
\email{giuseppe.savare\,@\,unipv.it}

\author{Oliver Tse}
\address{O.\ Tse, Department of Mathematics and Computer Science, Eindhoven University of Technology, 5600 MB Eindhoven, The Netherlands}
\email{o.t.c.tse\,@\,tue.nl}

\begin{abstract}
We have created a functional framework for a class of non-metric gradient systems. The state space is a space of nonnegative measures, and the class of systems includes
 the Forward Kolmogorov equations for the laws of Markov jump processes on Polish spaces. This framework comprises
  a definition of a notion of solutions, a method to prove existence, and an archetype uniqueness result.
We do this by using only the structure that is provided directly by the dissipation functional, which need not be homogeneous, and we do not appeal to any metric structure.
\end{abstract}

\maketitle

\tableofcontents

\section{Introduction}

The study of dissipative variational evolution equations has seen a tremendous activity in the last two decades. A general class of such systems is that of \emph{generalized gradient flows}, which formally can be written as 
\begin{equation}
\label{eq:GGF-intro-intro}
\dot \rho =   \rmD_\upzeta \EEE \sfR^*(\rho,-\rmD_\rho \sfE(\rho))
\end{equation}
in terms of a \emph{driving functional} $\sfE$ and a \emph{dual dissipation potential}
  $\sfR^* = \sfR^*(\rho,\upzeta)$, where $\rmD_\upzeta$ 
and $\rmD_\rho$ denote derivatives with respect to 
 $\upzeta$
and $\rho$. The most well-studied of these are  classical gradient flows~\cite{AmbrosioGigliSavare08}, for which 
 $\upzeta \mapsto \rmD_\upzeta \sfR^*(\rho,\upzeta) = \mathbb K(\rho) \upzeta$  \EEE  is a linear operator $\mathbb K(\rho)$, and rate-independent systems~\cite{MielkeRoubicek15}, 
for which 
 $\upzeta\mapsto \rmD_\upzeta\sfR^*(\rho,\upzeta)$ \EEE
 is zero-homogeneous. 

However, various models naturally lead to gradient structures that are neither classic nor rate-independent. For these systems, the map 
 $\upzeta\mapsto \rmD_\upzeta\sfR^*(\rho,\upzeta)$  \EEE
is neither linear nor zero-homogeneous, and in many cases it is not even homogeneous of any order. Some examples are 
\begin{enumerate}
\item Models of chemical reactions, where $\sfR^*$ depends exponentially on 
$\upzeta$~\cite{Feinberg72,Grmela10,ArnrichMielkePeletierSavareVeneroni12,LieroMielkePeletierRenger17},
\item The Boltzmann equation, also with exponential $\sfR^*$~\cite{Grmela10},
\item Nonlinear viscosity relations such as the Darcy-Forchheimer equation for porous media flow~\cite{KnuppLage95,GiraultWheeler08},
\item Effective, upscaled descriptions in materials science, where the effective potential~$\sfR^*$ arises through a cell problem, and can have many different types of dependence on~$\upzeta$ \cite{ElHajjIbrahimMonneau09,PerthameSouganidis09,PerthameSouganidis09a,MirrahimiSouganidis13,LieroMielkePeletierRenger17,DondlFrenzelMielke18TR,PeletierSchlottke19TR,MielkeMontefuscoPeletier20TR},
\item  Gradient structures that arise from  large-deviation principles for sequences of stochastic processes, in particular jump processes~\cite{MielkePeletierRenger14,MielkePattersonPeletierRenger17}.
\end{enumerate}
The last example is the inspiration for this paper.

Regardless whether $\sfR^*$ is classic, rate-independent, or otherwise,  equation~\eqref{eq:GGF-intro-intro} typically is only formal, and it is a major mathematical challenge to construct an appropriate functional framework for this equation. Such a functional framework should give the equation a rigorous meaning, and provide the means to prove well-posedness, stability, regularity and approximation results to facilitate the study of the equation. 

For classical gradient systems, in which $\rmD_\upzeta\sfR^*$ is linear and $\sfR^*$ is quadratic in $\upzeta$ (therefore also  called `quadratic' gradient systems)  and when $\sfR^*$ generates a metric space, a rich framework has been created by Ambrosio, Gigli, and Savar\'e~\cite{AmbrosioGigliSavare08}. For rate-independent systems, in which $\sfR^*$ is $1$-homogeneous in $\upzeta$, the complementary concepts of `Global Energetic solutions' and `Balanced Viscosity solutions' give rise to two different frameworks~\cite{MielkeTheilLevitas02,Dal-MasoDeSimoneMora06,MielkeRossiSavare12a,MRS13,MielkeRoubicek15}. 

For the examples (1--5) listed above, however, $\sfR^*$ is  not  homogeneous in~$\upzeta$, and neither the rate-independent frameworks nor the metric-space theory  apply. Nonetheless, the existence of such models of real-world systems with a formal variational-evolutionary structure suggests that there may exist a functional framework for such equations that relies on this structure.
 In this paper we build exactly such a framework for an important class of equations of this type, those that describe Markov jump processes. We expect the approach
 advanced here to be applicable  to a broader range of systems.

\subsection{Generalized gradient systems for Markov jump processes}
%

Some generalized gradient-flow structures of evolution equations are generated by the large deviations of an underlying, more microscopic stochastic process~\cite{AdamsDirrPeletierZimmer11,AdamsDirrPeletierZimmer13,DuongPeletierZimmer13,MielkePeletierRenger14,MielkePeletierRenger16,LieroMielkePeletierRenger17}. This explains the origin and interpretation of such structures, and it can be used to identify hitherto unknown gradient-flow structures~\cite{PeletierRedigVafayi14,GavishNyquistPeletier19TR}. 

It is the example of Markov \emph{jump} processes that inspires the
results of this paper, and we describe this example here;
 nonetheless, \EEE
 the general
setup that starts in Section~\ref{ss:assumptions} has wider
application. We think of Markov  jump processes as jumping from one
`vertex' to another `vertex' along an `edge' of a `graph'; we place
these terms between quotes because the space $V$ of vertices may be
finite, countable, or even uncountable, and similarly the space
$\edg:= V\times V$ 
of edges may be finite, countable, or uncountable (see
Assumption~\ref{ass:V-and-kappa} below). In this paper, $V$ is a
 standard Borel space.

The laws of such processes are time-dependent measures
$t\mapsto \rho_t\in \Dom(V)$
(with $\Dom(V)$ the space of
 positive finite Borel \EEE
\nc
measures---see Section~\ref{ss:3.1}). These laws satisfy 
the Forward Kolmogorov equation
\begin{align}\label{eq:fokker-planck}
    \partial_t\rho_t = Q^*\rho_t, \qquad 
      (Q^*\rho)(\dd x) = \int_{y\in V} \rho(\dd y) \kappa(y,\dd x)  - \rho(\dd x)\int_{y\in V} \kappa(x,\dd y).
\end{align}
Here  $Q^*:\calM(V)\to \calM(V)$ 
is the dual of the infinitesimal generator
 $Q:\Bb(V)\to \Bb(V)$ 
of the
process, which
for an arbitrary bounded Borel function $\varphi\in \Bb(V)$
\nc is given by
\begin{equation}
\label{eq:def:generator}
 (Q\varphi)(x) = \int_V [\varphi(y)-\varphi(x)]\,\kappa(x,\dd y).
\end{equation}
The jump kernel $\kappa$  in these definitions characterizes the process: $\kappa(x,\cdot)\in \calM^+(V)$ is the infinitesimal rate of jumps of a particle from the point $x$ to  points in $V$.  Here we   address \EEE the reversible case, which means that the process has an invariant measure $\pi\in \calM^+(V)$, i.e., $Q^*\pi=0$, and that the joint measure $\pi(\dd x) \kappa(x,\dd y)$ is symmetric in $x$ and~$y$.

\medskip

In this paper we consider evolution equations of the form~\eqref{eq:fokker-planck} for the nonnegative measure $\rho$, as well as various linear and nonlinear generalizations. We will view them as gradient systems of the form~\eqref{eq:GGF-intro-intro}, and use this gradient structure to study their properties. 

The gradient structure for equation~\eqref{eq:fokker-planck}  consists of the state
space $\calM^+(V)$, a driving functional
$\calS:\calM^+(V)\to[0,\pinfty]$, and a dual dissipation potential
$\calR^*:\calM^+(V)\times \Bb(\edg)\to[0,\pinfty]$ (where $\Bb(\edg)$ denotes the space of bounded Borel functions  on $\edg$). We now describe this structure in formal terms, and   making it rigorous is one of the aims of this paper. 

The functional that drives the evolution is the relative entropy with
respect to the invariant measure $\pi$,
namely 
\begin{equation}
\label{eq:def:S}
\calS(\rho) =
\half\calF_{\upphi}(\rho|\pi):= 
\begin{cases}
\displaystyle
 \half \int_{V}   \upphi\bigl(u(x)\bigr) \pi(\dd x)  & \displaystyle \text{ if } \rho \ll \pi, \text{ with } u =\frac{\dd \rho }{\dd \pi},
\\
\pinfty & \text{ otherwise},
\end{cases}
\end{equation}
where for the example of Markov jump processes the `energy density' $\upphi$ is given by
\begin{equation}
\label{def:phi-log-intro}
\upphi(s) := \half s\log s  - s + 1.
\end{equation}
(In the general development below we  consider more general functions $\upphi$, such as those that arise in strongly interacting particle systems; see e.g.~\cite{KipnisOllaVaradhan89,DirrStamatakisZimmer16}). 

The dissipation potential $\sfR^*$ is best written in terms of an alternative potential $\calR^*$, 
\[
\sfR^*(\rho,\upzeta) := \calR^*(\rho,\ona \upzeta).
\]
Here the `graph gradient' $\ona:\Bb(V) \to \Bb(\edg)$ and its
 negative dual, the 
`graph divergence operator' $ \odiv:\calM(\edg)\to\calM(V)$,  \EEE are
defined as follows:
\begin{subequations}
\label{eq:def:ona-div}
\begin{align}
\label{eq:def:ona-grad}
(\ona \varphi)(x,y) &:= \varphi(y)-\varphi(x) &&\text{for any }\varphi\in \Bb(V),\\
  \nc (\odiv \bj )(\dd x) &:= \int_{y\in V}
                         \bigl[\bj (\dd x,\dd y)-\bj (\dd y,\dd
                         x)\bigr]\nc
                                              &&\text{for any }\bj \in \calM(\edg), \EEE
       \label{eq:def:div}
\end{align}
\end{subequations}
\GGGO and are linked by
\begin{equation}
  \label{eq:nabladiv}
  \iint_\edg \ona\varphi(x,y)\,\bj (\dd x,\dd y)=
  -\int_V \varphi(x) \,\odiv \bj (\dd x)\quad
  \text{for every }\varphi\in \Bb(V).
\end{equation}
\nc
The dissipation functional $\calR^*$ is defined for
$\xi\in \Bb(\edg)$
by
\begin{align}
\label{eq:def:R*-intro}
  &\calR^*(\rho,\xi) :=
     \frac 12 
    \int_{\edg} \Psi^*(\xi(x,y)) \, \bnu_\rho(\dd x \,\dd y),
\end{align}
where the function $\Psi^*$ and the  `edge'   measure $\bnu_\rho$ will be fixed in \eqref{eq:def:alpha} below.

With these definitions, the gradient-flow equation~\eqref{eq:GGF-intro-intro} can be written alternatively as 
\begin{equation}
\label{eq:GF-intro}
\partial_t \rho_t = -
\odiv \Bigl[ \rmD_\xi\calR^*\Bigl(\rho_t,-\thalf
\ona\upphi'\Bigl(\frac{\dd \rho_t}{\dd\pi}\Bigr)\Bigr)\Bigr],
\end{equation}
which can be recognized by observing that 
\[
\bigl\langle \rmD_\upzeta\sfR^*(\rho,\upzeta),\tilde \upzeta\bigr \rangle 
= \frac{\dd }{\dd h} \calR^*(\rho,\ona \upzeta+h\ona \tilde \upzeta)\Big|_{h=0}
=\bigl\langle \rmD_\xi\calR^*(\rho,\ona \upzeta),\ona \tilde \upzeta\bigr \rangle 
=\bigl\langle -\odiv \rmD_\xi\calR^*(\rho,\ona \upzeta), \tilde \upzeta\bigr \rangle,
\]
and  $\rmD \calS(\rho) = \half
\upphi'(u)$  (which corresponds to $\half \log u$
for the logarithmic entropy \eqref{def:phi-log-intro}). 
This $(\odiv,\ona)$-duality structure is a common feature in both physical and probabilistic models, and
 has its origin in the distinction between `states' and `processes'; see~\cite[Sec.~3.3]{PeletierVarMod14TR} and~\cite{Ottinger19} for discussions.

\bigskip

For this example of Markov jump processes we consider a class of
generalized gradient structures   
of the type above, given by $\calS$ and $\calR^*$ (or equivalently by
 the densities 
 $\upphi$, $\Psi^*$, and  the measure  $\bnu_\rho$), with the property  that  equations~\eqref{eq:GGF-intro-intro} and~\eqref{eq:GF-intro}  coincide with~\eqref{eq:fokker-planck}. Even for fixed $\calS$ there exists a range of choices for $\Psi^*$ and $\bnu_\rho$ that achieve this (see also the discussion in~\cite{GlitzkyMielke13,MielkePeletierRenger14}).
 A simple calculation 
  (see the discussion at the end of Section \ref{ss:assumptions}) 
 shows that, if one chooses for the measure $\bnu_\rho$ the form
 \begin{equation}
 \label{eq:def:alpha}
 \bnu_\rho(\dd x\,\dd y) = 
 \upalpha(u(x),u(y))\, \pi(\dd x)\kappa(x,\dd y),
 \end{equation}
 \GGG
 for a suitable fuction $\upalpha:[0,\infty)\times [0,\infty)\to [0,\infty)$,
 \EEE
 and one introduces the map $\rmF:(0,\infty)\times(0,\infty)\to\R$
 \begin{equation}
   \label{eq:184}
   \rmF(u,v):=
   (\Psi^*)'\big[\upphi'(v)-\upphi'(u)\big]\upalpha(u,v)\quad
   u,v>0,
 \end{equation}
 then \eqref{eq:GF-intro} takes the form of the integro-differential
 equation
 \begin{equation}
\partial_t u_t(x) = \int_{y\in V} \fw\bigl(u_t(x),u_t(y)\bigr)\,
\kappa(x,\dd y),\label{eq:180}
\end{equation}
in terms of the density $u_t$ of $\rho_t$ with respect to $\pi$.
Therefore, \nc
a pair $(\Psi^*,\bnu_\rho)$ leads to equation~\eqref{eq:fokker-planck}
 whenever
 \GGGO $(\Psi^*,\upphi,\upalpha)$ satisfy the \emph{compatibility property}
 \begin{equation}
   \label{cond:heat-eq-2}
   \rmF(u,v)=v-u
   \quad
   \text{for every }u,v>0.
 \end{equation}
 The classical quadratic-energy, quadratic-dissipation choice 
 \begin{equation}
   \label{eq:68}
   \Psi^*(\xi)=\tfrac 12\xi^2,\quad
   \upphi(s)=\tfrac 12s^2,\quad
   \upalpha(u,v)=1
 \end{equation}
 corresponds to the Dirichlet-form approach to
 \eqref{eq:fokker-planck}
 in $L^2(V,\pi)$. Here $\calR^*(\rho,\bj)=\calR^*(\bj)$
 is in fact independent of $\rho$: if one introduces
 the symmetric bilinear form
 \begin{equation}
   \label{eq:185}
   \llbracket u,v\rrbracket:=\frac 12\iint_\edg \ona u(x,y)\,\ona
   v(x,y)\,\tetapi(\dd x,\dd y),\quad
   \llbracket u,u\rrbracket=\frac 12\iint_\edg \Psi(\ona u)\,\dd \tetapi,
 \end{equation}
  with $\tetapi (\dd x, \dd y) = \pi(\dd x ) \kappa(x, \dd y)$ (cf.\ \eqref{nu-pi} ahead), \EEE
 then \eqref{eq:180} can also be formulated as
 \begin{equation}
   \label{eq:186}
   (\dot u_t, v)_{L^2(V,\pi)}+
   \llbracket u_t,v\rrbracket=0\quad
   \text{for every }v\in L^2(V,\pi).
 \end{equation}
 \nc
 Two other choices have received attention in the
  recent literature. Both of these are based not on the quadratic energy $\upphi(s)=\tfrac 12s^2$, but on the Boltzmann entropy functional $\upphi(s) = s\log s - s + 1$:
\begin{subequations}
\label{choices}
\begin{enumerate}
\item The large-deviation characterization~\cite{MielkePeletierRenger14} leads to the choice
\begin{equation}
\label{choice:cosh}
\Psi^*(\xi) :=  4\bigl(\cosh (\xi/2) - 1\bigr)
\quad \text{and}\quad \upalpha(u,v) := \sqrt{uv}.
\end{equation}
The corresponding primal dissipation potential $\Psi := (\Psi^*)^*$ is given by
\[
\Psi(s) := 2s\log \left(\frac{s+\sqrt{s^2+4}}2 \right)  - \sqrt{s^2 + 4} + 4.
\]

\item The `quadratic-dissipation' choice introduced independently by Maas~\cite{Maas11},   Mielke \cite{Mielke13CALCVAR}, and Chow, Huang, and Zhou~\cite{ChowHuangLiZhou12} for Markov processes on \emph{finite} graphs,  
\begin{equation}
\label{choice:quadratic}
\Psi^*(\xi) := \tfrac12 \xi^2, \quad \Psi(s) = \tfrac12 s^2 , 
\quad \text{and}\quad \upalpha(u,v) := \frac{
  u-v
}{
  \log(u) -
  \log(v) }.
\end{equation}
\end{enumerate}
\end{subequations}
Other examples are discussed in \S
\ref{subsec:examples-intro}.
 With  \EEE the quadratic choice~\eqref{choice:quadratic}, the gradient
system fits into the metric-space structure (see e.g.~\cite{AmbrosioGigliSavare08})
and this feature has been used extensively to investigate the properties of general Markov  jump processes~\cite{Maas11,Mielke13CALCVAR,ErbarMaas12,Erbar14,Erbar16TR,ErbarFathiLaschosSchlichting16TR}.  In this paper, however, we focus on functions $\Psi^*$ that are not homogeneous,  as in~\eqref{choice:cosh}, and 
such that  the   corresponding structure is not covered by the
 usual metric
framework. On the other hand, there are various arguments why this structure nonetheless has a certain `naturalness' (see Section~\ref{ss:comments}), and these motivate  our aim to develop a functional framework based on this structure.

\subsection{Challenges}

Constructing a `functional framework' for the gradient-flow equation~\eqref{eq:GF-intro} with the choices~\eqref{def:phi-log-intro} and~\eqref{choice:cosh} presents a number of independent challenges.

\subsubsection{Definition of a solution} 
\label{ss:def-sol-intro}
As it stands, the formulation of equation~\eqref{eq:GF-intro}
and of the functional $\mathcal R^*$ of \eqref{eq:def:R*-intro} presents
many difficulties:
the definition of $\mathcal R^*$ and the measure $\bnu_\rho$ when $\rho$ is not absolutely
continuous
with respect to~$\pi$, the concept of  time differentiability for  the curve of  measures $\rho_t$, \EEE
whether  $\rho_t$ is necessarily absolutely continuous with respect to $\pi$ along an evolution,
what happens if $\dd \rho_t /\dd \pi$ vanishes
and $\upphi$ is not differentiable at $0$ as in the case of the
logarithmic entropy, 
etcetera. As a result of these difficulties, it is not clear what constitutes a solution of equation~\eqref{eq:GF-intro}, let alone whether such solutions exist. In addition, a good solution concept should be robust under taking limits, and the formulation~\eqref{eq:GF-intro} 
does not seem to satisfy this requirement either. 

For quadratic and rate-independent systems, successful functional frameworks have been constructed on the basis of the  Energy-Dissipation balance~\cite{Sandier-Serfaty04,Serfaty11,MRS2013,LieroMielkePeletierRenger17,MielkePattersonPeletierRenger17}, and we follow that example here. In fact, the same large-deviation principle that gives rise to the `cosh' structure above formally yields the `EDP' functional
\begin{equation}
\label{eq:def:mathscr-L}
\mathscr L(\rho,\bj ) := 
\begin{cases}
\displaystyle
\int_0^T \Bigl[ \calR(\rho_t, \bj _t) + \calR^*\Bigl(\rho_t, -\thalf\ona \upphi'\Bigl(\frac{\dd \rho_t}{\dd\pi}\Bigr) \Bigr) \Bigr]\dd t + \calS(\rho_T) - \calS(\rho_0)\hskip-8cm&\\
&\text{if }\partial_t \rho_t + \odiv \bj _t = 0 \text{ and } \rho_t \ll
\pi \text{  for all
  $t\in [0,T],$ \nc}\\
\pinfty &\text{otherwise.}
\end{cases}
\end{equation}
In this formulation, $\calR$ is the Legendre dual of $\calR^*$
 with respect to the $\xi$ variable,
\nc which can be written in terms of the Legendre dual $\Psi:=\Psi^{**}$ of $\Psi^*$ as
\begin{equation}
  \label{eq:def:R-intro}
\calR(\rho,\bj ) := \frac 12\nc\int_{\edg} \Psi\left( 2\frac{\dd \bj}{\dd \bnu_\rho}\right)\dd\bnu_\rho. \qquad
\end{equation}
 Along smooth curves $\rho_t=u_t\pi$ with strictly positive
densities,
the functional $\mathscr L$ is nonnegative, since 
\nc
\begin{align}
\notag
  \nc\frac \dd{\dd t} \calS(\rho_t)
  &= 
   \int_V \upphi'(u_t)\partial_t u_t\, \dd\pi \nc
    =\int_V \upphi'(u_t(x)) \partial_t\rho_t(\dd x) 
    = - \half
    \int_V \upphi'(u_t(x)) (\odiv \bj_t)(\dd x)\\
  & = \half 
    \iint_\edg \ona \upphi'(u_t) (x,y) \,\bj_t(\dd x\,\dd y)
    = \half
    \iint_\edg \ona \upphi'(u_t) (x,y) \frac{\dd \bj_t}{\dd
    \bnu_{\rho_t}}(x,y)\;\bnu_{\rho_t} (\dd x\,\dd y)
 \label{eq:174} \\
  &\geq -
     \frac 12\nc
    \iint_\edg \left[ \Psi\left(  2\,
    \frac{\dd \bj_t}{\dd \bnu_{\rho_t}}(x,y)\right)
    + \Psi^*\left(-\half 
    \ona \upphi'(u_t) (x,y) \right) \right]  \bnu_{\rho_t}(\dd x\,\dd y).
\label{ineq:deriv-GF}
\end{align}
After time integration we find  that $\mathscr L(\rho,\bj )$ is nonnegative for any pair $(\rho,\bj )$. 

The minimum of $\mathscr L$ is formally achieved at value zero, at pairs $(\rho,\bj )$ satisfying
\begin{align}\label{eq:flux-identity}
  2\bj_t = (\Psi^*)'\left(-
  \thalf
  \ona\upphi'\Bigl(\frac{\dd \rho_t}{\dd\pi}\Bigr)\right)\bnu_{\rho_t}
\qquad \text{and} \qquad \partial_t \rho_t + \odiv \bj_t = 0,
\end{align}
which is an equivalent way of writing the gradient-flow
equation~\eqref{eq:GF-intro}. This can be recognized, as usual for
gradient systems, by observing that achieving equality in the
inequality~\eqref{ineq:deriv-GF} requires equality in the Legendre
duality of $\Psi$ and $\Psi^*$, which reduces to the equations
above. \GGG

\begin{remark}
  \label{rem:alpha-concave}
  It is worth noticing that the joint convexity of the functional
  $\calR$ of \eqref{eq:def:R-intro} (a crucial property for the
  development of our analysis) is equivalent to the \emph{convexity} of
  $\Psi$ and \emph{concavity}
  of the function $\upalpha$.
\end{remark}
\EEE

\begin{remark}
\label{rem:choice-of-2}
  Let us add a comment concerning the choice of
  the factor $1/2$ in front of $\Psi^*$ in \eqref{eq:def:R*-intro}, and the corresponding factors $1/2$ and $2$ in
  \eqref{eq:def:R-intro}. 
  The cosh-entropy combination~\eqref{choice:cosh} satisfies the linear-equation condition $\rmF(u,v) = v-u$ (equation~\eqref{cond:heat-eq-2}) because of the elementary identity
  \[
  2\,\sqrt{uv} \,\sinh \Bigl(\frac12 \log \frac vu\Bigr) = v-u.
  \]
The factor $1/2$ inside the $\sinh$ can be included in different
ways. In~\cite{MielkePeletierRenger14} it was included explicitly, by
writing expressions of the form $\rmD \sfR^*(\rho,-\tfrac12
\rmD\sfE(\rho))$; in this paper we
follow~\cite{LieroMielkePeletierRenger17} and  include this factor in
the definition of  
$\calR^*$. 
\end{remark}

\begin{remark}
  The continuity equation
  $\partial_t \rho_t + \odiv \bj _t = 0 $ is invariant with respect to skew-symmetrization of $\bj$, i.e.\ with respect to the
  transformation  $\bj\mapsto \bj^\flat$ with $\bj^\flat(\dd x,\dd
  y):=
  \frac12 \bigl(\bj(\dd x,\dd y)-\bj(\dd y,\dd x)\bigr)$. \EEE
  Therefore we could also write the second integral in
  \eqref{eq:174}
  as
  \begin{align*}
    &
    \iint_\edg \ona \upphi'(u_t) (x,y) \frac{\dd \bj^\flat_t}{\dd
      \bnu_{\rho_t}}(x,y)\;\bnu_{\rho_t} (\dd x\,\dd y)
                          \\
    &\qquad\geq -
    \frac 12
    \iint_\edg \left[ \Psi\left( 
    \frac{\dd  (2 \bj^\flat_t)\EEE}{\dd \bnu_{\rho_t}}(x,y)\right)
    + \Psi^*\left(-\half 
    \ona \upphi'(u_t) (x,y) \right) \right]  \bnu_{\rho_t}(\dd x\,\dd y).
  \end{align*}
  thus replacing
  $\Psi\left(  2\nc
    \frac{\dd \bj_t}{\dd \bnu_{\rho_t}}(x,y)\right)$ with
  the lower term
  $\Psi\left( 
    \frac{\dd  (2 \bj^\flat_t)\EEE}{\dd \bnu_{\rho_t}}(x,y)\right)$,  cf.\ Remark \ref{rem:skew-symmetric}, \EEE
  and obtaining a corresponding  equation as \eqref{eq:flux-identity}
  for  $(2\bj_t^\flat)$ instead of $2\bj_t$.  \EEE
  This would lead to a weaker gradient system, since the choice \eqref{eq:def:R-intro} forces $\bj_t$ to
  be skew-symmetric, whereas the choice of a dissipation involving
  only $\bj^\flat$ would not control the symmetric part of $\bj$.  
  On the other hand, the evolution equation generated by the gradient system would remain the same.
\end{remark}

Since at least formally equation~\eqref{eq:GF-intro} is
equivalent to the requirement $\mathscr L(\rho,\bj )\leq0$,
we adopt this variational point of view to define
solutions to the  generalized gradient  system $(\calS,\calR,\calR^*)$.
This
inequality is in fact the basis for the variational Definition~\ref{def:R-Rstar-balance}
below. In order to do this in a rigorous manner, however, we will need
\begin{enumerate}
  \item A study of the continuity equation
\begin{equation}
\label{eq:ct-eq-intro}
\partial_t \rho_t + \odiv \bj_t = 0,
\end{equation}
that appears in the definition of the functional $\mathscr L$
(Section~\ref{sec:ct-eq}).
\item A rigorous definition of the
   measure $\bnu_{\rho_t}$ and of the \nc
  functional $\calR$ (Definition~\ref{def:R-rigorous});
\item A class $\CER 0T$ of curves ``of finite action'' in $\calM^+(V)$ 
   along which the functional $\calR$ has finite integral \nc
  (equation~\eqref{adm-curves});
\item  An appropriate \nc
  definition of the \emph{Fisher-information} functional (see
  Definition~\ref{def:Fisher-information})
\begin{equation}\label{eq:formal-Fisher-information}
  \rho \mapsto \Fish(\rho) := \calR^*\bigl(\rho,-\thalf
  \ona \upphi'(\dd \rho/\dd \pi)\bigr);
\end{equation}
\item A proof of the lower bound $\mathscr L\geq 0$ (Theorem~\ref{th:chain-rule-bound}) via a suitable chain-rule inequality.
\end{enumerate}

\subsubsection{Existence of solutions} 
The first test of a new solution concept is whether solutions exist under reasonable conditions. In this paper we provide two existence proofs that complement each other. 

The first existence proof is based on a reformulation of the
equation~\eqref{eq:fokker-planck} as a differential equation
in the  Banach space $L^1(V,\pi)$, driven by
a continuous dissipative operator. \nc
Under general compatibility conditions on $\upphi$,  $\Psi$, and
$\upalpha$,
we show that the solution
provided by this abstract approach    is also \EEE a solution in the
 variational
sense that we discussed 
above. The proof is presented in Section~\ref{s:ex-sg}
and is quite robust for initial data whose density takes value in a compact
interval $[a,b]\subset (0,\infty)$. In order to deal
with a more general class of data, we will
adopt two different viewpoints. A first possibility is to take
advantage of the robust stability properties of the
 $(\calS,\calR, \calR^*)$ Energy-Dissipation balance \EEE when
the Fisher information $\Fish$ is lower semicontinuous. A second
possibility is to exploit the monotonicity properties of \eqref{eq:180}
when the map $\rmF$ in~\eqref{eq:184} exhibits good behaviour
at the boundary of $\R_+^2$ and at infinity.

Since 
we  believe \EEE that the variational formulation
reveals a relevant structure of such systems and
we expect that it may also be useful in
dealing with more singular cases and their stability issues, \nc
we also present a more intrinsic approach  by adapting  the well-established  `JKO-Min\-i\-miz\-ing-Movement'   method to the structure of this equation. This method has been used,  e.g., 
 for metric-space gradient flows~\cite{JordanKinderlehrerOtto98,AmbrosioGigliSavare08}, for rate-independent systems~\cite{Mielke05a},  for some  non-metric systems with formal metric structure~\cite{AlmgrenTaylorWang93,LuckhausSturzenhecker95}, and also for Lagrangian systems with local transport~\cite{FigalliGangboYolcu11}. 

 This approach
 relies on
 the
 \nc
 {\em Dynamical-Variational Transport cost} (DVT)
 $\DVT \tau{\mu}{\nu}$, which is the $\tau$-dependent transport cost
 between two measures $\mu,\nu\in\calM^+(V)$  induced
   by
   the dissipation potential $\calR$ via 
\begin{equation}
\label{def:W-intro}
  \DVT\tau{\mu}{\nu} := \inf\left\{ \int_0^\tau \calR(\rho_t,\bj_t)\,\dd t \, : \, \partial_t \rho_t + \odiv \bj_t = 0, \ \rho_0 = \mu, \text{ and }\rho_
  \tau = \nu\right\}.
\end{equation}
In the
Minimizing-Movement  scheme a single  increment with time step $\tau>0$ is defined by the minimization problem
\begin{equation}
\label{MM-intro}
\rho^n \in \argmin_\rho \, \left( \DVT\tau{\rho^{n-1}}\rho + \calS(\rho)\right) .
\end{equation}
By concatenating such solutions, constructing appropriate
interpolations, and proving a compactness result---all  steps \EEE similar to the
procedure in~\cite[Part~I]{AmbrosioGigliSavare08}---we find a curve
 $(\rho_t,\bj_t)_{t\in [0,T]}$ satisfying
the continuity equation \eqref{eq:ct-eq-intro}
such that 
\begin{equation}
\label{ineq:soln-rel-gen-slope-intro}
 \int_0^t \bigl[\calR(\rho_r,\bj_r) + \nuovorel(\rho_r)\bigr]\, \dd r + \calS(\rho_t) \le \calS(\rho_0)\qquad\text{for all $t\in[0,T]$},
\end{equation}
where $\nuovorel:\rmD(\calS)\to[0,\pinfty)$ is a suitable
{\em relaxed slope} of the energy functional $\calS$ with respect to
 the cost  
 $\DVTn$ (see~\eqref{relaxed-nuovo}).  Under 
 a  lower-semicontinuity \EEE \nc
 condition on $\Fish$
 we show that $\nuovorel\ge \Fish $.
 It then follows that $\rho$ is a solution as defined above  (see Definition~\ref{def:R-Rstar-balance}). 

Section~\ref{s:MM} is devoted to developing 
 the `Minimizing-Movement'  approach for general DVTs.
This requires establishing
\begin{enumerate}[resume]
\item Properties of $\DVTn$ that generalize those of the `metric version' $\DVT\tau\mu\nu = \frac1{2\tau}d(\mu,\nu)^2$ (Section~\ref{ss:aprio});
\item A generalization of the `Moreau-Yosida approximation' and  of  the `De Giorgi variational interpolant' to the non-metric case, and a generalization of their properties (Sections~\ref{ss:MM} and~\ref{ss:aprio});
\item A compactness result as $\tau\to0$,    based on the properties of $\DVTn$ (Section~\ref{ss:compactness});
\item A proof of $\nuovorel\ge \Fish $ (Corollary~\ref{cor:cor-crucial}).
\end{enumerate}
 This procedure leads to  our existence result, Theorem
 \ref{thm:construction-MM}, of solutions  in the sense of  Definition
 \ref{def:R-Rstar-balance}.

\subsubsection{Uniqueness of solutions}
We prove uniqueness of variational solutions
under suitable convexity conditions
of $\Fish $ and $\calS$ (Theorem~\ref{thm:uniqueness}),
following an idea by Gigli~\cite{Gigli10}. \nc

\subsection{Examples}
\label{subsec:examples-intro}
We will use the following two guiding examples to illustrate the
results of this paper. Precise assumptions are given in
Section~\ref{ss:assumptions}.
In both  examples
the
state space consists of measures $\rho$  on a
standard Borel space $(V,\mathfrak B)$ endowed with a reference
Borel measure $\pi$.
The kernel $x\mapsto \kappa(x,\cdot)$ is a measurable family of nonnegative
measures  with uniformly bounded mass, 
such that the pair $(\pi,\kappa)$ satisfies detailed balance (see Section~\ref{ss:assumptions}).
\nc

\medskip
\emph{Example 1:   Linear equations \GGG driven by the Boltmzann
  entropy. \EEE}
This is the example that we have been using in this introduction. The equation is the  linear equation~\eqref{eq:fokker-planck},
\[
\partial_t\rho_t(\dd x) = \int_{y\in V} \rho(\dd y) \kappa(y,\dd x)  - \rho(\dd x)\int_{y\in V} \kappa(x,\dd y),
\]
which can also be written in terms of the density $u =\dd\rho/\dd \pi$ as
\[
\partial_t u_t(x) = \int_{y\in V} \bigl[u_t(y)-u_t(x)\bigr] \, \kappa(x,\dd y),
\]
\GGG and corresponds to the linear field $\rmF$ of
\eqref{cond:heat-eq-2}.
Apart from the classical quadratic setting of \eqref{eq:68}, \EEE
two gradient structures for this equation have recently received
attention in the literature, both driven by the Boltzmann entropy
\eqref{def:phi-log-intro}
$\upphi(s) = s\log s - s + 1$ as described in~\eqref{choices}:
\begin{enumerate}[label=\textit{(\arabic*)}]
\item The `cosh' structure: $\Psi^*(\xi) = 4\bigl(\cosh(\xi/2)
  \nc -1\bigr)$ and $\upalpha(u,v) = \sqrt{uv}$;
\item The `quadratic' structure: $\Psi^*(\xi) = \tfrac12 \xi^2$	 and $\upalpha (u,v) = (u-v)/\log(u/v)$.
\end{enumerate}
However, the approach of this paper applies to more general
combinations $(\upphi,\Psi^*,\upalpha)$ that lead to the same
equation. 
\GGG Due to the particular structure of \eqref{eq:184}, it is clear
that the $1$-homogeneity of the linear map $\rmF$ \eqref{cond:heat-eq-2} and
the $0$-homogeneity of the term $\upphi'(v)-\upphi'(u)$ associated  with \EEE
the Boltzmann entropy \eqref{def:phi-log-intro}
restrict the range of possible $\upalpha$ to \emph{$1$-homogenous
functions}
like the `mean functions'  $\upalpha(u,v) = \sqrt{uv}$ (geometric) and $\upalpha (u,v) = (u-v)/\log(u/v)$ (logarithmic). 

 Confining \EEE the analysis to concave functions (according to Remark
\ref{rem:alpha-concave}), \EEE
we observe that every concave and
$1$-homogeneous function $\upalpha$ can be obtained by
the concave generating function $\frf:(0,\pinfty)\to (0,\pinfty)$
\begin{equation}
  \label{eq:150}
  \upalpha(u,v)=u\frf(v/u)=v\frf(u/v),\quad
  \frf(r):=\alpha(r,1),\quad
  u,v,r>0.
\end{equation}
The symmetry of $\upalpha $ corresponds to the property
\begin{equation}
  \label{eq:151}
  r\frf(1/r)=\frf(r)\quad\text{for every }r>0,
\end{equation}
and shows that the function
\begin{equation}
  \label{eq:152}
  \mathfrak g(s):=\frac{\exp(s)-1}{\frf(\exp(s))}\quad s\in \R, \text{
    is odd}.
\end{equation}
The concaveness of $\frf$ also shows that $\mathfrak g$ is increasing,
so that we can define
  \begin{equation}
    \label{eq:149}
    \Psi^*(\xi):=\int_0^{\xi} \mathfrak g(s)\,\dd s
    =\int_1^{\exp(\xi)}\frac{r-1}{\frf (r)}\frac{\dd r}r,\quad \xi\in \R,
  \end{equation}
  which is convex, even, and superlinear if
  \begin{equation}
    \label{eq:153}
    \upalpha(0,1)=\frf(0)=
    \lim_{r\to0}r\frf\Bigl(\frac1r\Bigr)=0.
  \end{equation}
  A natural class of concave and $1$-homogeneous weight functions is provided by the \OLI {\em Stolarsky means} $\frc_{p,q}(u,v)$ with appropriate $p,q\in\R$, and any $u,v>0$ \cite[Chapter VI]{Bullen2003handbook}:
\[
	\upalpha(u,v) = \frc_{p,q}(u,v) := \begin{cases}
		\Bigl(\frac pq\frac{v^q-u^q}{v^p-u^p}\Bigr)^{1/(q-p)} &\text{if $p\ne q$, $q\ne 0$},\\
		\Bigl( \frac{1}{p}\frac{v^p-u^p}{\log(v) - \log(u)}\Bigr)^{1/p} &\text{if $p\ne 0$, $q= 0$}, \\
		e^{-1/p}\Bigl(\frac{v^{v^p}}{u^{v^p}}\Bigr)^{1/(v^p-u^p)} &\text{if $p= q\ne 0$}, \\
		\sqrt{uv} &\text{if $p= q= 0$},
	\end{cases}
\]
from which we identify other simpler means, such as the {\em power means} $\frm_p(u,v) = \frc_{p,2p}(u,v)$ with $p\in [-\infty, 1]$:
\begin{equation}
  \label{eq:147}
  \frm_p(u,v) =
  \begin{cases}
    \Big(\frac 12\big(u^p+v^p\big)\Big)^{1/p}&\text{if $0<p\le 1$ or
      $-\infty<p<0$ and $u,v\neq0$},\\
    \sqrt{uv}&\text{if }p=0,\\
    \min(u,v)&\text{if }p=-\infty,\\
    0&\text{if }p<0\text{ and }uv=0,
  \end{cases}
  \end{equation}
  and the generalized logarithmic mean $\frl_p(u,v)=\frc_{1,p+1}(u,v)$, $p\in[-\infty,-1]$. \EEE

  The power means are obtained from the concave generating functions
  \begin{equation}
    \label{eq:148}
    \frf_p(r):=2^{-1/p}(r^p+1)^{1/p} \quad \text{if }p\neq 0,\quad
    \frf_0(r)=\sqrt r,\quad
    \frf_{-\infty}(r)=\min(r,1),\quad
    r>0.
  \end{equation}
  We can thus define
  \begin{equation}
    \label{eq:149p}
    \Psi_p^*(\xi):=2^{1/p}\int_1^{\exp \xi}
    \frac{r-1}{(r^p+1)^{1/p}}\,\frac{\dd r}r,\quad \xi\in \R,
    \quad
    p\in (-\infty,1]\setminus 0,
  \end{equation}
  with the obvious changes when $p=0$ (the case
  $\Psi_0^*(\xi)=4(\cosh(\xi/2)-1$))
  or $p=-\infty$ (the case $\Psi_{-\infty}^*(\xi)=
  \exp(|\xi|)-|\xi|$).

  It is interesting to note that the case $p=-1$ (harmonic mean)
  corresponds to
  \begin{equation}
    \label{eq:154}
    \Psi_{-1}^*(\xi)=\cosh(\xi)-1.
  \end{equation}
  We finally note that the arithmetic mean
  $\upalpha(u,v)=\frm_1(u,v)=(u+v)/2$
  would yield $\Psi_1^*(\xi)=4\log(1/2(1+\rme^\xi))-2\xi$, which is not superlinear.


\medskip
\emph{Example 2: Nonlinear equations.}
  We consider a combination of \EEE
$\upphi$, $\Psi^*$, and $\upalpha$ such that the
function $\rmF$ introduced in \eqref{eq:184}
has a continuous extension up to the boundary of
$[0,\pinfty)^2$ and satisfies a
suitable growth and monotonicity condition (see Section~\ref{s:ex-sg}).
The resulting integro-differential equation
is given by \eqref{eq:180}.
Here is a list of some interesting cases (we will neglect all the
 issues \EEE  concerning growth and regularity).
\begin{enumerate}
\item A field of the form $\rmF(u,v)=f(v)-f(u)$ with $f:\R_+\to \R$ monotone
  corresponds to the equation
  \[
\partial_t u_t(x) = \int_{y\in V} \bigl(f(u_t(y))-f(u_t(x))\bigr)\, \kappa(x,\dd y),
\]
and can be classically considered in the framework of the Dirichlet forms, i.e.~$\upalpha   \equiv \EEE 1$,
  $\Psi^*(r)= r^2/2$, with energy $\upphi$ satisfying $\upphi' = f$.
\item
  The case $\rmF(u,v)=g(v-u)$, with $g:\R\to \R$ monotone and odd,
  yields the equation
    \[
      \partial_t u_t(x) = \int_{y\in V} g\bigl(u_t(y)-u_t(x)\bigr)\, \kappa(x,\dd y),
\]
and can be obtained with the choices
$\upalpha  \equiv \EEE 1$, $\upphi(s):=s^2/2$ and $\Psi^*(r):=\int_0^r g(s)\,\dd s$.
\item Consider now the case when $\rmF$ is
  positively $q$-homogeneous, with $q\in [0,1]$.
  It is then natural to consider a $q$-homogeneous $\upalpha$
  and the logarithmic entropy $\upphi(r)=r\log r-r+1$.
  If the function $h:(0,\infty)\to \R$, 
  $h(r):=\rmF(r,1)/\upalpha(r,1)$ is increasing, then setting as in \eqref{eq:149p}
  \begin{displaymath}
    \Psi^*(\xi):=\int_1^{\exp (\xi)}h(r)\,\dd r
  \end{displaymath}
   equation \eqref{eq:180} provides an example of
  generalized gradient system $(\calS,\calR,\calR^*)$.
  Simple examples are $\rmF(u,v)=v^q-u^q$,
  corresponding to the equation
  \[
\partial_t u_t(x) = \int_{y\in V} \bigl(u^q_t(y)-u^q_t(x)\bigr)\, \kappa(x,\dd y),
\]
  with $\upalpha(u,v):=
  \frm_p(u^q,v^q)$ and $\Psi^*(\xi):=\frac 1q\Psi_p^*(q\xi)$,
  where $\Psi^*_p$ has been defined in~\eqref{eq:149p}.
  In the case $p=0$ we get $\Psi^*(\xi)=\frac
  4q\big(\cosh(q\xi/2)-1\big)$.

  As a last example, we can consider
  $\rmF(u,v)=\operatorname{sign} (v-u)|v^m-u^m|^{1/m}$, $m>0$, and $\upalpha(u,v)=\min(u, v)$;
  in this case,  the function $h$ given by \EEE  $h(r)=(r^m-1)^{1/m}$ when $r\ge1$, and
  $h(r)=-(r^{-m}-1)^{1/m}$ when  $r<1$,  satisfies the required
  monotonicity property.
\end{enumerate}

\subsection{Comments}
\label{ss:comments}

\emph{Rationale for studying this structure.}
\GGG We think that the structure of generalized gradient systems
$(\calE,\calR,\calR^*)$ is sufficiently rich and interesting to
deserve a careful analysis.
It provides a genuine extension of the more familiar quadratic
gradient-flow structure of Maas, Mielke, and Chow--Huang--Zhou, which
better fits into the metric framework of \cite{AmbrosioGigliSavare08}.
In Section~\ref{s:ex-sg} we will also show
its connection with the theory of dissipative evolution equations.

Moreover, \EEE
the specific non-homogeneous structure based on the $\cosh$ function~\eqref{choice:cosh} has a number of arguments in its favor, which can be summarized in the statement that it is `natural' in various different ways:
\begin{enumerate}
\item It appears in the characterization of  large deviations of Markov processes; see Section~\ref{ss:ldp-derivation} or~\cite{MielkePeletierRenger14,BonaschiPeletier16};
\item It arises in evolutionary limits of other gradient structures (including quadratic ones) \cite{ArnrichMielkePeletierSavareVeneroni12,Mielke16,LieroMielkePeletierRenger17,MielkeStephan19TR};
\item It `responds naturally' to external forcing \cite[Prop.~4.1]{MielkeStephan19TR};
\item It can be generalized to nonlinear equations \cite{Grmela84,Grmela10}.
\end{enumerate}
We will explore these claims in more detail in a forthcoming paper.
Last but not least, the very fact that non-quadratic, generalized gradient flows may arise in the limit of gradient flows suggests that,
allowing for a broad class of dissipation mechanisms is crucial in order to (1) fully exploit the flexibility of the   gradient-structure \EEE formulation, and  (2) explore its robustness with respect to $\Gamma$-converging
energies and dissipation potentials. \EEE
\medskip

\emph{Potential for generalization.}
In this paper we have chosen to concentrate on the consequences of non-homogeneity of the dissipation potential $\Psi$ for the techniques that are commonly used in gradient-flow theory. Until now, the lack of a sufficiently general rigorous construction of the functional $\calR$ and its minimal integral over curves $\DVTn$ have impeded the use of this variational structure in rigorous proofs, and a main aim of this paper is to provide a way forward by constructing a rigorous framework for these objects, while keeping the setup (in particular, the ambient space $V$) as general as possible. \EEE

In order to restrict the length of this paper, we considered only
simple driving functionals $\calS$, which are of the
local variety
$\calS(\rho) = \thalf
\int \upphi(\dd\rho/\dd\pi)\dd\pi$.
\nc
Many gradient systems appearing in the literature are driven by more general functionals, that include interaction and other nonlinearities~\cite{ErbarFathiLaschosSchlichting16TR,ErbarFathiSchlichting19TR,RengerZimmer19TR,HudsonVanMeursPeletier20TR},
 and we expect that the techniques of this paper will be of use in the study of such systems.

As one specific direction of generalization, we note that the Minimizing-Movement construction  on which the proof of Theorem \ref{thm:construction-MM} is based has a scope wider than  that of the  generalized gradient structure $(\calS, \calR, \calR^*)$ \EEE
under consideration. In fact, as we  show in Section~\ref{s:MM}, Theorem~\ref{thm:construction-MM} yields the existence of (suitably formulated) gradient flows  in 
a general \emph{topological space} endowed with a cost fulfilling suitable properties. While we do not  develop this discussion in this paper, at places  throughout the paper \EEE we hint at this prospective generalization: the `abstract-level' properties of the DVT cost are addressed in Section~\ref{ss:4.5}, and the whole proof of Theorem \ref{thm:construction-MM} is carried out under more general conditions than those required on the `concrete' system 
  set up in Section \ref{s:assumptions}. \EEE
 
\medskip

\emph{Challenges for generalization.}
A well-formed functional framework includes a concept of solutions that behaves well under the taking of limits, and the existence proof is the first test of this. Our existence proof highlights a central challenge here, in the appearance  of \emph{two} slope functionals $\nuovorel$ and $\Fish$ that both represent rigorous versions of the `Fisher information' term $\calR^*\bigl(\rho,-\ona\upphi'(\dd \rho/\dd \pi)\bigr)$. The chain-rule lower-bound inequality holds under general conditions for $\Fish$ (Theorem~\ref{th:chain-rule-bound}), but the Minimizing-Movement construction leads to the more abstract object $\nuovorel$. Passing to the limit in the minimizing-movement approach requires connecting the two through the inequality $\nuovorel\geq \Fish$. 
 We prove it by first obtaining the inequality $\nuovo \geq \Fish$, cf.\  Proposition
\ref{p:slope-geq-Fish}, under the condition that a solution to the $(\calS, \calR, \calR^*)$ system exists (for instance, by the approach developed in Section \ref{s:ex-sg}).
We then deduce the inequality $\nuovorel \geq \Fish$ under the further condition that $\Fish$ be lower semicontinuous, which can be in turn proved under  a suitable convexity condition (cf.\ Prop.\ \ref{PROP:lsc}). 
 \EEE
  We hope that more effective ways of dealing with these issues will be found in the future.

\smallskip
\emph{Comparison with the Weighted Energy-Dissipation method.}
It would be interesting to develop the analogous variational
 approach
 based on
  studying the \EEE
  limit behaviour as $\ep\downarrow0$
 of the minimizers $(\rho_t,\bj_t)_{t\ge0}$ of the Weighted Energy-Dissipation  (\textrm{WED}) \EEE 
 functional
 \begin{equation}
   \label{eq:69}
   \DVTn_\ep(\rho,\bj ):=\int_0^\pinfty
   \mathrm e^{-t/\ep}
   \Big(\calR(\rho_t,\bj_t)+\frac1\ep\calS(\rho_t)\Big)\,\dd t
 \end{equation}
 among the solutions to the continuity equation with initial datum
 $\rho_0$,
 see \cite{RSSS19}.  Indeed, the \emph{intrinsic character} of the \textrm{WED} functional, which only features the dissipation potential $\calR$, makes it suitable to the present non-metric framework. \EEE

\subsection{Notation}

The following table collects the notation used throughout the paper. 

\begin{center}\bigskip
\newcommand{\specialcell}[2][c]{%
  \begin{tabular}[#1]{@{}l@{}}#2\end{tabular}}
\begin{small}
\begin{longtable}{lll}
$\ona$, $\odiv$ & graph gradient and divergence &\eqref{eq:def:ona-div}\\
$\upalpha(\cdot,\cdot)$ & multiplier in flux rate $\bnu_\rho$ & Ass.~\ref{ass:Psi}\\
$\upalpha^\infty$, $\upalpha_*$ & recession function, Legendre transform & Section~\ref{subsub:convex-functionals} \\
$\upalpha[\cdot|\cdot]$, $\hat\upalpha$ &  measure map, perspective function & Section~\ref{subsub:convex-functionals} \\
$\CER ab$ & set of curves $\rho$ with finite action & \eqref{def:Aab}\\
  $ \|\kappa_V\|_\infty$  \EEE & upper bound on $\kappa$ & Ass.~\ref{ass:V-and-kappa}\\
$
  \Cb 
  $ & space of bdd, ct.\ functions with supremum norm\\
$\CE ab$ & set of pairs $(\rho,\bj )$ satisfying the continuity equation & Def.~\ref{def-CE}\\
  $\rmD_\upphi(u,v)$, $\rmD^\pm_\upphi(u,v)$ & integrands
             defining the Fisher information $\Fish$ & \eqref{subeq:D}\\
$\Fish$ & Fisher-information functional & Def.~\ref{def:Fisher-information}\\
$\edg = V\times V$ & space of edges & Ass.~\ref{ass:V-and-kappa}\\
$\calS$, $\rmD (\calS)$ & driving  entropy \EEE functional and its domain & \eqref{eq:def:S} \& Ass.~\ref{ass:S}\\
$\rmF$ & vector field & \eqref{eq:184}\\
  $\teta_\rho^\pm$,
                & $\rho$-adjusted jump rates & \eqref{def:teta}\\
$\tetapi$ & equilibrium jump rate & \eqref{nu-pi}\\
$\kappa$ & jump kernel &\eqref{eq:def:generator} \& Ass.~\ref{ass:V-and-kappa}\\
$\kernel\kappa\gamma$ & $\gamma \otimes \kappa$ & \eqref{eq:84}\\
$\mathscr L$ & Energy-Dissipation balance functional &\eqref{eq:def:mathscr-L}\\
  $\calM(\Omega;\R^m)$, $\calM^+(\Omega)$ & vector (positive) measures on $\Omega$ & Sec.~\ref{ss:3.1}\\
$\bnu_\rho$ & edge measure in definition of $\calR^*$, $\calR$ &\eqref{eq:def:R*-intro}, \eqref{eq:def:R-intro}, \eqref{eq:def:alpha}\\
$Q$, $Q^*$ & generator and dual generator & \eqref{eq:fokker-planck}\\
$\calR$, $\calR^*$ & dual pair of dissipation potentials & \eqref{eq:def:R*-intro}, \eqref{eq:def:R-intro}, Def.~\ref{def:R-rigorous}\\
$\R_+ := [0,\infty)$ \\
$\sfs$ & symmetry map $(x,y) \mapsto (y,x)$ & \eqref{eq:87}\\
$\nuovorel$ & relaxed slope & \eqref{relaxed-nuovo}\\
  $\Upsilon$ & perspective function associated with
               $\Psi$ and $\upalpha$& \eqref{Upsilon}\\
$V$ & space of states & Ass.~\ref{ass:V-and-kappa}\\
$\upphi$ &  density of $\calS$ & \eqref{eq:def:S} \& Ass.~\ref{ass:S}\\
$\Psi$, $\Psi^*$ & dual pair of dissipation functions & Ass.~\ref{ass:Psi}, Lem.~\ref{l:props:Psi}\\
$\DVTn$ & Dynamic-Variational Transport cost & \eqref{def:W-intro} \& Sec.~\ref{sec:cost}\\
  $\VarWn$ & $\DVTn$- action  & \eqref{def-tot-var}\\
$\sfx,\sfy$ & coordinate maps $(x,y) \mapsto x$ and $(x,y)\mapsto y$ & \eqref{eq:87}\\
\end{longtable}
\end{small}
\end{center}

\subsubsection*{\bf Acknowledgements} M.A.P.\ acknowledges support from NWO grant 613.001.552, ``Large Deviations and Gradient Flows: Beyond Equilibrium". R.R.\ and G.S. acknowledge support from the MIUR - PRIN project 2017TEXA3H ``Gradient flows, Optimal Transport and Metric Measure Structures". O.T.\ acknowledges support from NWO Vidi grant 016.Vidi.189.102, ``Dynamical-Variational Transport Costs and Application to Variational Evolutions". Finally, the authors thank Jasper Hoeksema for insightful and valuable comments during the preparation of this manuscript.

\section{Preliminary results}
\label{ss:3.1}

\subsection{Measure theoretic preliminaries}
Let $(Y,\frB)$ be a measurable space.
When $Y$ is endowed  with \EEE a
(metrizable and separable) topology $\tau_Y$
we will often assume that $\frB$ coincides with the Borel $\sigma$-algebra
$\frB(Y,\tau_Y)$ induced
by $\tau_Y$.
We recall
that $(Y,\frB)$ is called a \emph{standard Borel space}
if it is isomorphic (as a measurable space) to a Borel subset
of a complete and separable metric space; equivalently,
one can find
a Polish topology   $\tau_Y$ \EEE on $Y$ such that $\frB=\frB(Y,\tau_Y)$.
 \par 
We will denote by 
 $\calM(Y;\R^m)$  the space of  $\sigma$-additive measures  on
 $\mu: \frB \to \R^m$ 
 of \emph{finite} total variation
 $\|\mu\|_{TV}: =|\mu|(Y)<\pinfty$, where for every $B\in\frB$
 \[
   |\mu|(B): = \sup \left\{ \sum_{i=0}^\pinfty |\mu(B_i)|\, : \ B_i \in \frB,\, \ B_i \text{ pairwise disjoint}, \ B = \bigcup_{i=0}^\pinfty B_i \right\}.
 \]
 The set function  $|\mu|: \frB \to [0,\pinfty)$  is a positive
 finite
 measure on $\frB$ \cite[Thm.\ 1.6]{AmFuPa05FBVF}
 and $(\calM(Y;\R^m),\|\cdot\|_{TV})$ is a Banach space.
 
 In the case $m=1$, we will simply write $\calM(Y)$,
 and we shall denote the space of \emph{positive} finite
measures on $\frB$ by $\calM^+(Y)$.
   For $m>1$,
 we will identify any element $\mu \in \calM(Y;\R^m)$ with  a vector
 $(\mu^1,\ldots,\mu^m)$, with $\mu^i \in \calM(Y)$ for all
 $i=1,\ldots, m$.
If $\varphi
 =(\varphi^1,\ldots,\varphi^m)\in \Bb(Y;\R^m)$, the set of bounded $\R^m$-valued
 $\frB$-measurable
 maps, the duality between $\mu \in \calM(Y;\R^m)$ and $\varphi$
 can be expressed by \nc
\[
 \langle\mu,\varphi\rangle : = \int_{Y} \varphi \cdot \mu (\dd x) =
 \sum_{i=1}^m \int_Y  \varphi^i(x) \mu^i(\dd x).
 \]

 For every $\mu\in \calM(Y;\R^m)$ and $B\in \frB$
 we will denote by $\mu\mres B$ the restriction of $\mu$ to $B$, i.e.\ 
 $\mu\mres B(A):=\mu(A\cap B)$ for every $A\in \frB$.

 Let $(X,\mathfrak A)$ be another measurable space and let $\sfp:X\to
 Y$
 a measurable map. For every $\mu\in \calM(X;\R^m)$ we will denote by
 $\sfp_\sharp\mu$ the push-forward measure obtained by
 \begin{equation}
   \label{eq:82}
   \sfp_\sharp\mu(B):=\mu(\sfp^{-1}(B))\quad\text{for every }B\in \frB.
 \end{equation}
 For every couple $\mu\in \calM(Y;\R^m)$ and $\gamma\in \calM^+(Y)$
 there exist a unique (up to the modification
 in a $\gamma$-negligible set) $\gamma$-integrable map
 $\frac{\dd\mu}{\dd\gamma}:
 Y\to\R^m$, a $\gamma$-negligible set $N\in \frB$
 and a unique measure $\mu^\perp\in \calM(Y;\R^m)$
 yielding the \emph{Lebesgue decomposition}
 \begin{equation}
   \label{eq:Leb}
   \begin{gathered}
     \mu=\mu^a+\mu^\perp,\quad \mu^a=\frac{\dd\mu}{\dd\gamma}\,\gamma=
     \mu\mres(Y\setminus N),\quad \mu^\perp=\mu\mres N,\quad
     \gamma(N)=0\\
     |\mu^\perp|\perp \gamma,\quad |\mu|(Y)=\int_Y
     \left|\frac{\dd \mu}{\dd\gamma}\right|\,\dd\gamma+|\mu^\perp|(Y).
   \end{gathered}
 \end{equation}
 \subsection{Convergence of measures}
 \label{subsub:convergence-measures}
 Besides the topology of convergence in total variation (induced by
 the norm $\|\cdot\|_{TV}$), we will also consider
 \GGG the topology of \emph{setwise convergence},
 i.e.~the coarsest topology on $\calM(Y;\R^m)$ making all the functions
 \begin{displaymath}
   \mu\mapsto \mu(B)\quad B\in \frB
 \end{displaymath}
 continuous. \EEE
 For a sequence $(\mu_n)_{n\in\N}$
 and a candidate limit $\mu$ in $\calM(Y;\R^m)$
 we have the following equivalent characterizations of
 the corresponding convergence \cite[\S 4.7(v)]{Bogachev07}:
 \begin{enumerate}
    \item Setwise convergence:
   \begin{equation}
     \label{eq:71}
     \lim_{n\to\pinfty}\mu_n(B)=\mu(B)\qquad
     \text{for every set $B\in \frB$}.
   \end{equation}
 \item Convergence in duality with $\Bb(Y;\R^m)$:
   \begin{equation}
     \label{eq:70}
     \lim_{n\to\pinfty}\langle \mu_n,\varphi\rangle=
     \langle \mu,\varphi\rangle
     \qquad
     \text{for every $\varphi\in \Bb(Y;\R^m)$}.
   \end{equation}
 \item Weak topology of the Banach space:
   the sequence $\mu_n$ converges to $\mu$ in the weak topology
   of the Banach space $(\calM(Y;\R^m);\|\cdot\|_{TV})$.
 \item Weak $L^1$-convergence of the densities:
   there exists a common dominating measure
   $\gamma\in \Dom(Y)$
   such that $\mu_n\ll\gamma$, $\mu\ll\gamma$ and
   \begin{equation}
     \label{eq:72}
     \frac{\dd\mu_n}{\dd\gamma}\weakto
     \frac{\dd\mu}{\dd\gamma}
     \quad\text{weakly in }L^1(Y,\gamma;\R^m).
   \end{equation}
 \item Alternative form of weak $L^1$-convergence: \eqref{eq:72} holds \emph{for every} common dominating measure $\gamma$.
 \end{enumerate}
 We will refer to 
 \emph{setwise convergence} for sequences satisfying one of the
 equivalent properties above.
 The above topologies also share the same notion of compact subsets,
 as  stated in the following useful theorem, cf.\ 
 \cite[Theorem 4.7.25]{Bogachev07},   where we shall denote by 
 $\sigma(\calM(Y;\R^m) ; \Bb(Y;\R^m) )$
  the weak topology on $\calM(Y;\R^m)$  induced by the duality with $\Bb(Y;\R^m)$. \EEE
 \begin{theorem}
   \label{thm:equivalence-weak-compactness}
   For every set $\emptyset\neq M\subset \calM(Y;\R^m)$ the following properties are
   equivalent:
   \begin{enumerate}
     \item $M$ has a compact closure in the topology of  setwise
       convergence.
   \item $M$ has a compact closure in the topology
     $\sigma(\calM(Y;\R^m) ; \Bb(Y;\R^m) )$.
   \item $M$ has a compact closure 
     in the weak topology
     of $(\calM(Y;\R^m);\|\cdot\|_{TV})$.
     \item Every sequence in $M$ has a subsequence 
      converging \EEE
       on every set of $\frB$.
     \item There exists a measure $\gamma\in \Dom(Y)$
       such that
       \begin{equation}
         \label{eq:73}
         \forall\,\ep>0\ \exists\,\delta>0:
         \quad
         B\in \frB,\ \gamma(B)\le \delta\quad
         \Rightarrow\quad
         \sup_{\mu\in M}\mu(B)\le \eps.
       \end{equation}
     \item
       There exists a measure $\gamma\in \Dom(Y)$
       such that $\mu\ll\gamma$ for every $\mu\in M$ and
       the set $\{\dd\mu/\dd\gamma:\mu\in M\}$
       has compact closure in the weak topology of $L^1(Y,\gamma;\R^m)$.
   \end{enumerate}
 \end{theorem}
 We also recall a useful characterization of weak compactness in
 $L^1$.
 \begin{theorem}
   \label{thm:L1-weak-compactness}
   Let $\gamma\in \Dom(Y)$ and $\emptyset\neq F\subset
   L^1(Y,\gamma;\R^m)$.
   The following properties are equivalent:
   \begin{enumerate}
   \item $F$ has compact closure in the weak topology of
     $L^1(Y,\gamma;\R^m)$;
   \item
     $F$ is bounded in $L^1(Y,\gamma;\R^m)$ and
     \GGG equi-absolutely continuous, i.e.~\EEE
     \begin{equation}
         \label{eq:73bis}
         \forall\,\ep>0\ \exists\,\delta>0:
         \quad
         B\in \frB,\ \gamma(B)\le \delta\quad
         \Rightarrow\quad
         \sup_{f\in F}\int_B |f|\,\dd\gamma\le \eps.
       \end{equation}
     \label{cond:setwise-compactness-superlinear}
     \item There exists a convex and superlinear function
       $\beta:\R_+\to\R_+$
       such that
       \begin{equation}
         \label{eq:74}
         \sup_{f\in F}\int_Y \beta(|f|)\,\dd\gamma<\pinfty.
       \end{equation}
   \end{enumerate}
 \end{theorem}
The name `equi-absolute continuity' above derives from the interpretation that the {measure} $f\gamma$ is absolutely continuous with respect to $\gamma$ in a uniform manner; `equi-absolute continuity' is a shortening of Bogachev's terminology `$F$ has uniformly absolutely continuous integrals'~\cite[Def.~4.5.2]{Bogachev07}. A fourth equivalent property is equi-integrability with respect to $\gamma$~\cite[Th.~4.5.3]{Bogachev07}, a fact that we will not use.

\medskip

 When $Y$ is endowed with a (separable and metrizable) topology   $\tau_Y$, \EEE
 we will use the symbol $\Cb(Y;\R^m) $  to denote the space of
 bounded $\R^m$-valued continuous functions on  $(Y,\tau_Y)$. \EEE
 We will  consider the corresponding weak topology
 $\sigma(\calM(Y;\R^m);\Cb(Y;\R^m))$
 induced by the duality 
 with $\Cb(Y;\R^m)$.
 Prokhorov's Theorem yields that a subset $M\subset \calM(Y;\R^m)$
 has compact closure in this topology
 if it is bounded in the
 total  variation
 norm and it is equally  tight, i.e.
 \begin{equation}
   \label{eq:47}
   \forall\ep>0\
   \exists\, K\text{ compact in $Y$}:
   \quad
   \sup_{\mu\in M}|\mu|(Y\setminus K)\le \ep.
 \end{equation}
 It is obvious that for a sequence $(\mu_n)_{n\in \N}$
 convergence in total variation
 implies setwise convergence (or in duality with bounded measurable
 functions),
 and setwise convergence implies weak convergence in duality with
 bounded continuous functions.
 \subsection{Convex functionals  and concave transformations of measures}
 \label{subsub:convex-functionals}
We will use the following construction several times.
Let 
$\uppsi:\R^m\to [0,\pinfty]$ be convex and lower semicontinuous
and let us denote
by $\uppsi^\infty:\R^m\to [0,\pinfty]$ its recession function
\begin{equation}
  \label{eq:3}
  \uppsi^\infty(z):=\lim_{t\to\pinfty}\frac{\uppsi(tz)}t=\sup_{t>0}\frac{\uppsi(tz)-\uppsi(0)}t,
\end{equation}
which is a convex, lower semicontinuous, and positively $1$-homogeneous map  with $\uppsi^\infty(0)=0$. \EEE
We define the functional
$\calF_\uppsi:\calM(Y;\R^m) \times \calM^+(Y)\mapsto [0,\pinfty]$ by
\begin{equation}
\label{def:F-F}
\calF_\uppsi(\mu|\nu) :=
\int_Y \uppsi \Bigl(\frac{\dd \mu}{\dd \nu}\Bigr)\,\dd\nu+
\int_Y \uppsi^\infty\Bigl(\frac{\dd \mu^\perp}{\dd |\mu^\perp|}\Bigr) \,
\dd |\mu^\perp|,\qquad \text{for }\mu=\frac{\dd \mu}{\dd \nu}\nu+\mu^\perp.
\end{equation}
Note that when $\uppsi$ is superlinear then $\uppsi^\infty(x)=\pinfty$
in $\R^m\setminus\{0\}$. Equivalently, 
\begin{equation}
  \label{eq:5}
  \text{$\uppsi$ superlinear,}\quad
  \calF_\uppsi(\mu|\nu)<\infty\quad\Rightarrow\quad
  \mu\ll\nu,\quad
  \calF_\uppsi(\mu|\nu)=
\int_Y \uppsi \Bigl(\frac{\dd \mu}{\dd \nu}\Bigr)\,\dd\nu.
\end{equation}
We collect in the next Lemma a list of useful properties.
%
\begin{lemma}
  \label{l:lsc-general}\ 
  \begin{enumerate}
  \item\label{l:lsc-general:i3}
    When $\uppsi$ is also positively $1$-homogeneous, then
    $\uppsi\equiv \uppsi^\infty$, $\calF_\uppsi(\cdot|\nu)$ is
    independent of $\nu$ and  will also be denoted by
    $\calF_\uppsi(\cdot)$: it satisfies
    \begin{equation}
      \label{eq:78}
       \calF_\uppsi(\mu) \EEE =\int_Y
      \uppsi\left(\frac{\dd\mu}{\dd\gamma}\right)
      \,\dd\gamma\quad
      \text{for every }\gamma\in \calM^+(Y)\text{ such that } \mu\ll\gamma.
    \end{equation}
  \item If
    $\hat
    \uppsi:\R^{m+1}\to[0,\infty]$ denotes
    the 1-homogeneous, convex, perspective function associated  with \EEE
    $\uppsi$ by
    \begin{equation}
      \label{eq:76}
     \hat \uppsi(z,t):=
      \begin{cases}
        \uppsi(z/t)t&\text{if }t>0,\\
        \uppsi^\infty(z)&\text{if }t=0,\\
        \pinfty&\text{if }t<0,
      \end{cases}
    \end{equation}
    then
    \begin{equation}
      \label{eq:77}
      \calF_\uppsi(\mu|\nu)=\calF_{\hat \uppsi}(\mu,\nu)\quad
      \text{for every }  (\mu,\nu) \EEE \in \calM(Y;\R^m)\times \calM^+(Y)
    \end{equation}
     with $\calF_{\hat \uppsi}$ defined as in \eqref{eq:78}. \EEE
      \item In particular, 
        if $\gamma\in \calM^+(Y)$ is a common dominating measure such
    that $\mu=u\gamma$, $\nu=v\gamma$, and $Y':=\{x\in Y:v(x)>0\}$
    we also have
    \begin{equation}
      \label{eq:67}
      \calF_\uppsi(\mu|\nu)=
      \int_Y \hat\uppsi(u,v)\,\dd\gamma=\int_{Y'} \uppsi(u/v)v\,\dd\gamma+
      \int_{Y\setminus Y'} \uppsi^\infty(u)\,\dd\gamma.
    \end{equation}
  \item The functional $\calF_\uppsi$ is convex;
    if $\uppsi$ is also positively $1$-homogeneous then
    \begin{equation}
      \label{eq:81}
      \begin{aligned}
        \calF_\uppsi(\mu+\mu')&\le \calF_\uppsi(\mu)+\calF_\uppsi(\mu')\\
        \calF_\uppsi(\mu+\mu')&= \calF_\uppsi(\mu)+\calF_\uppsi(\mu')\quad
        \text{if }\mu\perp\mu'.
      \end{aligned}      
    \end{equation}
  \item Jensen's inequality:
    \begin{equation}
      \label{eq:79}
        \hat\uppsi(\mu^a(B),\nu(B)) 
      +\uppsi^\infty(\mu^\perp(B)) \EEE
      \le
      \calF_\uppsi(\mu\mres B\GGG |\EEE\nu\mres B)\quad
      \text{for every }B\in \frB
    \end{equation}
     (with $\mu = \mu^a + \mu^\perp $ the Lebesgue decomposition of $\mu$ w.r.t.\ $\nu$). \EEE
  \item If $\uppsi(0)=0$ then for every $\mu\in \calM(Y,\R^m)$,
    $\nu,\nu'\in \calM^+(Y)$ 
    \begin{equation}
      \label{eq:80}
      \nu\le \nu'\quad\Rightarrow\quad
      \calF_\uppsi(\mu|\nu)\ge \calF_\uppsi(\mu|\nu').
    \end{equation}
  \item \label{l:lsc-general:i1}
        $\calF_\uppsi$ is \GGG sequentially \EEE lower
    semicontinuous in $\calM(Y;\R^m) \times \calM^+(Y)$ with respect
    to the topology of setwise convergence.
  \item \label{l:lsc-general:i2}
    If $\frB$ is the Borel family induced
    by a \GGG Polish 
     topology  $\tau_Y$  on $Y$, \EEE
    $\calF_\uppsi$ is
    lower semicontinuous with respect to ~weak convergence  (in duality with continuous bounded functions). \EEE
  \end{enumerate}
\end{lemma}
\begin{proof}
  \GGG
  The above properties are mostly well known; we
  give a quick sketch of the proofs
  of the various claims for the ease of the reader.
  
  \medskip\noindent \textit{(1)}
  Let us set $u:=\dd \mu/\dd\nu$, $u^\perp:=\dd\mu^\perp/\dd|\mu|$
  and let $N\in\frB$ $\nu$-negligible such that $\mu^\perp=\mu\mres
  N$. We also set $N':=\{y\in Y\setminus N:|u(y)|> 0\}$; notice that
  $\nu\mres N'\ll |\mu|$. If $v$ is
  the Lebesgue density of $|\mu|$ w.r.t.~$\gamma$, since
  $\uppsi=\uppsi^\infty$ is positively $1$-homogeneous and $\uppsi(0)=0$, we have
  \begin{align*}
    \calF_\uppsi(\mu|\nu)
    &=\int_{N'}\uppsi(u)\,\dd \nu+
      \int_N \uppsi(u^\perp)\,\dd|\mu^\perp|
      =\int_{N'}\uppsi(u)/|u|\,\dd |\mu|+
      \int_N \uppsi(u^\perp)\,\dd|\mu^\perp|
      \\&=\int_{N'}v\uppsi(u)/|u|\,\dd \gamma+
    \int_N v\uppsi(u^\perp)\,\dd\gamma=
    \int_{N'}\uppsi(uv/|u|)\,\dd \gamma+
    \int_N \uppsi(u^\perp v)\,\dd\gamma
    \\&=    \int_{N'}\uppsi(\dd\mu/\dd\gamma)\,\dd \gamma+
    \int_N \uppsi(\dd \mu/\dd\gamma)\,\dd\gamma=
    \int_Y \uppsi(\dd \mu/\dd\gamma)\,\dd\gamma=\calF_\uppsi(\mu|\gamma),
  \end{align*}
  where we also used the fact that $|\mu|(Y\setminus (N\cup N'))=0$,
  so that $\dd \mu/\dd\gamma=0$ $\gamma$-a.e.~on $Y\setminus (N\cup
  N').$

  \medskip
  \noindent\textit{(2)}
  Since $\hat\uppsi$ is $1$-homogeneous, we can apply the previous
  claim and evaluate
  $\calF_{\hat\uppsi}(\mu,\nu)$ by choosing the dominating measure
  $\gamma:=\nu+\mu^\perp$.

  \medskip
  \noindent\textit{(3)} It is an immediate consequence of the first two
  claims.

  \medskip
  \noindent\textit{(4)}
  By \eqref{eq:77} it is sufficient to consider the $1$-homogeneous
  case.
  The convexity then follows by the convexity of $\uppsi$ and
  by choosing a common dominating measure to represent the integrals.
  Relations \EEE \eqref{eq:81} are also immediate.

  \medskip
  \noindent\textit{(5)}
  Using \eqref{eq:77} and selecting a dominating measure $\gamma$ with
  $\gamma(B)=1$, Jensen's inequality applied to the convex functional
  $\hat\uppsi$ yields
  \begin{displaymath}
    \hat\uppsi(\mu(B),\nu(B))=
    \hat\uppsi\Big(\int_B \frac{\dd\mu}{\dd\gamma}\,\dd\gamma,
    \int_B \frac{\dd\nu}{\dd\gamma}\,\dd\gamma\Big)\le
    \int_B
    \hat\uppsi\Big(\frac{\dd\mu}{\dd\gamma},\frac{\dd\nu}{\dd\gamma}\Big)\,\dd\gamma
    =\calF_{\hat\uppsi}(\mu\mres B,\nu\mres B).
  \end{displaymath}
  Applying now the above inequality to the mutally singular couples
  $(\mu^a,\nu)$ and $(\mu^\perp,0)$
  and using the second identity of \eqref{eq:81} we obtain \eqref{eq:79}.

  \medskip
  \noindent\textit{(6)}
  We apply \eqref{eq:77} and the first identity of \eqref{eq:67},
  observing that if $\uppsi(0)=0$ then $\hat\uppsi$ is decreasing with
  respect to its second argument.

  \medskip
  \noindent\textit{(7)}
  By \eqref{eq:77} it is not restrictive to assume that
  $\Psi$ is $1$-homogeneous.
  If $(\mu_n)_n$ is a sequence setwise converging to $\mu$ in
  $\calM(Y;\R^m)$ we can find
  a common dominating measure $\gamma$ such that \eqref{eq:72} holds.
  The claimed property is then reduced to the weak lower semicontinuity of the functional
  \begin{equation}
    \label{eq:121}
    u\mapsto \int_Y \Psi(u)\,\dd\gamma
  \end{equation}
  in $L^1(Y,\gamma;\R^m)$. 
  Since the functional of \eqref{eq:121}
  is convex and strongly lower semicontinuous in $L^1(Y,\gamma;\R^m)$
  (thanks to Fatou's Lemma), it is weakly lower
  semicontinuous as well.
  
  \medskip
  \noindent\textit{(8)}
  It follows by the same argument of \cite[Theorem 2.34]{AmFuPa05FBVF},
  by using a suitable dual formulation which holds
  also in Polish
  spaces, where all the finite Borel measures are Radon
  (see e.g.~\cite[Theorem 2.7]{LMS18} for positive measures).
\end{proof}

\subsubsection*{Concave transformation of vector measures}
\label{subsub:concave-transformation}
Let us set $\R_+:=[0,\pinfty[$, $\R^m_+:=(\R_+)^m$, and let
$\upalpha:\R^m_+\to\R_+$ be a continuous and concave function. It is
obvious that $\upalpha$ is non-decreasing with respect to each variable.
As for \eqref{eq:3}, the recession function $\upalpha^\infty$ is defined
by
\begin{equation}
  \label{eq:1}
  \upalpha^\infty(z):=\lim_{t\to\pinfty}\frac{\upalpha(tz)}t=\inf_{t>0}\frac{\upalpha(tz)-\upalpha(0)}t,\quad
  z\in \R^m_+.
\end{equation}
We define the corresponding map
$\Aalpha:\calM(Y;\R^m_+)\times\calM^+(Y)\to\calM^+(Y)$ by
\begin{equation}
  \label{eq:6}
  \Aalpha[\mu|\gamma]:=
  \upalpha\Bigl(\frac{\dd\mu}{\dd\gamma}\Bigr)\gamma+
  \upalpha^\infty\Bigl(\frac{\dd\mu}{\dd
    |\mu^\perp|}\Bigr)|\mu^\perp|\quad
  \mu\in \calM(Y;\R^m_+),\ \gamma\in \calM_+(Y),
\end{equation}
where as usual $\mu=\frac{\dd\mu}{\dd\gamma}\gamma+\mu^\perp$ is the
Lebesgue decomposition of $\mu$ with respect to ~$\gamma$; 
 in what follows, we will use the short-hand $\mu_\gamma := \frac{\dd\mu}{\dd\gamma}\gamma$.
We also mention in advance that, for shorter notation we will write   $\Aalpha[\mu_1,\mu_2|\gamma]$ in place of
$\Aalpha[(\mu_1,\mu_2)|\gamma]$. 
 \EEE
Like for $\calF$, it is not difficult to check that
$\Aalpha[\mu|\gamma]$
is independent of $\gamma$ if $\upalpha$ is positively $1$-homogeneous
(and thus coincides with $\upalpha^\infty$).
If we define the perspective function $\hat \upalpha:\R_+^{m+1}\to\R_+$
\begin{equation}
  \label{eq:83}
  \hat \upalpha(z,t):=
  \begin{cases}
    \upalpha(z/t)t&\text{if }t>0,\\
    \upalpha^\infty(z)&\text{if }t=0
  \end{cases}
\end{equation}
we also get $\Aalpha[\mu|\gamma]=\hat\upalpha(\mu,\gamma)$.

We denote by $\upalpha_*:\R^m_+\to[-\infty,0]$ the upper semicontinuous concave conjugate of $\upalpha$
\begin{equation}
  \label{eq:4}
  \upalpha_*(y):=\inf_{x\in \R^m_+} \left( y\cdot x-\upalpha(x)\right),\quad
  D(\upalpha_*):=\big\{y\in \R^m_+:\upalpha_*(y)>-\infty\big\}.
\end{equation}
The function $\upalpha_*$ provides simple affine upper bounds  for \EEE $\upalpha$
\begin{equation}
  \label{eq:8}
  \upalpha(x)\le x\cdot y-\upalpha_*(y)\quad\text{for every }y\in D(\upalpha_*)
\end{equation}
and Fenchel duality yields
\begin{equation}
  \label{eq:7}
  \upalpha(x)=\inf_{y\in \R^m_+} \left( y\cdot x-\upalpha_*(y) \right)=
  \inf_{y\in D(\upalpha_*)}\left( y\cdot x-\upalpha_*(y) \right).
\end{equation}
 We will also use that 
\begin{equation}
\label{form-4-alpha-infty}
\upalpha^\infty(z)  = \inf_{y\in D(\upalpha_*)}  y \cdot z\,. 
\end{equation}
Indeed,  on the one hand 
for every $y \in  D(\upalpha_*)$ and $t>0$ we have that 
\[
\upalpha^\infty(z) \leq \frac1t \left( \alpha(tz) - \alpha(0)\right) \leq \frac1t \left( y \cdot (tz) -\upalpha(0) - \upalpha^*(y) \right);
\]
by the arbitrariness of $t>0$, we conclude that $\upalpha^\infty(z) \leq  y \cdot z$ for every $y \in  D(\upalpha_*)$.
On the other hand,  by \eqref{eq:7} we have 
\[
\begin{aligned}
\upalpha^\infty(z)  = \inf_{t>0}\frac{\upalpha(tz)-\upalpha(0)}t  & =  \inf_{t>0} \inf_{y\in D(\upalpha^*)} \frac{y \cdot (tz) -\upalpha^*(y) - \upalpha(0)}t 
\\
& 
= \inf_{y\in D(\upalpha^*)} \left(  y \cdot z +\inf_{t>0} \frac{-\upalpha^*(y) - \upalpha(0)}t  \right)  = \inf_{y\in D(\upalpha^*)} y \cdot z,
\end{aligned}
\]
where we have used that $-\upalpha^*(y) - \upalpha(0) \geq 0$ since $\upalpha(0) = \inf_{y\in  D(\upalpha^*) } ({-}\upalpha^*(y))$. 
 \EEE
\par
For every Borel set $B\subset Y$, Jensen's inequality yields  (recall the notation $\mu_\gamma = \frac{\dd\mu}{\dd\gamma}\gamma$) \EEE
\begin{equation}
  \label{eq:9}
  \begin{aligned}
  \Aalpha[\mu|\gamma](B)&\le
\upalpha\Bigl(\frac{\mu_\gamma(B)}{\gamma(B)}\Bigr)\gamma(B)+
    \upalpha^\infty(\mu^\perp(B))\\
    \Aalpha[\mu|\gamma](B)&\le\upalpha(\mu(B))\quad\text{if
    }\upalpha=\upalpha^\infty
    \text{ is $1$-homogeneous.}
  \end{aligned}
  \end{equation}
  In fact, for every $y,y'\in D(\upalpha_*)$,
  \begin{align*}
  \Aalpha[\mu|\gamma](B)
  &=\int_B  \Aalpha[\mu|\gamma]\le
  \int_B \Bigl(y\cdot
  \frac{\dd\mu}{\dd\gamma}-\upalpha_*(y)\Bigr)\,\dd\gamma+
  \int_B \Bigl(y'\cdot
  \frac{\dd\mu}{\dd |\mu^\perp|}\Bigr)\,\dd\,|\mu^\perp|
  \\&=y\cdot \mu_\gamma(B)-\upalpha_*(y)\gamma(B)
  +y'\cdot \mu^\perp(B).
\end{align*}
Taking the infimum with respect to ~$y$ and $y'$,  and recalling \eqref{eq:7} and
\eqref{form-4-alpha-infty}, \EEE
  we find \eqref{eq:9}.
Choosing $y=y'$ in the previous formula we also obtain the linear
upper bound
\begin{equation}
  \label{eq:31}
  \Aalpha[\mu|\gamma]\le y\cdot\mu-\upalpha_*(y)\gamma
  \quad\text{for every }y\in D(\upalpha_*).
\end{equation}
\subsection{Disintegration and kernels}
\label{subsub:kernels}
Let $(X,\frA)$ and $(Y,\frB)$ be measurable spaces
and let $\big(\kappa(x,\cdot)\big)_{x\in X}$ be a
$\frA$-measurable family of measures in $\calM^+(Y)$, i.e.
\begin{equation}
  \label{eq:17}
  \text{for every $B\in \frB$,}\quad
  x\mapsto \kappa(x,B)\ \text{is $\frA$-measurable}. 
\end{equation}
We will set 
\begin{equation}
  \label{eq:155}
  \kappa_Y(x):=\kappa(x,Y),\quad
   \|\kappa_Y\|_\infty \EEE :=\sup_{x\in X}   |\kappa|(x,Y), \EEE
\end{equation}
and we say that $\kappa$ is a bounded kernel if $\|\kappa_Y\|_\infty$ is finite.
If $\gamma\in \calM^+(X)$
and
\begin{equation}
  \label{eq:89}
  \text{$\kappa_Y$ is $\gamma$-integrable, i.e.}\quad
  \int_X \kappa(x,Y)\,\gamma(\dd x)<\pinfty,
\end{equation}
then Fubini's Theorem \cite[II, 14]{Dellacherie-Meyer78} shows that
there exists a unique measure
$\kernel\kappa\gamma(\dd x,\dd y)=
\gamma(\dd x)\kappa(x,\dd y)$ on $(X\times Y,\frA\otimes \frB)$ such that
\begin{equation}
  \label{eq:84}
  \kernel\kappa\gamma(A\times B)=\int_A \kappa(x,B)\,\gamma(\dd x)\quad\text{for
    every }A\in \frA,\ B\in \frB.
\end{equation}
If $X=Y$, the measure $\gamma$
is called \emph{invariant}
if $\kernel\kappa\gamma$ has the same marginals; equivalently
\begin{equation}
  \label{eq:156}
  \sfy_\sharp \kernel\kappa\gamma(\dd y)=
  \int_X \kappa(x,\dd y)\gamma(\dd x)=
  \kappa_Y(y)\gamma(\dd y),
\end{equation}
 where $
\sfy:\edg \to V  $ denotes the projection on the second component, cf.\ \eqref{eq:87} ahead. 
We say that \EEE
$\gamma$ is \emph{reversible} if it satisfies \emph{the detailed
  balance condition}, i.e.
 $\kernel\kappa\gamma$ is symmetric:
$\sfs_\sharp \kernel\kappa\gamma=\kernel\kappa\gamma$.
The concepts of invariance and detailed balance correspond to the analogous concepts in stochastic-process theory; see Section~\ref{ss:assumptions}.
It is immediate to check that reversibility implies invariance.


If $f:X\times Y\to \R$ is a positive or bounded measurable function,
then
\begin{equation}
  \label{eq:86}
  \text{the map $x\mapsto \kappa f(x):=\int_Y f(x,y)\kappa(x,\dd y)$ is $\frA$-measurable}
\end{equation}
and
\begin{equation}
  \label{eq:85}
  \int_{X\times Y}f(x,y)\,\kernel\kappa\gamma(\dd x,\dd y)=
  \int_X\Big(\int_Y f(x,y)\,\kappa(x,\dd y)\Big)\gamma(\dd x).
\end{equation}
Conversely, if $X,Y$ are standard Borel spaces,
$\boldsymbol\kappa\in \calM^+(X\times Y)$ (with the product
$\sigma$-algebra) and
the first marginal $\sfp^X_\sharp \boldsymbol\kappa $ of $\boldsymbol\kappa$
is absolutely continuous with respect to ~$\gamma\in \calM^+(V)$, then
we may apply the disintegration Theorem \cite[Corollary
10.4.15]{Bogachev07}
to find a $\gamma$-integrable kernel $(\kappa(x,\cdot))_{x\in X}$
such that $\boldsymbol\kappa=\kernel\kappa\gamma$.

We will often apply the above construction in two cases:
when $X=Y:=V$, the main domain of our evolution problems
(see Assumptions \ref{ass:V-and-kappa} below),
and when $X:=I=(a,b)$ is an interval of the real line
endowed with the Lebesgue measure $\Lebone$.
In this case, we will denote by $t$ the variable in $I$
and by $(\mu_t)_{t\in X}$
a measurable family in $\calM(Y)$ parametrized by $t\in I$:
\begin{equation}
  \label{eq:99}
  \text{if }\int_I \mu_t(Y)\,\dd t<\pinfty
  \text{ then we set }
  \mu_\Lebone\in \calM(I\times Y),\quad
  \mu_\Lebone(\dd t,\dd y)=\lambda(\dd t)\mu_t(\dd y).
\end{equation}

\begin{lemma}
  \label{le:kernel-convergence}
  If $\mu_n\in \calM(X)$ is a sequence converging to $\mu$ setwise
  and $(\kappa(x,\cdot))_{x\in X}$ is a \emph{bounded} measurable
  kernel
  in $\calM^+(Y)$, then
  $\kernel\kappa{\mu_n}\to \kernel\kappa\mu$
  setwise in $\calM(X\times Y,\frA\otimes\frB)$.

  If $X,Y$ are Polish spaces and $\kappa$ also satisfies the weak Feller property, i.e.
  \begin{equation}
    \label{eq:Feller}
    x\mapsto \kappa(x,\cdot)\quad\text{is weakly continuous in }\calM^+(Y),
  \end{equation}
   (where `weak' means in duality with continuous bounded functions), \EEE
  then for every weakly converging sequence
  $\mu_n\to\mu$ in $\calM(X)$ we have
  $\kernel\kappa{\mu_n}\to \kernel\kappa\mu$
  weakly as well.
\end{lemma}
\begin{proof}
  If $f:X\times Y\to \R$ is a bounded $\frA\otimes\frB$-measurable map,
  then by \eqref{eq:86} also the map $\kappa f$ is bounded and $\frA$-measurable
  so that
  \begin{displaymath}
  \lim_{n\to\pinfty}\int_{X\times Y} f\,\dd \kernel\kappa{\mu_n}=
  \lim_{n\to\pinfty}\int_{X} \kappa f\,\dd \mu_n=
  \int_X \kappa f\,\dd \mu=\int_{X\times Y} f\,\dd \kernel\kappa\mu,
\end{displaymath}
  showing the setwise convergence. The other statement follows by a
  similar argument.
\end{proof}

\nc
\section{Jump processes, large deviations, and their generalized gradient structures}
\label{s:assumptions}

\subsection{The systems of this paper}
\label{ss:assumptions}

In the Introduction we described jump processes on~$V$ with kernel $\kappa$, and showed that the evolution equation $\partial_t\rho_t = Q^*\rho_t$ for the law $\rho_t$ of the process is a generalized gradient flow characterized by a driving functional $\calS$ and a dissipation potential~$\calR^*$. 

The mathematical setup of this paper is slightly different. Instead of starting with an evolution equation and proceeding to the generalized gradient system, our mathematical development  starts with the generalized gradient system; we then consider the equation to be defined by this system. In this Section, therefore, we describe assumptions that we make on~$\calS$ and $\calR^*$ that will allow us to set up the rigorous functional framework for the evolution equation~\eqref{eq:GF-intro}.

We first state the assumptions about the sets $V$ of `vertices' and
$\edg:=V\times V$ of `edges'. `Edges' are identified with ordered
pairs $(x,y)$ of vertices $x,y\in V$. 
We will denote by $\mathsf x,\mathsf y:\edg\to V$
and $\symmapn:\edg\to\edg$ the coordinate and the symmetry maps
defined by
\begin{equation}
\mathsf x(x,y):=x,\quad
\mathsf y(x,y):=y,\quad
\symmapn(x,y): = (y,x)\quad
\label{eq:87}
\text{for every }x,y\in V.
\end{equation}

\begin{Assumptions}{$V\!\pi\kappa$}
\label{ass:V-and-kappa}
We assume that
\begin{equation}
\label{locally-compact}
\begin{gathered}
  \text{$(V,\frB,\pi)$ is a standard Borel measure space,
    $\pi\in \calM^+(V)$,
  }
\end{gathered}
\end{equation}
\medskip
\noindent
$(\kappa(x,\cdot))_{x\in V}$ is bounded kernel in $\calM^+(V)$ (see
\S \ref{subsub:kernels}),
satisfying the \emph{detailed-balance condition} 
\begin{equation}\label{detailed-balance}
\int_{A} \kappa(x,B)\,\pi(\dd x)  = \int_{B} \kappa(y,A)\,\pi(\dd y)
\qquad \text{for all } A,B\in
\frB,
\end{equation}
and the uniform upper bound
 \begin{equation}
 \label{bound-unif-kappa}
 \|\kappa_V\|_\infty= \EEE \sup_{x\in V} \,\kappa(x,V) <\pinfty.
 \end{equation}
\end{Assumptions}

The measure $\pi\in\calM^+(V)$ often is referred to as the
\emph{invariant measure}, and it will be stationary under the evolution generated by the generalized gradient system.
By Fubini's Theorem (see \S\,\ref{subsub:kernels}) we also introduce the measure
$\tetapi$ \nc
on $\edg$ given by
\begin{equation}
\label{nu-pi}
\tetapi(\dd x\,\dd y) =
\kernel\kappa\pi(\dd x,\dd y)=\pi(\dd x)\kappa (x,\dd y),\quad
\tetapi(A{\times}B) = \int_{A} \kappa(x,B)\, \pi(\dd x) \,.
\end{equation}
Note that  the invariance of the  measure $\pi$  \EEE and
\nc the detailed balance condition \eqref{detailed-balance} can be
rephrased in terms of 
 $\tetapi$ 
as 
\begin{equation}
\label{symmetry-nu-pi}
  \mathsf x_\sharp \tetapi=\mathsf y_\sharp\tetapi,\qquad\nc
\symmap \tetapi = \tetapi\,. \EEE
\end{equation}
Conversely, if 
we choose a symmetric measure $\tetapi\in \calM^+(\edg)$
such that
\begin{equation}
  \label{eq:88}
  \mathsf x_\sharp \tetapi\ll\pi,\quad
  \frac{\dd (\mathsf x_\sharp \tetapi)}{\dd \pi}\le  \|\kappa_V\|_\infty \EEE <\pinfty
  \quad\text{$\pi$-a.e.}
\end{equation}
then the disintegration Theorem \cite[Corollary 10.4.15]{Bogachev07}
shows the existence of a bounded
measurable kernel $(\kappa(x,\cdot))_{x\in X}$ satisfying
\eqref{detailed-balance} and \eqref{nu-pi}.
\medskip

\nc We next turn to the driving functional,
 which  is given by the construction in \EEE \eqref{def:F-F} and \eqref{eq:5}
for a superlinear density $\uppsi=\half\upphi$ and for the choice $\gamma=\pi$.
\begin{Assumptions}{$\calS\upphi$}
\label{ass:S}
The driving functional $\calS: \calM^+(V) \to [0,\pinfty]$ is of the form
\begin{equation}
\label{driving-energy}
\calS(\rho): =
 \half\calF_{\upphi}(\rho|\pi)=\nc
\begin{cases}
  \displaystyle
     \half\int_V \upphi\Bigl(\frac{\dd \rho}{\dd\pi}\Bigr)\,\dd\pi
  &\text{if } \rho \ll \pi,
\\
\pinfty &\text{otherwise,}
\end{cases} 
\end{equation}
with
\begin{multline}
\label{cond-phi}
 \upphi\in \mathrm{C}([0,\pinfty))\cap\mathrm{C}^1((0,\pinfty)),\; \min \upphi= 0, \text{ and $\upphi$ is convex}\\
 \text{with superlinear growth at infinity.} 
 \end{multline}
\end{Assumptions}

\noindent Under these assumptions the functional $\calS$ is lower
semicontinuous on $\calM^+(V)$ both with respect to the  topology of 
setwise convergence, and \EEE any compatible weak topology (see Lemma~\ref{l:lsc-general}). 
A central example was already mentioned in the introduction, i.e.\ the Boltzmann-Shannon entropy function 
\begin{equation}
\label{logarithmic-entropy}
\upphi(s) = s\log s - s + 1, \qquad s\geq 0.
\end{equation}
Finally, we state our assumptions on the dissipation.
\begin{Assumptions}{$\calR^*\Psi\upalpha$}
\label{ass:Psi}
We assume that the {dual}  dissipation density $\Psi^*$ satisfies
\begin{equation}
\left.\begin{gathered}
\label{Psi1}
\Psi^*: \R \to [0,\pinfty) \text{ is convex,  differentiable, even, with  $\Psi^*(0)=0$, and} \\ \lim_{|\xi|\to\infty} \frac{\Psi^*(\xi)}{|\xi|} =\pinfty\, .
\end{gathered}\quad\right\}
\end{equation}
The flux density map 
$\upalpha: [0,\pinfty) \times [0,\pinfty) \to [0,\pinfty)$,
with $\upalpha\not\equiv0$,
 is continuous, concave, symmetric:
\begin{equation}
\label{alpha-symm}
\text{} \upalpha(u_1,u_2) = \upalpha(u_2,u_1) \quad\text{  for all } u_1,\, u_2 \in [0,\pinfty),
\end{equation}
and its recession function $\upalpha^\infty$ vanishes on the boundary
of $\R_+^2$:
\begin{equation}
  \label{alpha-0}
  \text{for every }u_1,u_2\in \R^2_+:\quad
  u_1u_2=0\quad\Longrightarrow\quad
  \upalpha^\infty(u_1,u_2) = 0.
\end{equation}
\end{Assumptions}
\medskip
Note that since $\upalpha$ is nonnegative, concave, and not trivially
$0$, it cannot vanish in the interior of $\R^2_+$, i.e.~
\begin{equation}
  \label{eq:38}
  u_1u_2>0\quad\Rightarrow\quad
  \upalpha(u_1,u_2)>0.
\end{equation}
\nc
The examples that we gave in the introduction of the cosh-type dissipation~\eqref{choice:cosh} and the quadratic dissipation~\eqref{choice:quadratic} both fit these assumptions;  other examples are
\[
\upalpha(u,v) = 1 \qquad\text{and}\qquad \upalpha(u,v) = u+v.
\]

In some cases we will use an additional property, namely that $\upalpha$ is positively $1$-homogeneous, i.e.\ $\upalpha(\lambda u_1,\lambda u_2) = \lambda \upalpha(u_1,u_2)$ for all $\lambda \geq0$.
This $1$-homogeneity is automatically satisfied under
 the compatibility condition \EEE
 \eqref{cond:heat-eq-2},
with the Boltzmann entropy function $\upphi(s) = s\log s - s + 1$. 

Concaveness of $\upalpha$ is a natural assumption
in view of the convexity of $\calR$ \nc
(cf.\ \RICKYNEW Remark \ref{rem:alpha-concave}  and \EEE  Lemma \ref{l:alt-char-R} ahead), while $1$-homogeneity
will make the 
definition of $\mathcal{R}$ independent of the choice of a reference measure.
It is interesting to observe that  the
concavity and symmetry conditions, that one has to naturally assume to ensure the aforementioned properties of $\calR$, were singled out for the analog of the function
$\upalpha$ in the construction of the distance yielding the quadratic gradient structure  of \cite{Maas11}.
\medskip

The choices for $\Psi^*$ above generate corresponding properties for the Legendre dual $\Psi$: 
\begin{lemma}
\label{l:props:Psi}
Under Assumption~\ref{ass:Psi}, the function $\Psi: \R\to\R$  is even and  satisfies 
\begin{subequations}
\label{props:Psi}
\begin{gather}
\label{basic-props-Psi}
0=\Psi(0) < \Psi(s) < \pinfty \text{ for all }s\in \R\setminus \{0\}.\\
\Psi\text{ is strictly convex, \GGG strictly increasing, \EEE  and superlinear.}
\label{eq:propsPsi:Psi-strictly-convex}
\end{gather}
\end{subequations}
\end{lemma}
\begin{proof}
The superlinearity of $\Psi^*$ implies that $\Psi(s)<\pinfty$ for all $s\in \R$, and similarly the finiteness of $\Psi^*$ on $\R$ implies that $\Psi$ is superlinear. Since $\Psi^*$ is even, $\Psi$ is convex and even, and therefore $\Psi(s) \geq \Psi(0) =  \sup_{\xi\in \R} [-\Psi^*(\xi)] = 0.$
Furthermore,  since for all $p\in \R$, $\argmin_{s\in\R} (\Psi(s) -p
s)= \partial \Psi^*(p)$ (see e.g.~\cite[Thm.\
11.8]{Rockafellar-Wets98}) and $\Psi^*$ is differentiable at every
$p$, we conclude that $\argmin_s (\Psi(s)-ps ) = \{(\Psi^*)'(p)\}$;
therefore each point of the graph of $\Psi$ is an exposed point. It
follows that $\Psi$ is strictly convex, and $\Psi(s)>0$ for all
$s\not=0$.
\end{proof}

As described in the introduction, we use $\Psi$, $\Psi^*$, and
$\upalpha$ to define the dual pair of dissipation potentials $\calR$
and~$\calR^*$, which
for a couple of measures $\rho=u\pi\in \Dom(V)$ and $\bj\in \calM(\edg)$
 are \EEE formally given  by
\begin{equation}\label{eq:dissipation-pair}
  \calR(\rho,\bj ) :=
  \frac12
  \nc \int_{\edg} \Psi\left(2 \frac{\dd \bj }{\dd \bnu_\rho}\right)\dd\bnu_\rho, \qquad
  \calR^*(\rho,\xi) :=
  \frac12\nc\int_{\edg} \Psi^*(\xi) \, \dd\bnu_\rho,
\end{equation}
with
\begin{equation}
	\label{def:nu_rho}
\bnu_\rho(\dd x\,\dd y) := \upalpha\big(u(x),u(y)\big)\,\tetapi(\dd x\,\dd y)
= \upalpha\big(u(x),u(y)\big)\,\pi(\dd x)\kappa(x,\dd y).
\end{equation}
This expression for the edge measure $\bnu_\rho$ also is implicitly present in the structure built in \cite{Maas11,Erbar14}. \EEE
The above  definitions are made rigorous in Definition~\ref{def:R-rigorous}  and in \eqref{eq:27}  below.

The three sets of 
 conditions 
 above, Assumptions~\ref{ass:V-and-kappa}, \ref{ass:S}, and~\ref{ass:Psi}, are the main assumptions of this paper. Under these assumptions, the evolution equation~\eqref{eq:GF-intro} may be linear or nonlinear in $\rho$. 
The equation coincides with the Forward Kolmogorov equation~\eqref{eq:fokker-planck} if and only if condition~\eqref{cond:heat-eq-2} is satisfied,  as shown below.

\subsubsection*{Calculation for \eqref{cond:heat-eq-2}.}
Let us call $\mathscr Q[\rho]$ the right-hand side of
\eqref{eq:GF-intro}
and let us compute
\[
\langle
\mathscr Q[\rho], \varphi \rangle =\bigl \langle -\odiv \Bigl[
\rmD_\xi\calR^*\Bigl(\rho,-
\thalf \ona\upphi'\Bigl(\frac{\dd \rho}{\dd\pi}\Bigr)\Bigr)\Bigr], \varphi \bigr\rangle
\]
for every  $ \varphi \in \Bb(V) $ and $\rho \in \calM^+(V)$ with
$\rho \ll \pi$.
\nc
With $u = \frac{\dd \rho}{\dd \pi} $ we thus obtain
\begin{align}
  \notag
  \langle
  \mathscr Q[\rho], \varphi \rangle&=
  \bigl\langle \rmD_\xi\calR^*\bigl(\rho,-\thalf \ona\upphi'(u) \bigr),\ona \varphi \bigr \rangle 
  \\ & =
       \label{almost-therepre}
\frac12 \iint_{\edg} 
\big(\Psi^*\big)' \left( -\thalf \ona\upphi'(u)(x,y) \right) \ona \varphi (x,y) \bnu_\rho(\dd x,\dd y)\,.
\end{align}
Recalling the definitions~\eqref{def:nu_rho} of $\bnu_\rho$ and~\eqref{eq:184} of $\rmF$, \eqref{almost-therepre}
thus  becomes
 \begin{align}
   \notag
   \langle
  \mathscr Q[\rho], \varphi \rangle&=
  \frac 12\iint_{\edg} \bigl(\Psi^*\bigr)' \bigl( \thalf  \upphi'(u(x)) -
     \upphi'(u(y))   \bigr) \,\ona \varphi(x,y)
 \,\upalpha(u(x),u(y))  \tetapi(\dd x,\dd y)
 \\&=
 \frac12 \iint_{\edg} \rmF(u(x),u(y))
 \big(\varphi(x)-\varphi(y)\big)\,\tetapi(\dd x,\dd y)\label{eq:187}
 \\& \stackrel{(*)}{=} \EEE \iint_{\edg} \rmF(u(x),u(y))\notag
 \varphi(x)\,\tetapi(\dd x,\dd y)=
 \int_V \varphi(x) \Big(\int_V \rmF(u(x),u(y))\kappa(x,\dd
 y)\Big)\pi(\dd x)
\end{align}
where  for $(*)$ \EEE we used the symmetry of $\tetapi$ (i.e.\ the detailed-balance
condition). This calculation justifies \eqref{eq:180}.

In the linear case of \eqref{eq:fokker-planck} it is immediate to see that
\begin{align}
  \langle Q^*\rho,\varphi\rangle\notag
  = \langle \rho,Q\varphi\rangle
  &=\iint_{\edg} [\varphi(y)-\varphi(x)]\,\kappa(x,\dd y)\rho(\dd x )
  \\&= \frac12 \iint_{\edg} \ona \varphi(x,y) \bigl[\kappa(x,\dd y)\rho(\dd x) - \kappa(y,\dd x)\rho(\dd y)\bigr] \notag
\\\label{for-later}
 & =\frac12 \iint_{\edg} \ona \varphi(x,y) \bigl[u(x) -u(y)\bigr] \tetapi(\dd x,\dd y),
\end{align}
Comparing \eqref{for-later} and \eqref{eq:187}
we obtain  that $\rmF$ has to fulfill \EEE \eqref{cond:heat-eq-2}.
\nc
\subsection{Derivation of the cosh-structure from large deviations}
\label{ss:ldp-derivation}

We mentioned in the introduction that the choices 
\begin{equation}
\label{choices:ldp}
\upphi(s) = s\log s - s + 1,\qquad  \Psi^*(\xi) = 4\bigl(\cosh(\xi/2)-1\bigr),
\qquad\text{and}\qquad \upalpha(u,v) = \sqrt{uv}
\end{equation}
arise in the context of large deviations. In this section we  describe this context. Throughout this section we work under Assumptions~\ref{ass:V-and-kappa}, \ref{ass:S}, and~\ref{ass:Psi}, and since we are interested in the choices above, we will also assume~\eqref{choices:ldp}, implying that
\[
	\bnu_\rho(\dd x\,\dd y) = \sqrt{u(x)u(y)}\, \pi(\dd x)\kappa(x,\dd y), \qquad \text{if }\rho  = u\pi \ll \pi.
\]

Consider a sequence of independent and identically distributed stochastic processes $X^i$, $i=1, 2, \dots$ on $V$, each  described by the jump kernel $\kappa$, or equivalently by the generator $Q$ in~\eqref{eq:def:generator}. With probability one, a realization of each process has a countable number of jumps in the time interval $[0,\pinfty)$, and we write  $t^i_k$ for the $k^{\mathrm{th}}$ jump time of $X^i$. We can assume that $X^i$ is a c\`adl\`ag function of time.

We next define the empirical measure $\rho^n$ and the empirical flux $\bj ^n$ by
\begin{align*}
&\rho^n: [0,T]\to \calM^+(V), 
 &\rho^n_t &:= \frac1n \sum_{i=1}^n \delta_{X^i_t},
\\
&\bj^n\in \calM^+((0,T)\times \edg), &\qquad
\bj^n(\dd t\,\dd x\,\dd y) &:= \frac1n \sum_{i=1}^n \sum_{k=1}^\infty \delta_{t^i_k}(\dd t) \delta_{(X^i_{t-},X^i_{t})}(\dd x\,\dd y),
\end{align*}
where $t^i_k$ is the $k^{\mathrm{th}}$ jump time of $X^i$, and $X^i_{t-}$ is the left limit (pre-jump state) of $X^i$ at time~$t$. Equivalently, $\bj^n$ is defined by
\[
\langle \bj^n, \varphi\rangle := \frac1n \sum_{i=1}^n \sum_{k=1}^\infty \varphi\bigl(t_k^i,X^i_{t_k^i-},X^i_{t_k^i}\bigr), 
\qquad \text{for }\varphi\in \CB([0,T]\times \edg).
\]

A standard application of Sanov's theorem yields a large-deviation characterization of the pair $(\rho^n,\bj^n)$ in terms of two rate functions $I_0$ and $I$,
\[
\mathrm{Prob}\bigl((\rho^n,\bj^n)\approx (\rho,\bj )\bigr)
\sim \exp\Bigl[ -n \bigl(I_0(\rho_0) + I(\rho,\bj )\bigr)\Bigr],
\qquad\text{as }n\to\infty.
\]
The rate function $I_0$ describes the large deviations of the initial datum $\rho^n_0$; this functional is  determined by the choices of the initial data of $X^i_0$ and is independent of the stochastic process itself, and we therefore disregard it here. 

The functional $I$ characterizes the large-deviation properties of the dynamics of the pair $(\rho^n,\bj^n)$ conditional on the initial state, and has the expression
\begin{equation}
\label{eq:def:I}
I(\rho,\bj ) = \int_0^T \calF_\upeta(\bj_t | \teta_{\rho_t}^-)\, \dd t.
\end{equation}
In this expression we write $\teta_{\rho_t}^-$ for the measure $\rho_t(\dd x)\kappa(x,\dd y)\in \calM(\edg)$ (see also~\eqref{def:teta} ahead). The function $\upeta$ is the Boltzmann entropy function that we have seen above,
\[
\upeta(s) := s\log s - s+ 1,\qquad\text{for }s\geq0,
\]
and the functional $\calF_\upeta:\calM^+(E)\times \calM^+(E) \to [0,\infty]$ is given by~\eqref{def:F-F}. Even though the function $\upeta$ coincides in this section with $\upphi$, we choose a different notation  to emphasize that the roles of $\upphi$ and $\upeta$ are different: the function $\upphi$ defines the entropy of the system, which is related to the large deviations of the empirical measures $\rho^n$ in equilibrium (see~\cite{MielkePeletierRenger14}); the function $\upeta$ characterizes the large deviations of the time courses of $\rho^n$ and $\bj^n$. 

\begin{remark}
Sanov's theorem can be found in many references on large deviations (e.g.~\cite[Sec.~6.2]{DemboZeitouni98}); the derivation of the expression~\eqref{eq:def:I} is fairly
 well known  and can be found in e.g.~\cite[Eq.~(8)]{MaesNetocny08} or \cite[App.~A]{KaiserJackZimmer18}. Instead of proving~\eqref{eq:def:I} we give an interpretation of the expression~\eqref{eq:def:I} and the function $\upeta$ in terms of exponential clocks. An exponential clock with rate parameter $r$ has large-deviation behaviour given by $r\eta(\cdot/r)$ (see~\cite[Exercise 5.2.12]{DemboZeitouni98} or~\cite[Th.~1.5]{Moerters10}) in the following sense: for each $t>0$, 
\[
\mathrm{Prob}\bigl( \text{ $\approx \beta nt$ firings in time $nt$ }\bigr) 
\sim \exp \Bigl[ -n tr\,\upeta(\beta/r)\Bigr]\qquad\text{as }n\to\infty.
\]

The expression~\eqref{eq:def:I} generalizes this to a field of exponential clocks, one for each edge~$(x,y)$. In this case, the rescaled rate parameter $r$ for the clock at edge $(x,y)$ is equal to $\rho_t(\dd x)\kappa(x,\dd y)$, since it is proportional to the number of particles $n\rho_t(\dd x)$ at $x$ and to the rate of jump $\kappa(x,\dd y)$ from $x$ to $y$. The flux $n\bj_t(\dd x\,\dd y)$ is the observed number of jumps from $x$ to $y$, corresponding to firings of the clock associated with the edge $(x,y)$. In this way, the functional $I$ in~\eqref{eq:def:I} can be interpreted as  characterizing the large-deviation fluctuations in the clock-firings for each edge $(x,y)\in\edg$.
\end{remark}

The expression~\eqref{eq:def:I} leads to the functional $\mathscr L$ in~\eqref{eq:def:mathscr-L} after a symmetry reduction, which we now describe (see also~\cite[App.~A]{KaiserJackZimmer18}). Assuming that we are more interested in the fluctuation properties of $\rho$ than those of $\bj$, we might decide to minimize $I(\rho,\bj )$ over a class of fluxes $\bj $ for a fixed choice of $\rho$. Here we choose to minimize over the class of fluxes with the same skew-symmetric part, 
\[
A_\bj := \bigl\{\bj'\in \calM([0,T]\times \edg): \bj' - \symmap\bj' = \bj  - \symmap \bj \bigr\}.
\]
By the form~\eqref{eq:ct-eq-intro} of the continuity equation and the definition~\eqref{eq:def:div} of the divergence we have  $\odiv \hj = \odiv \bj $  for all $\hj\in A_j$, so that replacing $\bj $ by $\hj$ preserves the continuity equation.


\begin{Flemma}
The minimum of $I(\rho,\hj\,)$ over all $\hj\in A_j$ is achieved for the `skew-sym\-me\-tri\-za\-tion'
 $\bj^\flat = \tfrac12(\bj -\symmap \bj )$, and for 
$\bj^\flat$ \EEE
the result equals $\tfrac12\mathscr L$:
\begin{equation}
\label{eq:infI=infL}
\inf_{\hj\,\in A_j} I(\rho,\hj\,) = \inf_{\hj\in A_j} \tfrac12\mathscr L(\rho,\hj\,) = 
 \tfrac12\mathscr L(\rho,\bj^\flat).
 \EEE 
\end{equation}
Consequently, for a given curve $\rho:[0,T]\to\calM^+(V)$, 
\begin{align*}
\inf_j \Bigl\{ I(\rho,\bj ) : \partial_t \rho + \odiv \bj  = 0\Bigr\}
&= \inf_j \Bigl\{ \tfrac12\mathscr L(\rho,\bj ) : \partial_t \rho + \odiv \bj  = 0\Bigr\},
\\
\noalign{\noindent and in this final expression the flux can be assumed to be   skew-symmetric:} \EEE
%
&=\inf_j \Bigl\{ \tfrac12\mathscr L \bigl(\rho,\bj \bigr): \partial_t \rho + \odiv \bj  = 0 \text{ and } \symmap \bj  = - \bj \Bigr\}.
\end{align*}
\end{Flemma}
This implies that the two functionals $I$ and $\mathscr L$ can be considered to be the same, if one  is only interested in $\rho$, not in $\bj $. By the Contraction Principle (e.g.~\cite[Sec.~4.2.1]{DemboZeitouni98}) the functional $\rho \mapsto \inf_j I(\rho,\bj  ) = \inf_j \tfrac12\mathscr L(\rho,\bj )$ also can be viewed as the large-deviation rate function of the sequence of empirical measures $\rho^n$.

The above lemma is only formal because we have not  given a rigorous definition of the functional~$\mathscr L$. While it would be possible to do so, using the construction of Lemma~\ref{l:lsc-general} and the arguments of the proof below, actually the rest of this paper deals with this question in a more detailed  manner. In addition, at this stage this lemma only serves to explain why we consider this specific class of functionals~$\mathscr L$. Therefore  here we only give heuristic arguments. 

\begin{proof}
We assume throughout this (formal) proof that all measures are absolutely continuous,  strictly positive, and finite where necessary. 
Note that  writing $\rho_t = u_t \pi$ we have $\teta_{\rho_t}^-(\dd x\,\dd y) = u_t(x) \teta(\dd x\,\dd y)$, and using~\eqref{choices:ldp} we  therefore have 
\begin{gather*}
\sqrt{\teta_{\rho_t}^- \, \symmap\teta_{\rho_t}^-}\; (\dd x\,\dd y) = \sqrt{u_t(x)u_t(y)}\;\tetapi(\dd x\,\dd y) = \bnu_{\rho_t}(\dd x\,\dd y), \qquad\text{and}\\
 \log \frac{\dd \symmap \teta_{\rho_t}^-}{\dd \teta_{\rho_t}^-} (x,y)
= 
\log \frac {u_t(y)}{u_t(x)}
= 
\ona \upphi'(u_t)(x,y).
\end{gather*}
For the length of this proof we write $\hat \upeta$ for the perspective function corresponding to $\upeta$ (see
\RICKYNEW \eqref{eq:76} in \EEE
 Lemma~\ref{l:lsc-general})
\[
\RICKYNEW \hat{\upeta}(a,b) \EEE
 := \begin{cases}
	a\log \dfrac ab - a + b & \text{if $a,b>0$,}\\
	0 &\text{if }a = 0,\\
	+\infty &\text{if $a>0$, $b=0$.}
\end{cases}
\]
We now rewrite  $\inf_{\hj\,\in A_j} I(\rho,\hj\,)$ as
\begin{align*}
\inf_{\hj \,\in A_j} &\int_0^T \iint_\edg   \upeta\biggl( \frac{\dd \hj_t}{\dd \teta_{\rho_t}^-}\biggr)\, \dd \teta_{\rho_t}^- \dd t
= \inf_{\hj \,\in A_j} \int_0^T \iint_\edg  u_t\,  \upeta\biggl( \frac1 {u_t} \frac{\dd \hj_t}{\dd  \teta}\biggr)  \,\dd \teta\, \dd t\\
&=\inf_{\hj  =  \zeta \EEE \teta \,\in A_j} \int_0^T \iint_\edg \hat \upeta\bigl(  \zeta_t(x,y), \EEE  u_t(x)\bigr)  \teta(\dd x,\dd y)\,\dd t \\
  &=
      \frac12 \nc
    \inf_{\hj  =  \zeta \EEE \tetapi \,\in A_j} \int_0^T \iint_\edg 
  \Bigl\{\hat \upeta\bigl(  \zeta_t(x,y), \EEE u_t(x)\bigr) 
    + \hat \upeta\bigl(  \zeta_t(y,x), \EEE u_t(y)\bigr)  \Bigr\}\,\teta(\dd x,\dd y)\,\dd t .
\end{align*}
Since 
 $\zeta(x,y) - \zeta(y,x) = 
 \dd (\hj - \symmap \hj\,)/\dd\tetapi$ \EEE  is constrained in $A_j$, we follow the expression inside the second integral and set  
\[
\psi:\R\times[0,\pinfty)^2 \to [0,\pinfty], \qquad
\psi(s\,; c,d) := \inf_{a,b\geq0} \Bigl\{
\bigl[\hat \upeta(a,c) + \hat \upeta(b,d)\bigr] : a-b = 2s\Bigr\},
\]
for which a calculation gives the explicit formula (for $c,d>0$)
\[
\psi(s\,; c,d) = \frac{\sqrt{cd}}2\; \biggl\{\Psi\biggl( \frac { 2s}{\sqrt{cd}}\biggr) +
\Psi^*\biggl( -
\log \frac dc\biggr) \biggr\} + s \;
\log \frac dc,
\]
in terms of the function $\Psi^*(\xi) = 4\bigl(\cosh \xi/2  - 1\bigr)$ and its Legendre dual $\Psi$.  This minimization corresponds to minimizing over all fluxes for which the `net flux' $\bj  - \symmap \bj =  2 \tj $ \EEE is the same; see e.g.~\cite{Renger18,KaiserJackZimmer18} for discussions.

 Let 
$\cw(x,y) :=
 (w(x,y) - w(y,x)) = 
 \frac{\dd(2\tj)}{\dd \tetapi}$ \EEE and $\upalpha_t := \upalpha_t(x,y) = \sqrt{u_t(x)u_t(y)}$.  We find
\begin{align}
\inf_{\hj\, \in A_j} &\int_0^T \calF_\upeta\bigl( \hj_t | \teta_{\rho_t}^-\bigr) \,\dd t
\notag
\\
  &=  \frac12\nc
    \int_0^T \iint_\edg \psi\bigl(
     \tfrac12 \cw_t(x,y) \EEE \,;\, u_t(x),u_t(y)\bigr)\, \tetapi(\dd x,\dd y) \dd t
     \notag\\
&=\frac 12 \int_0^T \iint_\edg \biggl\{ \frac{\upalpha_t}2 \Psi\biggl(\frac{ 	\cw_t}{\upalpha_t}\biggr)
  + \frac{\upalpha_t}2 \Psi^*\biggl(- \ona \upphi'(u_t)\biggr) 
  +  \frac12  \cw_t\, \EEE \ona \upphi'(u_t)  \biggr\}\,\dd\teta\dd t
\notag\\
&= \frac12 \int_0^T \iint_\edg \frac12 \biggl\{ \Psi\biggl(\frac{{2\dd\tj_t} \EEE}{\dd \bnu_{\rho_t}}\biggr)
  + \Psi^*\biggl(- \ona \upphi'(u_t)\biggr)\biggr\}\,\dd\bnu_{\rho_t} \, \dd t
  + \frac12 \calS(\rho_T) - \frac12 \calS(\rho_0).
\label{eq:formal-motivation-L}
\end{align}
In the last identity we used the fact that since   $\odiv \tj_t = -\partial_t \rho_t$, \EEE formally we have
\[
 \int_0^T \iint_\edg    \frac12  \cw_t\, \ona \upphi'(u_t)  \,\dd\teta\dd t = \int_0^T \iint_\edg   \ona \upphi'(u_t) \,    \dd\tj_t \,\dd t \EEE
= \int_0^T \langle \upphi'(u_t) ,\partial_t \rho_t \rangle\,\dd t
= \calS(\rho_T)- \calS(\rho_0).
\]
The expression on the right-hand side of~\eqref{eq:formal-motivation-L} is one half times the functional $\mathscr L$ defined in~\eqref{eq:def:mathscr-L} (see also~\eqref{ineq:deriv-GF}). This proves that 
\[
\inf_{\hj\,\in A_j} I(\rho,\hj\,) =  \frac12 \mathscr L\bigl(\rho,  \tj \EEE\,\bigr).
\]
From convexity of $\Psi$ and symmetry of $\bnu_\rho$ we deduce that   $\mathscr L(\rho,  \tj) \EEE
\leq \mathscr L(\rho,\bj )$ for any $\bj $; see Remark~\ref{rem:skew-symmetric}. The identity   $\mathscr L\bigl(\rho,\tj\,\bigr) = \inf_{\hj\,\in A_j} \mathscr L (\rho,\hj\,)$  \EEE then follows immediately; this proves~\eqref{eq:infI=infL}. 

To prove the second part of the Lemma, we write
\begin{align*}
\inf_j \Bigl\{ I(\rho,\bj ) : \partial_t \rho + \odiv \bj  = 0\Bigr\}
&= \inf_j \Bigl\{ \Bigl[\,\inf _{\hj\in A_j} I(\rho,\hj\,) \Bigr]: \partial_t \rho + \odiv \bj  = 0\Bigr\},\\
&= \inf_j \Bigl\{ \Bigl[\,\inf _{\hj\in A_j} \tfrac12 \mathscr L(\rho,\hj\,) \Bigr]: \partial_t \rho + \odiv \bj  = 0\Bigr\},\\
&= \inf_j \Bigl\{ \tfrac12\mathscr L(\rho,  \tj \EEE ) : \partial_t \rho + \odiv \bj  = 0\Bigr\},\\
&= \inf_j \Bigl\{ \tfrac12\mathscr L(\rho, \tj \EEE) : \partial_t \rho +  \odiv \tj \EEE = 0\Bigr\}.
\end{align*}
This concludes the proof. 
\end{proof}


\section{Curves in \texorpdfstring{$\calM^+(V)$}{M+(V)}}

A major challenge in any rigorous treatment of an equation such as~\eqref{eq:GGF-intro-intro} is finding a way   to deal with the time derivative. The Ambrosio-Gigli-Savar\'e framework for metric-space gradient systems, for instance,  is organized around absolutely continuous curves. These are a natural choice  because on the one hand this class admits a `metric velocity' that generalizes the time derivative, while on the other hand solutions are automatically absolutely continuous by the superlinear growth of the dissipation potential. 

For the systems of this paper, a similar role is played by curves such that the `action' $\int \calR\,\dd t$ is finite; we show below that the superlinearity of $\calR(\rho,\bj )$ in $\bj $  leads to similarly beneficial properties. In order to exploit this aspect, however, a number of intermediate steps need to be taken:
\begin{enumerate}[label=(\alph*)]
\item \label{intr:curves:1}
We define the class $\CE0T$ of solutions  $(\rho,\bj )$ of the continuity equation~\eqref{eq:ct-eq-intro} (Definition~\ref{def-CE}).
\item For such solutions, $t\mapsto \rho_t$ is continuous in
  the total variation distance (Corollary~\ref{c:narrow-ct}).
\item We give a rigorous definition of the functional $\calR$ (Definition~\ref{def:R-rigorous}), and describe its behaviour on absolutely continuous and singular parts of $(\rho,\bj)$ (Lemma~\ref{l:alt-char-R} and Theorem~\ref{thm:confinement-singular-part}).
\item If the action functional $\int \calR$ is finite along a solution
    $(\rho,\bj)$ of the continuity equation in $[0,T]$, then
    the property that $\rho_t$ is absolutely continuous with respect to ~$\pi$
    at some time $t\in [0,T]$ propagates to all the interval $[0,T]$ (Corollary~\ref{cor:propagation-AC}).
\item We prove a chain rule for the derivative of convex entropies along curves of finite $\calR$-action (Theorem~\ref{th:chain-rule-bound}) and derive an estimate involving $\calR$ and a Fisher-information-like term (Corollary~\ref{th:chain-rule-bound2}).
\item\label{intr:curves:5}
    If the action $\int\calR$ is uniformly bounded along a sequence $(\rho^n,\bj^n)\in\CE0T$, then the sequence is compact in an appropriate sense (Proposition~\ref{prop:compactness}).
\end{enumerate}

 Once properties~\ref{intr:curves:1}--\ref{intr:curves:5} have been established,   the next step is to consider finite-action curves that also  connect two given values $\mu,\nu$, leading to the definition of the Dynamical-Variational Transport (DVT) cost 
\begin{equation}
\label{def-psi-rig-intro-section}
\DVT \tau\mu\nu : = \inf\left\{ \int_0^\tau \calR(\rho_t,\bj_t)\, \dd t \, : \,  (\rho,\bj ) \in  \CE 0\tau, \ \rho_0 = \mu, \ \rho_\tau = \nu   \right\}\,. 
\end{equation}
This definition is in the spirit of the celebrated Benamou-Brenier formula for the Wasserstein distance \cite{Benamou-Brenier}, generalized to a broader family of 
transport distances \cite{DolbeaultNazaretSavare09} and to jump processes \cite{Maas11,Erbar14}.
 However, a major difference with those constructions 
 is that $\DVTn$ also depends on the time variable $\tau$ and that
 $\DVT\tau\cdot\cdot$ is not a (power of a) distance,
 since $\Psi$ is not, in general, positively homogeneous of any order.
 Indeed, \EEE
 when $\calR$ is  $p$-homogeneous in $\bj $, for $p\in (1,\pinfty)$,
 we have  (see also the discussion at the beginning of Sec.\ \ref{ss:MM}) 
\begin{equation}
\label{eq:DVT=JKO}
\DVT \tau\mu\nu = \frac1{\tau^{p-1}} \DVT 1\mu\nu=
\frac1{ p \EEE \tau^{p-1}}d_{\calR}^p(\mu,\nu),
\end{equation}
where $d_\calR$ is an extended distance and
\nc
is a central object in the  usual Minimizing-Movement  construction. 
In Section~\ref{s:MM}, the DVT cost~$\DVTn$ will replace 
 the rescaled $p$-power of the distance  
and play a similar role  for the Minimizing-Movement  approach. \EEE

For the rigorous construction of $\DVTn$,
\begin{enumerate}[label=(\alph*),resume]
\item we show that minimizers of~\eqref{def-psi-rig-intro-section} exist (Corollary~\ref{c:exist-minimizers});
\item we establish properties of $\DVTn$ that generalize those of the metric-space version~\eqref{eq:DVT=JKO} (Theorem~\ref{thm:props-cost}). 
\end{enumerate}
Finally, 
\begin{enumerate}[label=(\alph*),resume]
\item we close the loop by showing that from a given functional $\DVTn$ integrals of the form $\int_a^b\calR$ can be reconstructed (Proposition~\ref{t:R=R}).
\end{enumerate}
 Throughout this section we adopt
 \textbf{Assumptions~\ref{ass:V-and-kappa} and  \ref{ass:Psi}}.

\subsection{The continuity equation}
\label{sec:ct-eq}
We now
  introduce the formulation of the continuity equation we will work
  with. Hereafter, for a given function $\mu :I \to \calM(V)$, or
  $\mu : I  \to \calM(\edg)$, with $I=[a,b]\subset\R$,
  we shall often
  write $\mu_t$ in place of $\mu(t)$ for  a given $t\in I$ and denote
  the time-dependent function $\mu $ by  $(\mu_t)_{t\in I}$.
     We will write $\Lebone$ for the Lebesgue measure on $I$.
  The following definition mimics  those  given in
  \cite[Sec.~8.1]{AmbrosioGigliSavare08} and \cite[Def.~4.2]{DNS09}.

  \begin{definition}[Solutions $(\rho,\bj)$ of the continuity equation]
\label{def-CE}
Let $I=[a,b]$ be a closed interval of $\R$. We denote by 
$\CEI I$ the set of pairs $(\rho,\bj)$
given by 
\begin{itemize}
\item  a family of time-dependent measures $\rho=(\rho_t)_{t\in I} \subset \calM^+(V)$,
and
\item a measurable family $(\bj_t)_{t\in I} \subset
\calM(\edg)$
with $\int_0^T |\bj_t|(\edg)\,\dd t <\pinfty$,
satisfying the  continuity equation 
\begin{equation}
\label{eq:ct-eq-def}
\dot\rho + \odiv \bj=0 \quad\text{ in } I\times V,
\end{equation}
in the following sense:
\begin{equation}
    \label{2ndfundthm}
    \int_V \varphi\,\dd \rho_{t_2}-\int_V \varphi\,\dd\rho_{t_1}=
    \iint_{J\times \edg} \dnabla \varphi \,\dd \bj_\lambda
    \quad\text{for all $\varphi\in \Bb(V)$, $J=[t_1,t_2]\subset I$}.
  \end{equation}
  where   
  $\bj_\lambda(\dd t,\dd x,\dd y):=\Lebone(\dd
   t)\bj_t(\dd x,\dd y)$.
\end{itemize}
 Given $\rho_0,\, \rho_1 \in \calM^+(V)$, we will use the notation
\[
\CEIP I{\rho_0}{\rho_1} : = \bigl\{(\rho,\bj ) \in \CEI I\, : \ \rho(a)=\rho_0, \ \rho(b) = \rho_1\bigr\}\,.
\]
\end{definition}

\begin{remark}
  \label{rem:expand}
  The requirement~\eqref{2ndfundthm} shows in particular that  $t\mapsto \rho_t$ is continuous with respect to the total variation metric.  Choosing $\varphi\equiv1$ in~\eqref{2ndfundthm}, one immediately finds that
  \begin{equation}
    \label{eq:91}
    \text{the total mass }\rho_t(V)\text{ is constant in $I$}.
  \end{equation}
  By the disintegration theorem,
  it is equivalent to
  assign the measurable family $(\bj_t)_{t\in I}$ in $\calM(\edg)$ or
  the measure $\bj_\lambda$ in $\calM(I\times \edg)$.
  \end{remark}
  
We can in fact prove a more refined property. The proof of the Corollary below is postponed to
Appendix \ref{appendix:proofs}. 
\begin{cor}
\label{c:narrow-ct}
If $(\rho,\bj )\in\CE0T$,
then  there exist a \emph{common} dominating
  measure $\gamma\in \calM^+(V)$  (i.e., $\rho_t \ll \gamma$ for all $t\in [a,b]$), \EEE and an absolutely continuous map
  $\tilde u:[a,b]\to L^1(V,\gamma)$ such that
  $\rho_t=\tilde u_t\gamma\ll \gamma$ for every $t\in [a,b]$.
\end{cor}

The interpretation of the continuity equation in Definition \ref{def-CE}---in duality with all bounded measurable functions---is quite strong, and in particular much stronger than the more common continuity in duality with \emph{continuous} and bounded functions. However, 
this continuity \RICKYNEW equation \EEE can be recovered starting from a much weaker formulation.
The following result illustrates this; it is a translation of \cite[Lemma 8.1.2]{AmbrosioGigliSavare08} (cf.\ also \cite[Lemma
4.1]{DNS09}) to the present setting. The proof adapts the argument for  \cite[Lemma 8.1.2]{AmbrosioGigliSavare08} and is
given in  Appendix~\ref{appendix:proofs}.

\begin{lemma}[Continuous representative]
  \label{l:cont-repr}
Let 
$(\rho_t)_{t\in I} \subset \calM^+(V)$ and $(\bj_t)_{t\in I}$ be measurable families
that are integrable with respect to ~$\Lebone$
and let $\tau$ be any separable and metrizable
topology inducing
$\frB$. If 
  \begin{equation}
    -\int_0^T \eta'(t) \left( \int_V \zeta(x) \rho_t (\dd x ) \right) \dd
    t = \int_0^T \eta(t)\Big(\iint_\edg \dnabla\zeta(x,y)\,  \bj_t(\dd
    x\,\dd y)\Big)\,\dd t
    \,,\label{eq:90}
  \end{equation}
  holds for every $\eta \in
  \mathrm{C}_\mathrm{c}^\infty((a,b))$ and $\zeta \in \Cb
  (V,\tau)$, then there exists a unique curve $I \ni t
  \mapsto \tilde{\rho}_t \in \calM^+ (V)$ such that $\tilde{\rho}_t =
  \rho_t$ for $\Lebone$-a.e. $t\in I$. The curve $\tilde\rho$ is continuous in the total-variation norm with estimate
  \begin{equation}
  \label{est:ct-eq-TV}	
    \|\tilde \rho_{t_2}-\tilde \rho_{t_1}\|_{TV} \leq 2
  \int_{t_1}^{t_2} |\bj_t|(\edg)\, \dd t \qquad
  \text{ for  all }  t_1 \leq t_2,
\end{equation}
  and satisfies
  \begin{equation}
    \label{maybe-useful}
    \int_V \varphi(t_2,\cdot) \,\dd\tilde\rho_{t_2} - \int_V \varphi(t_1,\cdot) \,\dd\tilde\rho_{t_1}
    = \int_{t_1}^{t_2} \int_V \partial_t \varphi \,\dd\tilde\rho_t\,\dd t
    +
    \int_{J\times \edg} \dnabla \varphi \,\dd \bj_\Lebone
  \end{equation}
  for all $\varphi \in \mathrm{C}^1(I;\Bb(V))$ and $J=[t_1,t_2]\subset
  T$.  
\end{lemma}

\begin{remark}
  In \eqref{2ndfundthm} we can always replace $\bj $ with the positive measure \RICKYNEW 
   $\bj ^+:=(\bj -\symmap \bj )_+ = (2\tj)_+$, since $\odiv  \bj =  \odiv \bj^+$  \EEE (see Lemma~\ref{le:A1}); therefore we can assume without loss of generality that $\bj $ is a positive
  measure.
\end{remark}
\nc
As another immediate consequence of \eqref{2ndfundthm}, the concatenation of two solutions of the continuity equation is again a solution; the result below also contains a statement about time rescaling of the solutions, whose proof follows from trivially adapting that of \cite[Lemma 8.1.3]{AmbrosioGigliSavare08} and is thus omitted. 

\begin{lemma}[Concatenation and time rescaling]
\label{l:concatenation&rescaling}
\begin{enumerate}
\item Let $(\rho^i,\bj^i) \in \CE 0{T_i}$, $i=1,2$, with $\rho_{T_1}^1 = \rho_0^2$. Define $(\rho_t,\bj_t)_{t\in [0,T_{1}+T_2]}$ by 
\[
\rho_t: = \begin{cases}
\rho_t^1 & \text{ if } t \in [0,T_1],
\\
\rho_{t-T_1}^2 & \text{ if } t \in [T_1,T_1+T_2],
\end{cases}
\qquad \qquad 
\bj_t: = \begin{cases}
\bj_t^1 & \text{ if } t \in [0,T_1],
\\
\bj_{t-T_1}^2 & \text{ if } t \in [T_1,T_1+T_2]\,.
\end{cases}
\]
Then, $(\rho,\bj ) \in \CE 0{T_1+T_2}$.
\item
Let $\mathsf{t} : [0,\hat{T}] \to [0,T]$ be strictly increasing and absolutely continuous, with inverse $\mathsf{s}: [0,T]\to [0,\hat{T}]$. Then, $(\rho, \bj ) \in \CE 0T$ if and only if 
$\hat \rho: = \rho \circ \mathsf{t}$ and $\hat \bj : = \mathsf{t}' (\bj  {\circ} \mathsf{t})$  fulfill $(\hat \rho, \hat \bj ) \in \CE 0{\hat T}$. 
\end{enumerate}
\end{lemma}

\subsection{Definition of the dissipation potential \texorpdfstring{$\calR$}R}
\label{ss:def-R}
In this section we give a rigorous definition of the dissipation potential $\calR$, following the formal descriptions above. In the special case when $\rho$ and $\bj$ are absolutely continuous,  i.e.
\begin{equation}
\rho=u\pi\ll\pi
\qquad\text{and}\qquad
2\bj = w\tetapi\ll\tetapi,
\end{equation}
we set
 \begin{equation}
   \label{concentration-set}
   \edg': =  \{ (x,y) \in \edg\, :  \upalpha(u(x),u(y))>0 \},
 \end{equation}
and in this case we can define the functional $\calR$
by the direct formula
\begin{equation}
  \label{eq:21}
  \calR(\rho,\bj )=
  \begin{cases}
    \displaystyle
    \frac12\int_{E'}
    \Psi\Bigl(\frac{w(x,y)}{\upalpha(u(x),u(y))}\Bigr)\upalpha(u(x),u(y))\,\tetapi(\dd
    x,\dd y) &\text{if }|\bj |(\edg\setminus E')
    =0,\\
    \pinfty&\text{if }|\bj |(\edg\setminus E')>0. 
  \end{cases}
\end{equation}
Recalling the definition of the perspective function $\hat\Psi$ \eqref{eq:76},
we can also write \eqref{eq:21} in the equivalent and more compact form
\begin{equation}
  \label{eq:92}
  \calR(\rho,\bj )=
  \frac12\iint_\edg \hat\Psi\big(w(x,y), \upalpha(u(x),u(y)) \big)\, \tetapi(\dd
  x,\dd y),\quad
  2\bj=w\tetapi\,.
\end{equation}
so that it is natural to introduce the function
$\Upsilon : [0,\pinfty)\times[0,\pinfty)\times\R\to[0,\pinfty] $,
	 \begin{equation}
	 \label{Upsilon}
         \Upsilon (u,v,w) :=
         \hat\Psi(w,\upalpha(u,v)),
       \end{equation}
      observing that
      \begin{equation}
        \label{eq:22}
        \calR(\rho,\bj )=
        \frac12\iint_\edg
        \Upsilon(u(x),u(y),w(x,y))\,\tetapi(\dd x,\dd y)\quad\text{for }
        2\bj=w\tetapi.
      \end{equation}
      \nc
%
%
\begin{lemma}\label{lem:Upsilon-properties}
  The function $\Upsilon:[0,\pinfty)\times[0,\pinfty)\times\R\to[0,\pinfty]$ defined above
  is convex and lower semicontinuous, with recession functional
  \begin{equation}
    \label{eq:23}
    \Upsilon^\infty(u,v,w)=
    \hat\Psi(w, \upalpha^\infty(u,v))=
	    \begin{cases}
	        \displaystyle
	         \Psi\left( \frac{w}{\upalpha^\infty(u,v)}\right) \upalpha^\infty(u,v) &
	         \text{
	           if $\upalpha^\infty(u,v)>0$}
	         \\
	         0 &\text{ if $w=0$}
	         \\
	         \pinfty & \text{ if $w\ne 0$ and $\upalpha^\infty(u,v)=0$.}
	         \end{cases}
  \end{equation}
  For any $u,v\in [0,\infty)$ with $\upalpha^\infty(u,v)>0$, the map $w\mapsto \Upsilon(u,v,w)$ is strictly convex.
  
  If $\upalpha$ is positively 1-homogeneous then $\Upsilon$ is positively 1-homogeneous as well.
\end{lemma}
\begin{proof}
	Note that $\Upsilon$ may be equivalently represented in the form
	\begin{equation}
	\label{eq:dual-formulation-Upsilon}
	 \Upsilon(u,v,w) = \sup_{\xi\in\R} \bigl\{\xi w - \upalpha(u,v)\Psi^*(\xi)\bigr\} =: \sup_{\xi\in\R} f_\xi(u,v,w)\,.
	\end{equation}
	The convexity of $f_\xi$ for each $\xi\in\R$ readily follows
        from its linearity in $w$ and the convexity of $-\upalpha$ in
        $(u,v)$. Therefore, $\Upsilon$ is convex
        and lower semicontinuous as the pointwise supremum of a family of convex
        continuous functions.

       The characterization~\eqref{eq:23} of $\Upsilon^\infty$ follows from observing that $\Upsilon(0,0,0)=\hat\Psi(0,0)=0$ and using the $1$-homogeneity of $\hat\Psi$:
        \begin{align*}
          \lim_{t\to\pinfty}
          t^{-1}\Upsilon(tu,tv,tw)
          &=
            \lim_{t\to\pinfty} t^{-1}
            \hat\Psi\Big(tw,\upalpha( tu,tv)\Big)
            =\lim_{t\to\pinfty}
            \hat\Psi\Big(w,t^{-1}\upalpha( tu,tv)\Big)
          \\&=
           \hat\Psi\Big(w,\upalpha^\infty(u,v)\Big)\,,
        \end{align*}
        where the last equality follows from the continuity of $r\mapsto \hat\Psi(w,r)$ for all $w\in \R$.
        
        The strict convexity of $w\mapsto \Upsilon(u,v,w)$ for any $u,v\in [0,\infty)$ with $\upalpha^\infty(u,v)>0$ follows directly from the strict convexity of $\Psi$ (cf. Lemma~\ref{l:props:Psi}).
\end{proof}

\bigskip

The choice~\eqref{eq:22} provides a rigorous definition of $\calR$ for couples of measures $(\rho,\bj)$ that are absolutely continuous with respect to $\pi$ and $\teta$.
In order to extend $\calR$ to  pairs $(\rho,\bj)$ that are not absolutely continuous, it is
useful to interpret the measure
\begin{equation}
\bnu_\rho(\dd x,\dd
y):=\upalpha(u(x),u(y))\tetapi(\dd x,\dd y)
\label{eq:26}
\end{equation}
in the integral of \eqref{eq:21}
in terms of a suitable concave transformation
as in \eqref{eq:6} of two couplings
generated by $\rho$. We therefore introduce the measures
\begin{equation}
\label{def:teta}
\begin{aligned}
  \teta_{\rho}^-(\dd x\,\dd y) :=
  \rho(\dd x)\kappa(x,\dd y),\qquad
  \teta_{\rho}^+(\dd x\,\dd y) :=
  \rho(\dd y)\kappa(y,\dd x)=
  s_{\#}\teta_\rho^-(\dd x\,\dd y),
\end{aligned}
\end{equation}
observing that
\begin{equation}
  \label{eq:24}
  \rho=u\pi\ll\pi\quad\Longrightarrow\quad
  \teta^\pm_\rho\ll\tetapi,\qquad
  \frac{\dd \teta_\rho^-}{\dd \tetapi}(x,y) = u(x), \quad \frac{\dd \teta_\rho^+}{\dd \tetapi}(x,y) = u(y).
\end{equation}
We thus obtain that \eqref{eq:26}, \eqref{eq:21} and
\eqref{eq:22} can be equivalently written as
\begin{equation}
  \label{eq:27}
  \bnu_\rho=\Aalpha[\teta^-_\rho,\teta^+_\rho|\tetapi],\quad
  \calR(\rho,\bj )=\frac12\calF_\Psi(2\bj |\bnu_\rho)\,,
\end{equation}
where $\Aalpha[\teta^-_\rho,\teta^+_\rho|\tetapi]$ stands for $\Aalpha[(\teta^-_\rho,\teta^+_\rho)|\tetapi]$, and the functional $ \calF_\psi(\cdot | \cdot)$ is from 
\eqref{def:F-F}, and also
\begin{equation}
  \label{eq:28}
  \calR(\rho,\bj )=\frac12\calF_\Upsilon(\teta^-_\rho,\teta^+_\rho,2\bj |\tetapi)\,,
\end{equation}
again writing for shorter notation $\calF_\Upsilon(\teta^-_\rho,\teta^+_\rho,2\bj |\tetapi)$ in place of $\calF_\Upsilon((\teta^-_\rho,\teta^+_\rho,2\bj) |\tetapi)$.

Therefore we can use the same expressions \eqref{eq:27} and \eqref{eq:28}
to extend the functional $\calR$ to measures $\rho$ and $\bj $
that need  not be absolutely
continuous with respect to ~$\pi$ and $\tetapi$; the next lemma shows that they provide equivalent characterizations.
We introduce the functions $u^\pm:\edg\to\R$, adopting the notation
\begin{multline}
	  \label{eq:93}
  u^-:=u\circ \sfx\quad \text{and}\quad    u^+:=u\circ \sfy,\\
  \text{or equivalently} \quad 
    u^-(x,y):=u(x),\quad
u^+(x,y):=u(y).
\end{multline}
(Recall that $\sfx$ and $\sfy$ denote the coordinate maps from $\edg$ to $V$).

\begin{lemma}
\label{l:4.8}
  For every $\rho\in \calM^+(V)$ and $\bj \in \calM(\edg)$ we have
  \begin{equation}
    \label{eq:29}
    \calF_\Upsilon(\teta^-_\rho,\teta^+_\rho, 2\bj \EEE |\tetapi)  =\calF_\Psi( 2\bj \EEE |\bnu_\rho).
  \end{equation}
  If $\rho=\rho^a+\rho^\perp$ and $\bj=\bj^a+\bj^\perp$
  are the Lebesgue decompositions of $\rho$ and $\bj$ with respect to ~$\pi$ and
  $\tetapi$, respectively, we have
  \begin{equation}
    \label{eq:94}
    \calF_\Upsilon(\teta^-_\rho,\teta^+_\rho, 2\bj \EEE |\tetapi)=
    \calF_\Upsilon(\teta^-_{\rho^a},\teta^+_{\rho^a}, 2\bj^a \EEE |\tetapi)+
    \calF_{\Upsilon^\infty}(\teta^-_{\rho^\perp},\teta^+_{\rho^\perp}, 2\bj^\perp). \EEE
  \end{equation}
\end{lemma}
\begin{proof}
  Let us consider the Lebesgue decomposition $\rho=\rho^a+\rho^\perp$,
  $\rho^a=u\pi$, 
  and a corresponding partition of $V$ in two disjoint Borel sets
  $R,P$ such that $\rho^a=\rho\mres R$, $\rho^\perp=\rho\mres P$ and
  $\pi(P)=0$,
  which yields
  \begin{equation}
    \label{eq:30}
    \teta^\pm_\rho=\teta^\pm_{\rho^a}+\teta^\pm_{\rho^\perp},\quad
    \teta^\pm_{\rho^a}\ll\tetapi,\quad
    \teta^-_{\rho^\perp}:=\teta^-_\rho\mres{P\times V},
    \quad
    \teta^+_{\rho^\perp}:=\teta^+_\rho\mres{V\times P}.
  \end{equation}
  Since $\tetapi(P\times V)=\tetapi(V\times P) \le   \|\kappa_V\|_\infty \EEE  \pi(P)=0$,
  $\teta^\pm_{\rho^\perp}$ are singular
  with respect to ~$\tetapi$.

  Let us also consider the Lebesgue decomposition
  $\bj=\bj^a+\bj^\perp$
  of $\bj$ with respect to ~$\tetapi$.
  We can select a measure $\bsigma\in \calM^+(\edg)$ such that 
  $\teta^\pm_{\rho^\perp}=z^\pm\bsigma\ll\bsigma$,
  $\bj^\perp\ll\bsigma$ and $\bsigma\perp\tetapi$,
  obtaining
  \begin{equation}
    \label{eq:34}
    \begin{aligned}
      \bnu_\rho=\Aalpha[\teta_\rho^-,\teta_\rho^+|\tetapi]=
      \bnu_\rho^1+\bnu_\rho^2,
      \quad 
      \bnu_\rho^1:=\upalpha(u^-,u^+)\tetapi,\quad
      \bnu_\rho^2:=\upalpha^\infty(z^-,z^+)\bsigma.
    \end{aligned}
  \end{equation}
  Since $\bj \ll \tetapi+\bsigma$,
  we can decompose 
  \begin{equation}
  \label{decomp-DF}
   2\bj =w\tetapi+w'\bsigma, \EEE
  \end{equation}
  and by the additivity property \eqref{eq:81}
  we obtain
  \begin{equation}
  \label{heartsuit}
  \begin{aligned}
    \calF_\Psi
    &(   2\bj \EEE |\bnu_\rho)
      =
      \calF_{\hat \Psi}(   2\bj, \EEE\bnu_\rho)=
      \calF_{\hat\Psi}(w\teta,\bnu_\rho^1)+
      \calF_{\hat\Psi}(w'\bsigma,\bnu_\rho^2)
      \\&\stackrel{(*)}=
      \iint_\edg \Upsilon(u(x),u(y),w(x,y))\,\tetapi(\dd x,\dd y)+
      \iint_\edg
    \Upsilon^\infty(z^-(x,y),z^+(x,y),w'(x,y))\,\bsigma(\dd x,\dd y)
    \\&=
    \calF_\Upsilon(\teta^-_{\rho^a},\teta^+_{\rho^a},    2\bj^a |\tetapi)+ \EEE
    \calF_{\Upsilon^\infty}(\teta^-_{\rho^\perp},\teta^+_{\rho^\perp},   2\bj^\perp) \EEE
    =\calF_\Upsilon(\teta^-_\rho,\teta^+_\rho,    2\bj |\tetapi). \EEE
\end{aligned}
\end{equation}

Indeed, identity (*) follows from the fact that, since $\hat{\Psi}$ is $1$-homogeneous,
\[
 \calF_{\hat\Psi}(w\teta,\bnu_\rho^1) = \iint_{\GGG \edg} \GGG
 \hat{\Psi}
 \left( \frac{\dd(w\teta,\bnu_\rho^1)}{\dd \gamma}\right)  \dd \gamma
\]
for every $\gamma \in \calM^+(\edg)$ such that $w\teta \ll \gamma$ and $\bnu_\rho^1 \ll \gamma$, cf.\ \eqref{eq:78}. Then, it suffices to observe that  $w\teta \ll \teta$ and $\bnu_\rho^1 
\ll \teta$ with $\frac{\dd \bnu_\rho^1}{\dd \teta} = \upalpha(u^-,u^+)$. The same argument applies to $      \calF_{\hat\Psi}(w'\bsigma,\bnu_\rho^2)$, 
cf.\ also   Lemma~\ref{l:lsc-general}(3). \EEE
%
\end{proof}
\begin{definition}
	\label{def:R-rigorous}
	The \textit{dissipation potential} $\calR: \calM^+(V)\times\calM(\edg) \to [0,\pinfty]$  is defined by
	\begin{equation}
	\label{def:action}
	\calR(\rho,\bj ) :=
          \frac12 \calF_\Upsilon(\teta^-_\rho,\teta^+_\rho,2\bj
          |\tetapi)  =
          \frac12 \calF_\Psi(2\bj |\bnu_\rho).
	\end{equation}
        where $\teta_{\rho}^\pm$ are defined by \eqref{def:teta}.
        If $\upalpha$ is $1$-homogeneous, then
        $\calR(\rho,\bj)$ is independent of $\tetapi$.
 \end{definition}


\begin{lemma}
\label{l:alt-char-R}
Let $\rho=\rho^a+\rho^\perp\in \calM^+(V)$ and $\bj =\bj^a+\bj^\perp\in
\calM(\edg)$,
with $\rho^a=u\pi$, $2\bj^a=w\tetapi$,
 and $\rho^\perp$, $j^\perp$ as in Lemma \ref{l:4.8}, \EEE satisfy
$\calR(\rho,\bj )<\pinfty$, and let $P\in \calB(V)$ be a $\pi$-negligible set
such that $\rho^\perp=\rho\mres P$.
\begin{enumerate}[label=(\arabic*)]
\item We have 
  $|\bj |(P\times (V\setminus P))=
  |\bj |((V\setminus P)\times P)=0$, $\bj^\perp=\bj \mres(P\times P)$, and 
  \begin{equation}
    \label{eq:37}
    \calR(\rho,\bj )=
    \calR(\rho^a,\bj^a)+
    \frac12 \calF_{\Upsilon^\infty}(\teta^-_{\rho^\perp},\teta^+_{\rho^\perp},2\bj^\perp).
  \end{equation}
  In particular, if $\upalpha$ is $1$-homogeneous we have
  the decomposition
  \begin{equation}
    \label{eq:175}
     \calR(\rho,\bj )=
     \calR(\rho^a,\bj^a)+
     \calR(\rho^\perp,\bj^\perp).
  \end{equation}
\item 
\label{l:alt-char-R:i2}
If $\rho\ll \pi$ or $\upalpha$ is sub-linear,
  i.e.~$\upalpha^\infty\equiv0$,
  or
  $\kappa(x,\cdot)\ll\pi$ for every $x\in V$, 
  then $\bj \ll\tetapi$ and $\bj^\perp\equiv0$. In any of these three cases,
  $\calR(\rho,\bj )  
=
\calR(\rho^a,\bj)$, and 
setting
$\edg'$ as in \eqref{concentration-set}
we have $w=0$ $\tetapi$-a.e.\ on $\edg\setminus\edg'$, and
\eqref{eq:21} holds.
\item
Furthermore, $\calR$ is
convex and lower semicontinuous with respect to ~setwise convergence in $(\rho,\bj)$. If $\kappa$ satisfies the weak Feller
property, then $\calR$ is also lower semicontinuous with respect to weak
convergence  in duality with continuous bounded functions. \EEE
\end{enumerate}
 \end{lemma}
 \begin{proof}
   \textit{(1)}
   Equation~\eqref{eq:37} is an immediate consequence of \eqref{eq:94}.
   
   To prove the properties of $\bj$, set $R = V\setminus P$ for convenience. By using the decompositions
   $\bj =w\tetapi+w'\bsigma$ and $\teta_{\rho}^\pm = \teta_{\rho^a}^\pm + \teta_{\rho^\perp}^
   \pm = \teta_{\rho^a}^\pm + z^\pm \bsigma$
   introduced in the proof of the previous Lemma,
   the definition~\eqref{eq:30} implies that
   $\teta^+_{\rho^\perp}(P\times R)=0$, so that $z^+=0$
   $\bsigma$-a.e.~in $P\times R$; analogously $z^-=0$
   $\bsigma$-a.e.~in $R\times P$.
   By \eqref{alpha-0} we find that $\upalpha^\infty(z^-,z^+)=0$,
   $\bsigma$-a.e.~in $(P\times R)\cup (R\times P)$
   and therefore $w'=0$
    as well, since
   $\Upsilon^\infty(z^-,z^+,w')<\pinfty$
   $\bsigma$-a.e  (see \eqref{heartsuit}). \EEE
   We eventually deduce that $\bj^\perp=\bj \mres P\times P$.

   \textit{(2)}
   When $\rho\ll\pi$ we can choose $P=\emptyset$ so that
   $\bj^\perp=\bj\mres P=0$.
   When $\upalpha$ is sub-linear then
   $\bnu_\rho\ll \tetapi$ so that $\bj \ll\tetapi$ since $\Psi$ is
   superlinear.
   
   If $\kappa(x,\cdot)\ll \pi$ for every $x\in V$, then
   $\sfy_\sharp \teta^-_{\rho^\perp}\ll \pi$ and
   $\sfx_\sharp \teta^+_{\rho^\perp}\ll \pi$,
   so that
   $\teta^\pm_{\rho^\perp}(P\times P)=0$,
   since $P$ is
   $\pi$-negligible. We deduce that $\bj^\perp(P\times P)=0$ as well.
      
   \noindent

   \textit{(3)}
   The convexity of $\calR$ follows by the convexity of the functional
   $\calF_\Upsilon$. 
   The lower semicontinuity follows by combining Lemma
   \ref{le:kernel-convergence}
   with Lemma \ref{l:lsc-general}. \nc
 \end{proof}
 \begin{cor}
   \label{cor:decomposition}
   Let $\pi_1,\pi_2\in \calM^+(V)$ be
   mutually singular measures satisfying the detailed balance
   condition with respect to ~$\kappa$, and let $\tetapi_i=\boldsymbol \kappa_{\pi_i}$
   be the corresponding symmetric measures in $\calM^+(\edg)$ (see Section~\ref{subsub:kernels}).
   For every pair $(\rho,\bj)$ with $\rho=\rho_1+\rho_2$,
   $\bj=\bj_1+\bj_2$ for $\rho_i\ll\pi_i$ and $\bj_i\ll\tetapi_i$,
   we have
   \begin{equation}
     \label{eq:176}
     \calR(\rho,\bj)=\calR_1(\rho_1,\bj_1)+\calR_2(\rho_2,\bj_2),
   \end{equation}
   where $\calR_i$ is the dissipation functional induced by
   $\tetapi_i$. When $\upalpha$ is $1$-homogeneous, $\calR_i=\calR$.
 \end{cor}

\subsection{Curves with finite \texorpdfstring{$\calR$}R-action}

In this section, we study the properties of curves with finite $\calR$-action, i.e.,
elements of
\begin{equation}
\label{def:Aab}
\CER ab: = \biggl\{ 
(\rho,\bj ) \in \CE ab:\ 
\int_a^b \calR(\rho_t,\bj_t)\, \dd t <\pinfty 
\biggr\}.
\end{equation}
%
The finiteness of the $\calR$-action leads to the following remarkable property: 
A curve $(\rho,\bj)$ with finite $\calR$-action can be separated into two mutually singular curves $(\rho^a,\bj^a),\ (\rho^\perp,\bj^\perp)\in \CER ab$ that evolve independently, and contribute independently to~$\calR$. Consequently, finite $\calR$-action preserves $\pi$-absolute continuity of $\rho$: if $\rho_t\ll\pi$ at any $t$, then $\rho_t\ll\pi$ at all~$t$. 
These properties and others are proved in Theorem~\ref{thm:confinement-singular-part} and Corollary~\ref{cor:propagation-AC} below.

\begin{remark}\label{rem:skew-symmetric}
  If $(\rho,\bj )\in \CER ab$ then the `skew-symmetrization' 
$  \tj=(\bj -\symmap \bj )/2$ \EEE  of $\bj $ gives rise to a pair
 $(\rho,\tj)\in \CER ab$ \EEE as well, and  it
has lower $\calR$-action:
\[
\int_a^b \calR(\rho_t,  \tj_t)\, \dd t \EEE
\leq 
\int_a^b \calR(\rho_t,\bj_t)\, \dd t.
\]
This follows from the convexity of $w\mapsto\Upsilon(u_1,u_2,w)$, the symmetry of $(u_1,u_2)\mapsto\Upsilon(u_1,u_2,w)$, and the invariance of the continuity equation~\eqref{eq:ct-eq-def} under the `skew-symmetrization' $\bj  \mapsto   \tj$
\RICKYNEW (cf.\ also the calculations in the proof of Corollary \ref{th:chain-rule-bound2}). 
\EEE 
As a result, we can often assume without loss of generality that a flux $\bj $ is skew-symmetric, i.e.\ that $\symmap \bj  = -\bj $. 
\end{remark}

\begin{theorem}
  \label{thm:confinement-singular-part}
  Let $(\rho,\bj )\in \CER ab$
  and let us consider the Lebesgue decompositions
  $\rho_t=\rho_t^a+\rho_t^\perp$
  and
  $\bj_t =\bj_t^a+\bj_t^\perp$ of $\rho_t$ with respect to ~$\pi$ and
  of $\bj_t $ with respect to ~$\tetapi$.
  \begin{enumerate}
  \item We have $(\rho^a,\bj^a)\in \CER ab$ with
    \begin{equation}
      \label{eq:55}
      \int_a^b \calR(\rho^a_t,\bj^a_t)\, \dd t \le \int_a^b \calR(\rho_t,\bj_t)\, \dd t .
    \end{equation}
    In particular $t\mapsto \rho_t^a(V)$ and $t\mapsto\rho_t^\perp(V)$ are  constant.
  \item If $\upalpha$ is $1$-homogeneous then also
    $(\rho^\perp,\bj^\perp)\in \CER ab$ and
        \begin{equation}
      \label{eq:55bis}
      \int_a^b \calR(\rho^a_t,\bj^a_t)\, \dd t +
      \int_a^b \calR(\rho^\perp_t,\bj^\perp_t)\, \dd t=
      \int_a^b \calR(\rho_t,\bj_t)\, \dd t .
    \end{equation}
  \item If $\upalpha$ is sub-linear or
    $\kappa(x,\cdot)\ll\pi$ for every $x\in V$,
    then $\rho_t^\perp$ is constant in
      $[a,b]$ and $\bj^\perp\equiv0$.
     \end{enumerate}
   \end{theorem}
   \begin{proof}
     \textit{(1)} Let $\gamma\in \calM^+(V)$ be a dominating measure
     for the curve $\rho$ according to Corollary~\ref{c:narrow-ct}
     and let us denote by $\gamma=\gamma^a+\gamma^\perp$ the
     Lebesgue
     decomposition of $\gamma$ with respect to ~$\pi$; we also denote by
     $P\in\calB(V)$ a $\pi$-negligible Borel set such that
     $\gamma^\perp=\gamma\mres P$. Setting
     $R:=V\setminus P$, since $\rho_t\ll \gamma$
     we thus obtain $\rho^a_t=\rho_t\mres R$,
     $\rho^\perp_t=\rho_t\mres P$. By Lemma \ref{l:alt-char-R}
     for $\Lebone$-a.e.~$t\in (a,b)$ we obtain
     $\bj^\perp_t=\bj \mres(P\times P)$ and
     $\bj^a_t=\bj \mres(R\times R)$ with $|\bj_t|(R\times P)=|\bj_t|(P\times
     R)=0$.
     For every function $\varphi\in\Bb$
     we have $\ona(\varphi\chi_R)\equiv0$ on $P\times P$ so that
     we get
     \begin{align*}
       \int_V \varphi\,\dd \rho_{t_2}^a-
       \int_V \varphi\,\dd \rho_{t_1}^a
       &=
         \int_R \varphi\,\dd \rho_{t_2}-
         \int_R \varphi\,\dd \rho_{t_1}=
         \int_{t_1}^{t_2} \iint_{\edg}
         \ona(\varphi \chi_R)\,\dd (\bj^a_t+\bj^\perp_t)\,\dd t
       \\&=\int_{t_1}^{t_2} \iint_{R\times R}
       \ona(\varphi \chi_R)\,\dd \bj^a_t\,\dd t
       =\int_{t_1}^{t_2} \iint_{\edg}
       \ona\varphi\,\dd \bj^a_t\,\dd t,
     \end{align*}
     showing that $(\rho^a,\bj^a)$ belongs to $\CE ab$. Estimate
     \eqref{eq:55} follows by \eqref{eq:37}.
     \OLI From Lemma~\ref{l:cont-repr} \EEE we deduce that $\rho_t^a(V)$ and $\rho_t^\perp(V)$ are constant.

\medskip

     \textit{(2)} This follows by the linearity of the continuity
     equation and \eqref{eq:175}.
     
\medskip     
     
     \textit{(3)} If $\upalpha$ is sub-linear or $\kappa(x,\cdot)\ll\pi$ for every
     $x\in V$, then 
     Lemma~\ref{l:alt-char-R}
     shows that $\bj^\perp\equiv 0$. Since by linearity
     $(\rho^\perp,\bj^\perp)\in \CE ab$, we deduce that $\rho^\perp_t$ is constant.
   \end{proof}
   
   \begin{cor}\label{cor:propagation-AC}
    \OLI Let $(\rho,\bj )\in \CER ab$. \EEE If
    there exists $t_0\in [a,b]$ such that
    $\rho_{t_0}\ll\pi$, then we have $\rho_t\ll\pi$ for every $t\in [a,b]$,
    $\bj^\perp\equiv0$, and
    $\odiv \bj_t\ll \pi$ for $\Lebone$-a.e.~$t\in (a,b)$.
    In particular, there exists an absolutely continuous and
    a.e.~differentiable map $u:[a,b]\to L^1(V,\pi)$ and
    a map $w\in L^1(\edg,\lambda\otimes\tetapi)$ 
    such that 
    \begin{equation}
      \label{eq:42}
      2\bj_\lambda=w \lambda\otimes\tetapi,\quad
      \partial_t u_t(x)=\frac12\int_V
      \big(w_t(y,x)-w_t(x,y)\big)\,\kappa(x,\dd y)
      \quad\text{for a.e.~}t\in (a,b).
    \end{equation}
    Moreover there exists a measurable map
    $\xi:(a,b)\times \edg\to \R$
    such that 
    $w=\xi\upalpha(u^-,u^+)$ $\Lebone\otimes\tetapi$-a.e.~and
    \begin{equation}
      \label{eq:45}
      \calR(\rho_t,\bj_t)=
      \frac12\iint_\edg
      \Psi(\xi_t(x,y))\upalpha(u_t(x),u_t(y))\,\tetapi(\dd x,\dd y)
      \quad\text{for a.e.~$t\in (a,b)$.}
    \end{equation}
    If $w$ is skew-symmetric, then $\xi$ is skew-symmetric as well and
    \eqref{eq:42} reads as
    \begin{equation}
      \label{eq:177}
       \partial_t u_t(x)=\int_V
       w_t(y,x)\,\kappa(x,\dd y)=
       \int_V
       \xi_t(y,x)\upalpha(u_t(x),u_t(y))\,\kappa(x,\dd y)
       \quad\text{a.e.~in }(a,b).
    \end{equation}
  \end{cor}
  \begin{remark}
    \label{rem:general-fact}
  Relations   \eqref{eq:42} and \eqref{eq:177} hold both in the sense of a.e.~differentiability
    of maps with values in $L^1(V,\pi)$ and pointwise
    a.e.~with respect to ~$x\in V$: more precisely, there exists a set $U\subset
    V$ of
    full $\pi$-measure such that for every $x\in U$ the map
    $t\mapsto u_t(x)$ is absolutely continuous and  equations
    \eqref{eq:42} and \eqref{eq:177} hold for every $x\in U$,
    a.e.~with respect to ~$t\in (0,T)$.
  \end{remark}
  \begin{proof}
    The first part of the statement is an immediate consequence of
    Theorem~\ref{thm:confinement-singular-part},
    which yields $\rho^\perp_t(V)= 0$ for every $t\in [a,b]$.
    We can thus write $2\bj =w(\Lebone\otimes \tetapi)$ for some
    measurable map $w:(a,b)\times \edg\to \R$.
    Moreover
    $\odiv \bj \ll\Lebone\otimes \pi$, since 
    $\sfs_\sharp\bj\ll\sfs_\sharp (\Lebone\otimes\tetapi)=
    \Lebone\otimes\tetapi$, and therefore
    \begin{equation}
      \label{eq:46}
     2 \bj^\flat=\bj-\sfs_\sharp\bj\ll \Lebone\otimes\tetapi\quad\Longrightarrow\quad
      \odiv \bj=
      \sfx_\sharp (2\bj^\flat)\ll \EEE \sfx_\sharp(\Lebone\otimes\tetapi)\OLI \ll \EEE\Lebone\otimes\pi.
    \end{equation}
    Setting $z_t=\dd(\odiv \bj_t)/\dd\pi$ we get
    for a.e.~$t\in (a,b)$
    \begin{align*}
      \partial_t u_t&=-z_t,\\
      -2\int_V \varphi \,z_t\,\dd\pi&=
                                    \iint_\edg
                   (\varphi(y)-\varphi(x))w_t(x,y)\tetapi(\dd x,\dd y)
                   =
                   \iint_\edg  \varphi(x)
                                    (w_t(y,x)-w_t(x,y))\tetapi(\dd x,\dd y)
                   \\&=
                   \int_V \varphi(x)
                   \Big(\int_V (w_t(y,x)-w_t(x,y))
                   \kappa(x,\dd
      y)\Big)\pi(\dd x),
    \end{align*}
    The existence of $\xi$ and  formula~\eqref{eq:45} follow from
    Lemma \ref{l:alt-char-R}\ref{l:alt-char-R:i2}.
  \end{proof}

  \subsection{Chain rule for convex entropies}
  \label{subsec:chain-rule}
  Let us now consider a continuous convex function $\upbeta:\R_+\to\R_+$
  that is differentiable in $(0,+\infty)$. The main choice for $\upbeta$ will be the function~$\upphi$ that appears in the definition of the driving functional~$\calS$ (see Assumption~\ref{ass:S}), and the  example of the Boltzmann-Shannon entropy function~\eqref{logarithmic-entropy} illustrates why we only assume differentiability away from zero.
  
  By setting $\upbeta'(0)=\lim_{r\downarrow0} \upbeta'(r)\in
  [-\infty,\pinfty)$,
  we define 
  the function $\rmA_\upbeta:\R_+\times \R_+\to[-\infty,+\infty]$   by
  \begin{equation}
    \label{eq:102}
    \rmA_\upbeta(u,v):=
    \begin{cases}
      \upbeta'(v)-\upbeta'(u)&\text{if }u,v\in
      \R_+\times \R_+\setminus \{(0,0)\},\\
      0&\text{if }u=v=0.      
    \end{cases}
  \end{equation}
  Note that $\rmA_\upbeta$ is continuous
  (with extended real values) in $\R_+\times \R_+\setminus\{(0,0)\}$
  and is finite and continuous whenever $\upbeta'(0)>-\infty$.
  When $\upbeta'(0)=-\infty$ we have
  $\rmA_\upbeta(0,v)=-\rmA_\upbeta(u,0)=\pinfty$ for every $u,v>0$.

  In the following we will adopt the convention
  \begin{equation}
    \label{eq:75}
    |\pm\infty|=\pinfty,\quad
    a\cdot (\pinfty):=
    \begin{cases}
      \pinfty&\text{if }a>0,\\
      0&\text{if }a=0,\\
      -\infty&\text{if }a<0
    \end{cases}
    \quad
    a\cdot(-\infty)=-a\cdot (\pinfty),
  \end{equation}
  for every $a\in [-\infty,+\infty]$ and, using this convention,  we define the extended
  valued function
  $\rmB_\upbeta:\R_+\times\R_+\times \R\to [-\infty,+\infty]$ by
  \begin{equation}
    \label{eq:105}
    \rmB_\upbeta(u,v,w):=\rmA_\upbeta(u,v)w.
  \end{equation}
  We want to study the differentiability properties of the
  functional $\calF_\upbeta(\cdot|\pi)$ along
  solutions $(\rho,\bj)\in \CEI I$ of the continuity equation.
  Note that if $\upbeta$ is superlinear
  and $\calF_\upbeta$ is finite at a time $t_0\in I$,
  then Corollary \ref{cor:propagation-AC}
  shows that $\rho_t\ll\pi$ for every $t\in I$.
  If $\upbeta$ has linear growth then
  \begin{equation}
    \label{eq:100}
    \calF_\upbeta(\rho_t|\pi)=
    \int_V \upbeta(u_t)\,\dd\pi+\upbeta^\infty(1)\rho^\perp(V),\quad
    \rho_t=u_t\pi+\rho_t^\perp,
  \end{equation}
   where we have used that  $t \mapsto \rho_t^\perp(V)$ is constant. Thus,  we are  reduced to  studying 
  $\calF_\upbeta$ along $(\rho^a,\bj^a)$, which is still a solution
  of the continuity equation. 
  The absolute continuity property of $\rho_t$ with respect to ~$\pi$
  is therefore quite a natural assumption
  in the next result.

  \begin{theorem}[Chain rule I]
\label{th:chain-rule-bound}
Let $(\rho,\bj )\in \CER ab$ with $\rho_t=u_t\pi\ll \pi$
and   let $2\bj^\flat=\bj-\sfs_\sharp \bj=w^\flat\Lebone\otimes\tetapi$ \EEE as in Corollary \ref{cor:propagation-AC}
satisfy 
\begin{equation}
\label{ass:th:CR}
\int_V \upbeta(u_a)\,\dd\pi<\pinfty,\quad
\int_a^b\iint_\edg\Big(\rmB_\upbeta(u_t(x),u_t(y),w^\flat_t(x,y))\Big)_+\,\tetapi(\dd x,\dd y)\,\dd t<\pinfty
\end{equation}
Then the map $t\mapsto \int_V \upbeta(u_t)\,\dd\pi$ is absolutely
continuous in $[a,b]$,
the map $\rmB_\upbeta(u^-,u^+,w^\flat)$ is $\Lebone\otimes\tetapi$-integrable
and
\begin{equation}
\label{eq:CR}
  \frac{\dd}{\dd t}\int_V \upbeta(u_t)\,\dd\pi
  =
\frac12\iint_{\edg} \rmB_\upbeta(u_t(x),u_t(y),w^\flat_t(x,y)) \tetapi(\dd x,\dd y)
  \quad\text{for a.e.~}t\in (a,b).
\end{equation}
\end{theorem}

\begin{remark}
At first sight  condition~\eqref{ass:th:CR} on the positive part of $\rmB_\upbeta$ is remarkable: we only require the positive part of $\rmB_\upbeta$ to be integrable, but in the assertion we obtain integrability of the negative part as well. This integrability arises from the combination of the upper bound on $\int_V \upbeta(u_a)\,\dd \pi$ in~\eqref{ass:th:CR} with the lower bound $\upbeta\geq0$.
\end{remark}

\begin{proof}
{\em Step 1:\ Chain rule for an approximation.}
Define for $k\in \N$ an approximation $\upbeta_k$ of $\upbeta$ as follows:\ Let $\upbeta_k'(\sigma):=\max\{-k,\min\{\upbeta'(\sigma),k\}\}$ be the truncation of $\upbeta'$ to the interval $[-k,k]$. Due to the assumptions on $\upbeta$, we may assume that $\upbeta$ achieves a minimum at the point $s_0\in[0,\pinfty)$. Now set $\upbeta_k(s) := \upbeta(s_0) + \int_{s_0}^s \upbeta_k'(\sigma)\,\dd \sigma$. Note that $\upbeta_k$ is differentiable and globally Lipschitz, and converges monotonically to $\upbeta(s)$ for all $s\geq0$ as $k\to\infty$.

For each $k\in\N$ and $t\in [a,b]$
we define 
\[
  S_{k}(t): = \int_{V}  \upbeta_k(u_t)\, \dd \pi,\quad
  S(t): = \int_{V}  \upbeta(u_t)\, \dd\pi.
\]
By convexity and Lipschitz continuity  of $\upbeta_k$, we have that
\begin{align*}
	\upbeta_k(u_t(x))-\upbeta_k(u_s(x)) \le \upbeta_k'(u_t(x))( u_t(x)-u_s(x)) \le k| u_t(x)-u_s(x)|\,.
\end{align*}
Hence, we deduce by Corollary \ref{cor:propagation-AC}
that for every $a\le s<t\le b$
\begin{align*}
  S_{k}(t) - S_{k}(s) &= 
                        \int_{V} \bigl[\upbeta_k(u_t(x))-\upbeta_k(u_s(x))\bigr]\pi(\dd x) \\
                            &\le
                              k\|u_t-u_s\|_{L^1(V;\pi)}
                              \le k\int_s^t \|\partial_r
                              u_r\|_{L^1(V;\pi)}\,\dd r.
\end{align*}
%
We conclude that the function $t\mapsto S_k(t)$ is absolutely
continuous.
Let us pick a point $t\in (a,b)$ of differentiability for $t\mapsto
S_k(t)$:
it easy to check that 
\begin{align*}
  \frac{\dd }{\dd t}S_k(t) &= 
                             \int_{V}  \upbeta'_k(u_t)\,\partial_t  u_t
                             \,\dd\pi =
                             \frac12\iint_{\edg} \ona \upbeta'_k(u_t)w^\flat_t \, \dd\tetapi\,,
\end{align*}
which by integrating over time yields
\begin{equation}\label{ineq:est-S-k}
  S_k(t) - S_k(s) =
  \frac12\int_s^t\iint_{\edg} \ona \upbeta'_k(u_r)w^\flat_r\, \dd \tetapi\,\dd r  \qquad \text{for all } a \leq s \leq t \leq b. 
\end{equation}
\smallskip
\paragraph{\em Step 2:\ The limit $k\to\infty$}
Since $0\leq \upbeta_k''\leq \upbeta''$ we have
\begin{equation}
  \label{eq:103}
  0\le \rmA_{\upbeta_k}(u,v)=\upbeta_k'(v)-\upbeta_k'(u)\le
  \upbeta'(v)-\upbeta'(u)=\rmA_\upbeta(u,v)\quad\text{whenever }0\le
  u\le v
\end{equation}
and
\begin{equation}
  \label{eq:104}
  |\upbeta_k'(v)-\upbeta_k'(u)|\le |\rmA_\upbeta(u,v)|\quad
  \text{for every }u,v\in \R_+.
\end{equation}
We can thus estimate 
the right-hand side in \eqref{ineq:est-S-k} 
\begin{align}
  (B_k)_+=\left(
  \ona
  \upbeta'_k(u)\, w^\flat\right)_+
  &	
    \le
    \left(\rmA_\upbeta(u^-,u^+) w^\flat\right)_+=B_+
	\label{est:CR:w-ona-phi-k}
\end{align}
where we have used the short-hand notation 
\begin{equation}
  \label{eq:106}
  B_k(r,x,y)=\rmB_{\upbeta_k}(u_r(x),u_r(y),w^\flat_r(x,y)),\quad
  B(r,x,y):=\rmB_\upbeta(u_r(x),u_r(y),w^\flat_r(x,y)).
\end{equation}
Assumption~\eqref{ass:th:CR} implies that the right-hand side
in~\eqref{est:CR:w-ona-phi-k} is an element of $L^1([a,b]\times
\edg;\Lebone\otimes\tetapi)$,
so that in particular $B_+\in \R$ for
$(\Lebone\otimes\tetapi)$-a.e.~$(t,x,y)$.

Moreover, \eqref{ineq:est-S-k} yields
\begin{align}
  \int_a^b \iint_\edg (B_k)_-\,\dd \tetapi_\Lebone
  &=\notag
    \int_a^b \iint_\edg (B_k)_+\,\dd \tetapi_\Lebone+
    S_k(a)-S_k(b)
  \\
  &\le \int_a^b \iint_\edg (B)_+\,\dd \tetapi_\Lebone+
    S(a)<\pinfty.\label{eq:193}
\end{align}
Note that the sequence $k\mapsto (B_k)_-$ is definitely $0$ or
is monotonically increasing to $B_-$. Beppo Levi's Monotone
Convergence Theorem and the uniform estimate
\eqref{eq:193}
then yields that $B_-\in L^1((a,b)\times \edg,\Lebone\otimes\tetapi)$,
thus showing that $\rmB_\upbeta(u^-,u^+,w^\flat)$ is
$(\Lebone\otimes\tetapi)$-integrable
as well.

We can thus pass to the limit in \eqref{ineq:est-S-k} as
$k\to\pinfty$ and we have
\begin{equation}
\lim_{k\to\pinfty} \ona
  \upbeta'_k(u)\, w^\flat =B
\quad
\text{$\Lebone\otimes \tetapi$-a.e.~in
$(a,b)\times \edg$.}
\label{eq:57}
\end{equation}
The identity~\eqref{eq:57} is obvious if $\upbeta'(0)$ is finite, and if
$\upbeta'(0)=-\infty$ then it follows by the upper bound \eqref{est:CR:w-ona-phi-k}
and the fact that the right-hand side of \eqref{est:CR:w-ona-phi-k}
is finite almost everywhere.

\nc
The Dominated Convergence Theorem then implies that 
\[
\int_s^t\iint_{\edg} \ona \upbeta'_k(u_r)\, w^\flat_r \,\dd \tetapi\,\dd r 
\quad\longrightarrow\quad 
\int_s^t\iint_{\edg}B\, \dd \tetapi\,\dd r
\qquad\text{as}\quad k\to\infty\,.
\]
By the monotone convergence theorem  $S(t) = \lim_{k\to \pinfty}
S_k(t)\in [0,\pinfty]$
for all $t\in [a,b]$ and the limit is finite for $t=0$.
For all $t\in [a,b]$, therefore,
\[
S(t)
= S(a)+ \frac12\int_a^t \iint_\edg B \, \dd \tetapi  \,\dd  r,
\]
which shows that $S$ is absolutely continuous and \eqref{eq:CR} holds.
\end{proof}
%


\par
 We now introduce three functions associated with the (general) continuous convex function $\upbeta:\R_+\to\R_+$,
   differentiable in $(0,+\infty)$, that we have considered so far, and whose main example  will be the entropy density 
$\upphi$ from \eqref{cond-phi}. \GGGO
Recalling the definition \eqref{eq:102}, the convention \eqref{eq:75},
and setting $\Psi^*(\pm\infty):=\pinfty$, let us now introduce
the functions $\rmD^+_\upbeta, \rmD^-_\upbeta,\rmD_\upbeta:\R_+^2\to[0,\pinfty]$
\begin{subequations}
  \label{subeq:D}
  \begin{align}
    \label{eq:181}
    \rmD^-_\upbeta(u,v)&:=
      \Psi^*(\rmA_\upbeta(u,v))\upalpha(u,v)\\
      &\phantom:= \begin{cases}
        \Psi^*(\rmA_\upbeta(u,v))\upalpha(u,v)&\text{if
        }\upalpha(u,v)>0,\\
        0&\text{otherwise,}
         \end{cases}\notag\\[2\jot]
    \label{eq:52}
    \rmD^+_\upbeta(u,v)&:=
      \begin{cases}
        \Psi^*(\rmA_\upbeta(u,v))\upalpha(u,v)&\text{if
        }\upalpha(u,v)>0,\\
        0&\text{if }u=v=0,\\
        \pinfty&\text{otherwise, i.e.~if }\upalpha(u,v)=0,\ u\neq v,
      \end{cases}\\[2\jot]
    \label{eq:182}
    \rmD_\upbeta(\cdot,\cdot)&:=
    \text{the lower semicontinuous envelope of
      $\rmD_\upbeta^+$ in $\R_+^2$}.
  \end{align}
\end{subequations}
 The function $\rmD_\upphi$ corresponding to the choice $\upbeta = \upphi$ shall feature in the (rigorous) definition of the \emph{Fisher information} functional $\Fish$, cf.\ 
\eqref{eq:def:D} ahead. Nonetheless, it is significant to introduce the functions $ \rmD^-_\upphi$ and $ \rmD^+_\upphi$ as well, cf.\ Remarks \ref{rmk:why-interesting-1} 
and  
 \ref{rmk:Mark} 
ahead. \EEE
\begin{example}[The functions $\rmD^\pm_\upphi$ and $\rmD_\upphi$ in the quadratic and in the  $\cosh$ case]
 \label{ex:Dpm}
In the two examples of the linear equation~\eqref{eq:fokker-planck}, with Boltzmann entropy function~$\upphi$, and  with quadratic and cosh-type potentials $\Psi^*$ (see~\eqref{choice:cosh} and~\eqref{choice:quadratic}), the functions $\rmD^\pm_\upphi$ and $\rmD_\upphi$ take the following forms:
    \begin{enumerate}
        \item If $\Psi^*(s)=s^2/2$ and, accordingly,  $\upalpha(u,v)=(u-v)/(\log(u)-\log(v))$ for all $u,v >0$ (with $\upalpha(u,v)=0$ otherwise), then        \begin{align*}
        \OrmD^-_\upphi(u,v) &= \begin{cases}
                  \frac{1}{2}(\log(u)-\log(v))(u-v)  & \text{if }  u,\, v>0,
                  \\
                  0 & \text{if $u=0$ or $v=0$},
                  \end{cases}\\
        \OrmD_\upphi(u,v) = \OrmD^+_\upphi(u,v)&= \begin{cases}
                  \frac{1}{2}(\log(u)-\log(v))(u-v)  & \text{if }  u,\, v>0,
                  \\
                  0 & \text{if } u=v=0,
                  \\
                  \pinfty & \text{if } u=0 \text{ and } v \neq 0, \text{ or vice versa}.
                  \end{cases}
        \end{align*}
        {For this example $\OrmD_\upphi^+$ and $\OrmD_\upphi$ are convex, and all three functions are lower semicontinuous.}

        \item For the case $\Psi^*(s)=4\bigl(\cosh(s/2)-1\bigr)$ and, accordingly, $\upalpha(u,v)=\sqrt{u v}$ for all $u,v \geq 0$, then         \begin{align*}
       \OrmD^-_\upphi(u,v) &= \begin{cases}
                  2\Bigl(\sqrt{u}-\sqrt{v}\Bigr)^2 & \text{if }  u,\, v>0,
                  \\
                  0 & \text{if $u=0$ or $v=0$},
                  \end{cases}\\
          \OrmD_\upphi(u,v) &= 2\Bigl(\sqrt{u}-\sqrt{v}\Bigr)^2\qquad {\text{for all }u,v\geq 0,}\\
          \OrmD^+_\upphi(u,v) &= \begin{cases}
                  2\Bigl(\sqrt{u}-\sqrt{v}\Bigr)^2 & \text{if  $u, v>0$ or $u=v=0$},
                  \\
                  \pinfty & \text{if } u=0 \text{ and } v \neq 0, \text{ or vice versa}.
                  \end{cases} 
        \end{align*}
        For this example,  $\OrmD_\upphi^+$ and $\OrmD_\upphi$ again are convex, but only $\rmD^-_\upphi$ and $\rmD_\upphi$ are lower semicontinuous.
      \end{enumerate}
  \end{example}

We collect a number of general properties of $\rmD_\upbeta$ and $\rmD_\upbeta^\pm$.

%
  \begin{lemma}
  \label{le:trivial-but-useful}
  \begin{enumerate}[ref=(\arabic*)]
  \item $\rmD_\upbeta^-\leq \rmD_\upbeta\leq \rmD_\upbeta^+$;
  \item $\rmD_\upbeta^-$ and $\rmD_\upbeta$ are  lower semicontinuous;
   \item  \label{le:trivial-but-useful:ineq}
    For every $u,v\in \R_+$ and $w\in \R$ we have
    \begin{equation}
      \label{eq:107}
      \bigl|\rmB_\upbeta(u,v,w)\bigr|
      \le
      \Upsilon(u,v,w)+\rmD^-_\upbeta(u,v).
    \end{equation}
    \item 
    Moreover, when the right-hand side of \eqref{eq:107}
    is finite, then the equality 
    \begin{equation}
      \label{eq:107a}
        -\rmB_\upbeta(u,v,w)
      =
      \Upsilon(u,v,w)+\rmD^-_\upbeta(u,v)
    \end{equation}
    is equivalent to the condition
    \begin{equation}
      \label{eq:109}
      \upalpha(u,v)=w=0\quad\text{or}\quad\biggl[
      \upalpha(u,v)>0,\ \rmA_\upbeta(u,v)\in \R,\
      -w=(\Psi^*)'\big(\rmA_\upbeta(u,v)\big)\upalpha(u,v)\biggr].
    \end{equation}
      \end{enumerate}
  \end{lemma}
  \begin{proof}
  It is not difficult
  to check that 
  $\rmD^-_\upbeta$ is lower
  semicontinuous: such a property is trivial
  where $\upalpha$ vanishes, and  in all the other cases
  it is sufficient to use the positivity and the
  continuity of $\Psi^*$ in $[-\infty,+\infty]$,
  the continuity of $\rmA_\upbeta$ in $\R_+^2\setminus\{(0,0)\}$, and
  the continuity and the positivity of $\upalpha$.
  It is also obvious that $\rmD^-_\upbeta\le \rmD^+_\upbeta$,
  and therefore  $\rmD^-_\upbeta\le \rmD_\upbeta\le \rmD^+_\upbeta$.

For the inequality~\eqref{eq:107}, let us distinguish the various cases:
    \begin{itemize}
    \item If $w=0$ or $u=v=0$,
      then $\rmB_\beta(u,v,w) =0$ so that
      \eqref{eq:107} is trivially satisfied.
      We can thus assume
      $w\neq0$ and $u+v>0$.
    \item  When $\upalpha(u,v)=0$ then
      $\Upsilon(u,v,w) =\pinfty$ so that \eqref{eq:107} is trivially
      satisfied as well. We can thus assume $\upalpha(u,v)>0$.
    \item 
      If
      $\rmA_\upbeta(u,v)\in \{\pm\infty\}$ then
      $\rmD_\upbeta^-(u,v)=\pinfty$ and the right-hand side of
      \eqref{eq:107} is infinite.
    \item
      It remains to consider the case when
      $\rmA_\upbeta(u,v)\in \R$, $\upalpha(u,v)>0$ and $w\neq0$.
      In this situation
      \begin{align}
          \bigl|\rmB(u,v,w)\bigr|&=\bigl|\rmA_\upbeta(u,v)w\bigr|=
          \bigg|\rmA_\upbeta(u,v)\frac{w}{\upalpha(u,v)}\bigg| \upalpha(u,v)
          \notag \\&\le
          \Psi\Big(\frac w{\upalpha(u,v)}\Big) \upalpha(u,v)+
          \Psi^*\Big(\rmA_\upbeta(u,v)\Big) \upalpha(u,v)
          \notag\\
          &=\Upsilon(u,v,w)+\rmD_\upbeta^-(u,v).
                  \label{eq:108}
      \end{align}
This proves~\eqref{eq:107}.
    \end{itemize}
    
    It is now easy to study the case of equality in~\eqref{eq:107a}, when the right-hand side of
    \eqref{eq:107} and \eqref{eq:107a}
    is finite. This in particular implies
    that $\upalpha(u,v)>0$ and $\rmA_\upbeta(u,v)\in \R$ or
    $\upalpha(u,v)=0$ and $w=0$.
    In the  former  case, 
     calculations similar to \eqref{eq:108} show
    that $-w=(\Psi^*)'\big(\rmA_\upbeta(u,v)\big)\upalpha(u,v)$.   In the latter case, $\alpha(u,v) = w =0$ yields that 
     $   \rmB_\upbeta(u,v,w)=0$, $\Upsilon (u,v,w) = \hat{\Psi}(w,\alpha(u,v))  = \hat{\Psi}(0,0) = 0$, and 
     $\rmD_\upbeta(u,v) = \Psi^*(\rmA_\upbeta(u,v)) \alpha(u,v)=0$. \EEE
  \end{proof}
 \par As a consequence of Lemma \ref{le:trivial-but-useful}, we conclude a chain-rule inequality involving the smallest functional $\rmD_\upbeta^-$ and thus, a fortiori, the functional
$\rmD_\upbeta$ which, for $\upbeta=\upphi$, shall enter into the definition of the Fisher information $\Fish$.  \GGGO
  \begin{cor}[Chain rule II]
\label{th:chain-rule-bound2}
Let $(\rho,\bj )\in \CER ab$ with $\rho_t=u_t\pi\ll \pi$ and
$2\bj_\Lebone=w
(\Lebone\otimes\tetapi)$
satisfy 
\begin{equation}
\label{ass:th:CR2}
\int_V \upbeta(u_a)\,\dd\pi<\pinfty,
\quad
\int_a^b \iint_\edg \rmD^-_\upbeta(u_t(x),u_t(y))\,\tetapi(\dd x,\dd y)\dd t<\pinfty.
\end{equation}
Then the map $t\mapsto \int_V \upbeta(u_t)\,\dd\pi$ is absolutely
continuous in $[a,b]$ and
\begin{equation}
\label{eq:CR2}
\begin{aligned}
  \left|\frac{\dd}{\dd t}\int_V \upbeta(u_t)\,\dd\pi\right|
  \le
  \calR(\rho_t,\bj_t)+\frac12\iint_\edg \rmD^-_\upbeta(u_t(x),u_t(y))\,\tetapi(\dd
  x,\dd y)\quad\text{for
    a.e.~}t\in (a,b).
\end{aligned}
\end{equation}
If moreover 
\[
  -\frac{\dd}{\dd t}\int_V \upbeta(u_t)\,\dd\pi
 =\calR(\rho_t,\bj_t)+\frac12\iint_\edg \rmD^-_\upbeta(u_t(x),u_t(y))\,\tetapi(\dd
  x,\dd y)
\]
then $2\bj=\bj^\flat$ and
\begin{equation}
  \label{eq:110}
  -w_t(x,y)=(\Psi^*)'\big(\rmA_\upbeta(u_t(x),u_t(y))\big)\upalpha(u_t(x),u_t(y))
  \quad\text{for $\tetapi$-a.e.~$(x,y)\in \edg$}.
\end{equation}
In particular,
$  w_t=0$ $\tetapi$-a.e.~in $\big\{(x,y)\in \edg:
\upalpha(u_t(x),u_t(y))=0\big\}.$
\end{cor}

\begin{proof}
  We recall that for $\Lebone$-a.e.~$t\in (a,b)$
  \begin{displaymath}
    \calR(\rho_t,\bj_t)=
    \frac12\iint_\edg \Upsilon(u_t(x),u_t(y),w_t(x,y))\,\tetapi(\dd x,\dd y).
  \end{displaymath}
  We can then apply Lemma \ref{le:trivial-but-useful} and
  Theorem \ref{th:chain-rule-bound}, observing that
  \begin{equation}
    \label{eq:178}
    \iint_\edg \Upsilon(u_t(x),u_t(y),w^\flat_t(x,y))\,\tetapi(\dd
    x,\dd y)\le
    \iint_\edg \Upsilon(u_t(x),u_t(y), w(x,y))\,\tetapi(\dd
    x,\dd y)
  \end{equation}
  since
  \begin{align*}
    \Upsilon(u_t(x),u_t(y),w^\flat_t(x,y))&=
                                            \Upsilon(u_t(x),u_t(y),\frac12 (w_t(x,y)-w_t(y,x)))
                                            \\&\le
    \frac12\Upsilon(u_t(x),u_t(y),w_t(x,y))
    +\frac12\Upsilon(u_t(x),u_t(y),w_t(y,x))
  \end{align*}
  and the integral of the last term coincides with
  the right-hand side of \eqref{eq:178} thanks to the symmetry of $\tetapi$.
\end{proof}

\subsection{Compactness properties of curves with uniformly bounded
  $\calR$-action}
\label{subsec:compactness}
The next result shows an important compactness
property for collections of curves in $\CER ab$ with bounded
action.
Recalling the discussion and the notation of Section~\ref{subsub:kernels}, we will systematically associate with a given $(\rho,\bj)\in
\CEIR I$, $I=[a,b]$, a couple of measures
$\rho_\Lebone\in \calM^+(I\times V)$,
$\bj_\Lebone\in \calM(I\times \edg)$ by integrating with respect to ~the
Lebesgue measure $\Lebone$
in $I$:
\begin{equation}
  \label{eq:95}
  \rho_\Lebone(\dd t,\dd x)=
  \lambda(\dd t)\rho_t(\dd x),\quad
  \bj_\Lebone(\dd t,\dd x,\dd y)=\lambda(\dd t)\bj_t(\dd x,\dd y).
\end{equation}
Similarly, we define 
\begin{equation}
  \label{eq:97}
  \begin{aligned}
    \teta_{\rho,\Lebone}^\pm(\dd t,\dd x,\dd y):={}&
    (\teta_{\rho}^\pm)_\Lebone(\dd t,\dd x,\dd y)= \Lebone(\dd
    t)\teta_{\rho_t}^\pm(\dd x,\dd y)
    \\={}& \Lebone(\dd t)\rho_t(\dd
    x)\kappa(x,\dd y) =\teta_{\rho_\Lebone}^\pm(\dd t,\dd x,\dd y).
  \end{aligned}
\end{equation}
It is not difficult to check that
\begin{equation}
  \label{eq:96}
  \int_I \calR(\rho_t,\bj_t)\,\dd t=
  \frac12\calF_\Upsilon(\teta_{\rho,\Lebone}^-,\teta_{\rho,\Lebone}^+,2\bj_\Lebone|\Lebone\otimes\tetapi).
\end{equation}

\begin{prop}[Bounded $\int\calR$ implies compactness and lower semicontinuity]
\label{prop:compactness}
Let $(\rho^n,\bj^n)_n \subset \CER ab$ be a sequence such that
the initial states $\rho^n_a$ are $\pi$-absolutely-continuous and relatively compact with respect to setwise convergence. Assume that 
\begin{gather}
\label{eq:lem:compactness:assumptions}
M:=\sup_{n\in\N}\int_a^b \calR(\rho_t^n, \bj_t^n) \,\dd t<\pinfty.
\end{gather}
Then,
there exist a subsequence (not relabelled) and a pair $(\rho,\bj )\in
\CER ab$ such that,  for the measures $\bj_\Lebone^n \in
\calM([a,b]\times\edg)$ defined as in \eqref{eq:95} there holds 
 \begin{subequations}
 \label{cvs-rho-n-j-n}
\begin{align}
&
\label{converg-rho-n}
 \rho_t^n\to \rho_t\quad\text{setwise in $\calM^+(V)$ for all  $t\in[a,b]$}\,,\\
 &
\label{converg-j-n}
  \bj_\Lebone^n\weakto \bj_\Lebone \quad\text{setwise in $\calM([a,b]\times\edg)$}\,,
\end{align}
\end{subequations}
where $\bj_\Lebone$ is induced (in the sense of \eqref{eq:95}) 
 by
a $\Lebone$-integrable family $(\bj_t)_{t\in [a,b]}\subset \calM(\edg)$.
In addition, for any sequence $(\rho^n,\bj^n)$ converging to $(\rho,\bj )$ in the sense of~\eqref{cvs-rho-n-j-n}, we have 
\begin{equation}
\label{ineq:lsc-R}
\int_a^b \calR(\rho_t,\bj_t)\, \dd t \leq 
\liminf_{n\to\infty} \int_a^b \calR(\rho^n_t,\bj^n_t)\, \dd t .
\end{equation}
\end{prop}
%

\begin{proof}
  Let us first remark that the mass conservation property of the
  continuity equation yields 
  \begin{equation}
    \label{eq:62}
    \rho_t^n(V)=\rho_a^n(V)\le M_1\quad\text{for every }t\in [a,b],\
    n\in \N
  \end{equation}
  for a suitable finite constant $M_1$ independent of $n$.
  We deduce that for every $t\in [a,b]$ the measures
  $\tetapi_{\rho_t^n}^\pm$
  have total mass bounded by   $M_1 \|\kappa_V\|_\infty$, \EEE so that  estimate \eqref{eq:31} for
  $y=(c,c)\in D(\upalpha_*)$ yields
  \begin{equation}
    \label{eq:63}
    \bnu_{\rho^n_t}(\edg)=  \Aalpha[\tetapi^+_{\rho_t^n},\tetapi^-_{\rho_t^n}|\tetapi](\edg)  \le
    M_2
    \quad
    \text{for every }t\in [a,b],\
    n\in \N,
  \end{equation}
  where
  $M_2:=2 c\,M_1 \|\kappa_V\|_\infty \EEE-\upalpha_*(c,c)\tetapi(\edg)$.
  Jensen's inequality \eqref{eq:79}
  and the monotonicity property \eqref{eq:80} yield
  \begin{equation}
    \label{eq:64}
    \calR(\rho^n_t,\bj^n_t)\ge
    \frac12
    \hat\Psi\Bigl(2\bj_t^n(\edg),\bnu_{\rho^n_t}(\edg)\Bigr)\ge
    \frac12 \hat\Psi\Bigl(2\bj_t^n(\edg),M_2\Bigr)=
    \frac12 \Psi\Bigl(\frac{2\bj_t^n(\edg)} {M_2}\Bigr) M_2,
  \end{equation}
   with $\hat\Psi$ the perspective function associated with $\Psi$, cf.~\eqref{eq:76}. 
  Since $\Psi$ has superlinear growth,
  we deduce that the sequence of functions $t\mapsto |\bj_t^n|(\edg)$
  is equi-integrable.
  
  Since
  the sequence $(\rho_a^n)_n$, with $\rho_a^n = u_a^n \pi \ll \pi$, is relatively compact 
  with respect to setwise convergence, 
   by Theorems \ref{thm:equivalence-weak-compactness}(6) and \ref{thm:L1-weak-compactness}(3) 
  there exist a  
  convex
  superlinear function $\upbeta:\R_+\to\R_+$ and a constant $M_3<\pinfty$ such that
  \begin{equation}
    \label{eq:58}
    \mathscr F_\upbeta(\rho^n_a|\pi)=
    \int_V \upbeta(u_a^n)\,\dd\pi\le M_3\quad
    \text{for every }n\in \N.
  \end{equation}
  Possibly adding $M_1$ to $M_3$,
  it is not restrictive
  to assume that $\upbeta'(r)\ge1$.
  We can then apply Lemma \ref{le:slowly-increasing-entropy}
  and we can find a smooth convex superlinear function
  $\upomega:\R_+\to\R_+$ such that \eqref{eq:41} holds.
  In particular
  \begin{align}
    \label{eq:59}
    \int_V \upomega(u_a^n)\,\dd\pi&\le M_1,\\
    \notag
      \int_a^b \iint_\edg \rmD^-_\upomega(u_r^n(x),u_r^n(y))\,\tetapi(\dd
      x,\dd y)\,\dd r&\le \int_a^b \iint_\edg
      (u_r^n(x)+u^n_r(y))\,\tetapi(\dd x,\dd y)\,\dd r
      \\&\leq  2(b-a)M_1\|\kappa_V\|_\infty.\label{eq:60}
  \end{align}
  By Corollary~\ref{th:chain-rule-bound2} we obtain 
  \begin{equation}
    \label{eq:61}
    \int_V \upomega(u_t^n)\,\dd\pi\le M+(b-a)M_1 \|\kappa_V\|_\infty + M_1\quad
    \text{for every }t\in [a,b].
  \end{equation}
  By \eqref{est:ct-eq-TV} we deduce that
  \begin{equation}
    \label{eq:65}
    \|u^n_t-u^n_s\|_{L^1(V,\pi)}\le \zeta(s,t)\quad
    \text{where}\quad
    \zeta(s,t): = 2\sup_{n\in \N} \int_{s}^{t} |\bj_r^n|(\edg)\,\dd r \,.
  \end{equation}
 Since $t\mapsto |\bj^n_t|(\edg)$ is equi-integrable we have  
 \[
 \lim_{(s,t)\to (r,r)} \zeta(s,t) =0 \qquad \text{for all } r \in [a,b]\,,
 \]
We conclude that the sequence of maps $(u_t^n)_{t\in [a,b]}$ satisfies 
 the conditions of the compactness result \cite[Prop.\ 3.3.1]{AmbrosioGigliSavare08}, which yields the existence of a (not relabelled) subsequence  and of a
 $L^1(V,\pi)$-continuous (thus also weakly-continuous)
 function $[a,b]\ni t \mapsto u_t\in L^1(V,\pi)$
 such that $u^n_t\weakto u_t$ weakly in $L^1(V,\pi)$ for every
 $t\in [a,b]$. By \eqref{eq:72} 
  we also deduce that
 \eqref{converg-rho-n} holds, i.e.
\[
\rho_t^n\to \rho_t=u_t\pi\quad\text{setwise in $\calM(V)$ for all $t\in[a,b]$}.
\] 
It is also clear that for every $t\in [a,b]$ we have
$\tetapi_{\rho_t^n}^\pm\to\tetapi_{\rho_t}^\pm$ setwise.
The Dominated Convergence Theorem and~\eqref{eq:70}, \RICKYNEW \eqref{eq:85} \EEE imply that the corresponding measures
$\tetapi_{\rho^n,\Lebone}^\pm$ converge setwise  to
$\tetapi_{\rho,\Lebone}^\pm$, and are therefore equi-absolutely continuous
with respect to ~$\tetapi_\Lebone=\Lebone\otimes\tetapi$
(recall \eqref{eq:73bis}).

\GGG Let us now show that also the sequence $(\bj^n_\Lebone)_{n}$ is
equi-absolutely continuous
with respect to ~$\tetapi_\Lebone$, so that \eqref{converg-j-n}
holds up to extracting a further subsequence.

Selecting a constant $c>0$ sufficiently large so that
$\upalpha(u_1,u_2)\le c(1+u_1+u_2)$,
the trivial estimate $\bnu_\rho\le
c(\tetapi+\teta_\rho^-+\teta_\rho^+)$
 and the monotonicity property \eqref{eq:80}
 yield \GGG
\begin{equation}
  \label{eq:66}
  M\ge \int_a^b\calR(\rho^n_t,\bj^n_t )\,\dd t=
  \frac12\calF_\Psi(2\bj^n_\Lebone |\bnu_{\rho^n_\Lebone})\ge
  \calF_\Psi(\bj^n_\Lebone |\bsigma^n ),\
  \bsigma^n :=c(\tetapi_\Lebone+\tetapi_{\rho^n,\Lebone}^++\tetapi_{\rho^n,\Lebone}^-).
\end{equation}
%
For every $B\in \frA\otimes \frB$, $\frA$ being
the Borel $\sigma$-algebra of $[a,b]$, with $\tetapi_\Lebone(B)>0$,
Jensen's inequality \eqref{eq:79} yields
\begin{equation}
  \label{eq:111}
  \Psi\biggl(\frac{\bj_\Lebone^n(B)}{\bsigma^n(B)}\biggr) \bsigma^n(B)\le
  \calF_\Psi(\bj_\Lebone^n\mres B|\bsigma^n\mres B)\le M.
\end{equation}
Denoting by $U:\R_+\to\R_+$ the inverse function of $\Psi$,
we thus find
\begin{equation}
  \label{eq:112}
  \bj_\Lebone^n(B)\le \bsigma^n(B)\,U\biggl(\frac{M}{\bsigma^n(B)}\biggr).
\end{equation}
Since $\Psi$ is superlinear, $U$ is sublinear so that
\begin{equation}
  \label{eq:118}
  \lim_{\delta\downarrow0}\delta U(M/\delta)=0.
\end{equation}
For every $\eps>0$ there exists 
$\delta_0>0$ such that $\delta U(M/\delta)\le \eps$ for every
$\delta\in (0,\delta_0)$. Since $\bsigma^n$ is equi absolutely
continuous with respect to~$\tetapi_\Lebone$ we can also find $\delta_1>0$ such
that $\tetapi_\Lebone (B)<\delta_1$ yields
$\bsigma^n(B)\le \delta_0$. By \eqref{eq:112} we eventually conclude
that
$\bj^n_\Lebone(B)\le \eps$. \EEE
%

It is then easy to pass to the limit in the integral formulation
\eqref{2ndfundthm} of the continuity equation.
Finally, concerning \eqref{ineq:lsc-R},
it is sufficient to use the equivalent representation given by \eqref{eq:96}.
\end{proof}
\nc

\subsection{Definition and properties of the cost}\label{sec:cost}
We now define the Dynamical-Variational Transport cost $\mathcal{W} : (0,\pinfty) \times \calM^+(V)\times \calM^+(V) \to [0,\pinfty)$ by
\begin{equation}
\label{def-psi-rig}
\DVT \tau{\rho_0}{\rho_1} : = \inf\left\{ \int_0^\tau \calR(\rho_t,\bj_t) \,\dd t \, : \,  (\rho,\bj ) \in  \CEP 0\tau{\rho_0}{\rho_1}   \right\}\,. 
\end{equation}
In studying the properties of 
$\calW$, we will also often use the notation
\begin{equation}
\label{adm-curves}
\ADM 0\tau{\rho_0}{\rho_1}: = \biggl\{
(\rho,\bj )\in\CER 0{\tau}\, : \nc
\ \rho(0)=\rho_0, \ \rho(\tau) = \rho_1
\biggr\}\,,
\end{equation}
with $\CER 0\tau$ the class from \eqref{def:Aab}. 

For given $\tau>0$ and $\rho_0,\, \rho_1 \in  \calM^+(V)$,  if the set $\ADM 0\tau{\rho_0}{\rho_1}$ is non-empty, then it contains an exact minimizer for $\DVT {\tau}{\rho_0}{\rho_1}$. This is stated by the following result that is a direct consequence of Proposition~\ref{prop:compactness}.   \EEE
\begin{cor}[Existence of minimizers]
\label{c:exist-minimizers}
If  $\rho_0,\rho_1\in \calM^+(V)$
and $\ADM 0\tau{\rho_0}{\rho_1} $ is not empty, then
the infimum in~\eqref{def-psi-rig} is achieved. 
\end{cor}

\begin{remark}[Scaling invariance]
  \label{rem:scaling-invariance}
  Let us consider the perspective function 
  $\hat \Psi(r,s)$ associated wih $\Psi$ as
  in \eqref{eq:76}, $\hat \Psi(r,s)=s\Psi(r/s)$ if $s>0$.
  We call $\calR_s(\rho,\bj)$ the
  dissipation functional induced by $\hat{\Psi}(\cdot, s)$, with induced Dynamic-Transport cost
  $\mathscr W_s$.
  For every $\tau>0$, $\rho_0,\rho_1\in \calM^+(V)$ a rescaling argument yields
  \begin{equation}
    \label{eq:188}
    \DVTn(\tau,\rho_0, \rho_1) =\mathscr
    W_{\tau/\sigma}(\sigma,\rho_0,\rho_1)
    =
    \inf\left\{ \int_0^{\sigma} \calR_{\tau/\sigma}(\rho_t,\bj_t) \,\dd
      t \, : \,
      (\rho,\bj ) \in   \CEP 0{\sigma}{\rho_0}{\rho_1}  \EEE \right\}\,. 
  \end{equation}
  In particular, choosing $\sigma=1$  we find \GGGO
  \begin{equation}
    \label{eq:188bis}
    \DVTn(\tau,\rho_0, \rho_1) =\mathscr
    W_{\tau}(1,\rho_0,\rho_1).
  \end{equation}
  Since $\hat\Psi(\cdot,\tau)$ is convex, lower semicontinuous,
  and decreasing with respect to ~$\tau$,
  we find that $\tau\mapsto \DVTn(\tau,\rho_0, \rho_1) $
  is decreasing and convex as well.
%
\end{remark}
\par

Currently,  proving that  \emph{any} pair of measures can be connected by a curve with finite action $\int \calR$   under general conditions on $V$, $\Psi$ and $\upalpha$
is an open problem: in other words, in the general case we cannot exclude that  
 $\ADM 0\tau{\rho_0}{\rho_1} = \emptyset$, which would make $\DVT {\tau}{\rho_0}{\rho_1} = \pinfty$. 
Nonetheless,  in a more specific situation, Proposition \ref{prop:sufficient-for-connectivity} below provides 
sufficient conditions for this connectivity property, between two measures $\rho_0, \, \rho_1 \in  \calM^+(V) $ with the same mass and such that
$\rho_i \ll \pi$ for $i\in \{0,1\}$. Preliminarily, we give the following
\begin{definition}
Let $q\in (1,\pinfty)$.
We say that the measures $(\pi,\tetapi) $ satisfy a \emph{$q$-Poincar\'e inequality}  if 
there exists a constant $C_P>0$ such that 
for every $\xi \in L^q(V;\pi)$ with $\int_{V}\xi(x) \pi(\dd x) =0$ there holds
\begin{equation}
\label{q-Poinc}
\int_{V} |\xi(x)|^q \pi(\dd x) \leq C_P \int_{\edg} |\ona \xi(x,y)|^q \tetapi(\dd x, \dd y) .
\end{equation}
\end{definition}
We are now in a position to state the connectivity result, where we specialize the discussion to 
dissipation densities with $p$-growth for some $p \in (1,\pinfty)$.
\begin{prop}
\label{prop:sufficient-for-connectivity}
Suppose that 
\begin{equation}
\label{psi-p-growth}
\exists\,  p \in (1,\pinfty), \, \overline{C}_p>0 \ \ \forall\,r  \in \R  \, : \qquad \Psi(r) \leq \overline{C}_p(1{+}|r|^p),
\end{equation}
and that the measures $(\pi,\tetapi) $ satisfy a  $q$-Poincar\'e inequality for $q=\tfrac p{p-1}$. 
Let $\rho_0, \rho_1  \in  \calM^+(V) $ with the same mass be given by $\rho_i = u_i \pi$, with positive
$u_i \in L^1(V; \pi) \cap L^\infty (V; \pi) $, for $i \in \{0,1\}$. Then, for every $\tau>0$
the set   $\ADM 0\tau{\rho_0}{\rho_1} $ is non-empty and thus $\DVT \tau{\rho_0}{\rho_1}<\infty$.
\end{prop}
We postpone the proof of Proposition \ref{prop:sufficient-for-connectivity}
 to Appendix \ref{s:app-2}, where some preliminary results, also motivating the  role of the
 $q$-Poincar\'e inequality, will be provided.
 \EEE


\subsection{Abstract-level properties of \texorpdfstring{$\DVTn$}W}
\label{ss:4.5}
The main result of this section collects a series of properties
of the cost that will play a key role in 
the study of the \emph{Minimizing Movement} scheme   \eqref{MM-intro}. 
 Indeed, as already hinted in the Introduction, the analysis that we will carry out in Section  \ref{s:MM} ahead might well be extended to  
a scheme set up in a general topological space, endowed with a cost functional enjoying   properties \eqref{assW}
below. We will  now  check them for the cost $\DVTn$ associated  with 
 generalized gradient structure $(\calS,\calR,\calR^*)$ \EEE
 fulfilling \textbf{Assumptions~\ref{ass:V-and-kappa} and  \ref{ass:Psi}}.
In this section \emph{all convergences will be with respect to the setwise topology}.
\begin{theorem}
\label{thm:props-cost}
The cost $\DVTn$ enjoys the following properties:
\begin{subequations}\label{assW}
\begin{enumerate}[label={(\arabic*)}]
	\item \label{tpc:1} For all $\tau>0,\, \rho_0,\, \rho_1 \in \calM^+(V)$,
	\begin{equation}\label{e:psi2}
		\DVTn(\tau,\rho_0,\rho_1)= 0 \ \Leftrightarrow \ \rho_0=\rho_1.
	\end{equation}
	\item \label{tpc:2} For all 
 $\rho_1,\, \rho_2,\,\rho_3\in\calM^+(V)$  and  $\tau_1, \tau_2 \in (0,\pinfty)$ with $\tau=\tau_1
 +\tau_2$,
 	\begin{equation}\label{e:psi3}
 \DVTn(\tau,\rho_1,\rho_3) \leq  \DVTn(\tau_1,\rho_1,\rho_2) +  \DVTn(\tau_2,\rho_2,\rho_3).
 	\end{equation}
 	\item \label{tpc:3} For $\tau_n\to \tau>0, \ 
          \rho_0^n \to \rho,  \ \rho_1^n \to \rho_1$ in $\calM^+(V)$,
 	\begin{equation}\label{lower-semicont}
\liminf_{n \to \pinfty} \DVTn(\tau_n,\rho_0^n, \rho_1^n) \geq \DVTn(\tau,\rho_0,\rho_1).
 	\end{equation}
 	\item \label{tpc:4} For all $\tau_n
\downarrow 0$ and for all $(\rho_n)_n$, $ \rho \in \calM^+(V)$,
	\begin{equation}\label{e:psi4}
\sup_{n\in\N} \DVTn(\tau_n, \rho_n,\rho) <\pinfty \quad \Rightarrow \quad  \rho_n \to \rho.
	\end{equation}
	\item \label{tpc:5} For all $\tau_n
\downarrow 0$ and all $(\rho_n)_n$, $ (\nu_n)_n \subset \calM^+(V)$ with $\rho_n \to \rho, \ \nu_n \to\nu$,
	\begin{equation}\label{e:psi6}
\limsup_{n\to\infty} \DVTn(\tau_n, \rho_n,\nu_n) <\pinfty \quad \Rightarrow \quad \rho = \nu.
	\end{equation}
\end{enumerate}
\end{subequations}
\end{theorem}
\begin{proof}
\textit{\ref{tpc:1}} 
Since $\Psi(s)$ is strictly positive for $s\neq0$
it is immediate to check that $\calR(\rho,\bj)=0\ \Rightarrow\ \bj=0$.
For an optimal pair
$(\rho,\bj ) $ satisfying 
$
\int_0^\tau \calR(\rho_t,\bj_t)\,\dd t =0
$
we deduce that 
$\bj_t=0$
for a.e.~$t\in (0,\tau)$. The continuity equation then implies
$\rho_0=\rho_1$.
\nc

\medskip

\textit{\ref{tpc:2}} This can  easily be checked by using the existence of minimizers for $\DVTn(\tau,\rho_0, \rho_1)$. 

\medskip

\textit{\ref{tpc:3}}
Assume without loss of generality that $\liminf_{n\to+\infty} \DVTn(\tau_n,\rho_0^n, \rho_1^n) < \infty$. By \eqref{eq:188bis}
\nc we use that, for every $n\in \N$ and setting $\overline \tau = \sup_n \tau_n$, 
\[
\DVTn(\tau_n,\rho_n^0, \rho_n^1)   = \DVTn_{\tau_n} (1,\rho_n^0, \rho_n^1)
\leq \DVTn_{\overline \tau} (1,\rho_n^0, \rho_n^1)
  \stackrel{(*)}= \int_0^{1} \calR_{\overline \tau}(\rho_t^n,\bj_t^n) \,\dd t,
\]
where the identity $(*)$ holds for an optimal pair $(\rho^n,\bj^n) \in \CEP{0}{1}{\rho_0^n}{\rho_1^n}$.
Applying Proposition~\ref{prop:compactness}, we obtain the existence of $(\rho, \bj ) \in \CEP 01{\rho_0}{\rho_1}$ such that, up to a subsequence,  
\begin{equation}
\label{tilde-rho-j-n}
\begin{aligned}
&{\rho}_s^n \to {\rho}_s \text{ setwise in } \calM^+(V) \quad \text{for all } s \in [0,1]\,,\\
&{\bj}^n \to  {\bj} \text{ setwise in }  \calM(\RICKYNEW [0,1]\EEE{\times}\edg)\,,
\end{aligned}
\end{equation}
Arguing as in Proposition~\ref{prop:compactness}
and using the joint
lower semicontinuity of $\hat \Psi$, we find that 
\[
\liminf_{n\to\infty}  \int_0^{1}  \calR_{\tau_n}\left( {\rho}_s^n ,
  {\bj}_s^n \right)  \dd s
\geq  \int_0^{1}  \calR_\tau\left( {\rho}_s ,  {\bj}_s \right)
\dd s
\ge \mathscr W_\tau(1,\rho_0,\rho_1)=
\DVTn(\tau,\rho_0,\rho_1).
\]

\medskip

\textit{\ref{tpc:4}}
If we denote by $\calR_0$ the dissipation associated with
$\hat\Psi(\cdot,0)$, given by $\hat\Psi(w,0) = +\infty$ for $w\not=0$ and $\hat\Psi(0,0)=0$,  we find
\begin{equation}
  \label{eq:189}
  \calR_0(\rho,\bj)<\pinfty\quad\Rightarrow\quad
  \bj=0.
\end{equation}
By the same argument as for part~\ref{tpc:3},
every subsequence of $\rho_n$ has a
converging subsequence in the setwise topology; 
the lower semicontinuity result of the proof of part~\ref{tpc:3}
 shows that any limit point must coincide with $\rho$.

\medskip
\textit{\ref{tpc:5}}
The argument
combines \eqref{eq:189} and
part~\ref{tpc:3}.
\end{proof}

\subsection{The action functional \texorpdfstring{$\VarWn$}W and its properties}

The construction of $\calR$ and $\DVTn$ above proceeded in the order $\calR \rightsquigarrow \DVTn$: we first constructed $\calR$, and then $\DVTn$ was defined in terms of~$\calR$. It is a natural question whether one can invert this construction: given $\DVTn$, can one reconstruct~$\calR$, or at least integrals of the form $\int_a^b \calR\,\dd t$? The answer is positive, as we show in this section.

Given a functional $\DVTn$ satisfying the properties~\eqref{assW},
we define the `$\DVTn$-action'
of a curve $\rho:[a,b]\to\calM^+(V)$ as
\begin{equation}
\label{def-tot-var}
\VarW \rho ab: = \sup \left \{ \sum_{j=1}^M  
  \DVT{t^j - t^{j-1}}{\rho(t^{j-1})}{\rho(t^j)} \, : \ (t^j)_{j=0}^M \in \mathfrak{P}_f([a,b]) 
\right\} ,
\end{equation}
for all $[a,b]\subset [0,T]$
where $\mathfrak{P}_f([a,b])$  denotes the set of all partitions of a given interval $[a,b]$. 

If $\DVTn$ is defined by~\eqref{def-psi-rig}, then each term in the sum above is defined as an optimal version of $\int_{t^{j-1}}^{t^j} \calR(\rho_t,\cdot)\,\dd t$, and we might expect that $\VarW \rho ab$ is an optimal version of $\int_a^b \calR(\rho_t,\cdot)\,\dd t$. This is indeed the case, as is illustrated  by the following analogue of~\cite[Th.~5.17]{DolbeaultNazaretSavare09}:

\begin{prop}
\label{t:R=R}
Let $\DVTn$ be given by~\eqref{def-psi-rig}, and let $\rho:[0,T]\to \calM^+(V)$.
Then 
$\VarW \rho 0T<\pinfty$  if and only if there exists a
measurable map $\bj :[0,T]\to \calM(\edg)$ such that $(\rho,\bj )\in\CE 0T$ with $\int_0^T \calR(\rho_t,\bj_t)\dd t<\pinfty$\,.
In that case, 
\begin{equation}
\label{calR-leq-VarW}
\VarW \rho0T\leq \int_0^T \calR(\rho_t,\bj_t)\,\dd t,
\end{equation}
and there exists a unique $\bj_{\rm opt}$  such that equality is achieved.  
The optimal $\bj_{\rm opt}$ is skew-symmetric, i.e.\  $\bj_{\rm opt}= \tj_{\rm opt}$ \EEE
 (cf.~Remark~\ref{rem:skew-symmetric}).
 \end{prop}

Prior to proving Proposition \ref{t:R=R}, we establish  the following approximation result. 
\begin{lemma}
\label{l:convergence-of-interpolations}
Let $\rho:[0,T]\to \calM^+(V)$
 satisfy  $\VarW {\rho}0T<\pinfty$. 
 For a sequence of partitions $P_n =(t_n^j)_{j=0}^{M_n}\in \mathfrak{P}_f([0,T])$ with fineness $\tau_n = \max_{j=1,\ldots, M_n} (t_n^j{-}t_n^{j-1})$ converging to zero, let $\rho^n :[0,T]\to \calM^+(V)$ satisfy
\[
\rho^n(t_n^j) = \rho(t_n^j)\quad\text{for all } j =1,\ldots, M_n \qquad\text{and}\qquad 
\sup\nolimits_{n\in\N} \VarW {\rho^n}0T < \pinfty.
\]
Then $\rho^n(t) \to \rho(t)$  setwise \EEE for all $t\in[0,T]$ as $n\to\infty$.
\end{lemma}
\begin{proof}
First of all, observe that by the symmetry of $\Psi$, also the time-reversed curve    
 $\check \rho(t):= \rho(T-t)$ satisfies
  $\VarW {\check{\rho}}0T<\pinfty$.
   Let $\pwC {\sft}{n}$ and $\upwC {\sft}{n}$ be the  piecewise constant interpolants 
 associated with the partitions $P_n$, cf.\ \eqref{nodes-interpolants}. Fix $t\in [0,T]$; we estimate 
\begin{align*}
\DVTn\bigl(2(\pwC\sft n-t),\rho^n(t),\rho(t)\bigr) 
&\stackrel{(1)}{\leq} \DVTn\bigl(\pwC\sft n-t,\rho^n(t),\rho^n(\pwC {\sft}{n}(t))\bigr) +
  \DVTn\bigl(\pwC\sft n-t,\rho(\pwC {\sft}{n}(t)),\rho(t)\bigr) \\
&= \DVTn\bigl(\pwC\sft n-t,\rho^n(t),\rho^n(\pwC {\sft}{n}(t))\bigr) +
  \DVTn\bigl(\pwC\sft n-t,\check \rho(T-\pwC {\sft}{n}(t)),\check \rho(T-t)\bigr)\\
&\leq
\VarW {\rho^n}t{\pwC {\sft}{n}(t)} + \VarW{\check \rho}{T-\pwC {\sft}{n}(t)}{T-t}\\
&\leq \sup_{n\in\N} \VarW {\rho^n} 0T + \VarW{\check \rho}0T =: C<\pinfty,
\end{align*}
where (1) follows from property \eqref{e:psi3} of $\DVTn$.
Consequently, by property~\eqref{e:psi4} it follows that $\rho^n(t)
\to  \rho(t)$
setwise  in $\calM^+(V)$ for all $t\in[0,T]$.
\end{proof}

We are now in a position to prove Proposition~\ref{t:R=R}:

\begin{proof}[Proof of Proposition~\ref{t:R=R}]
One implication is straightforward: if a pair $(\rho,\bj )$ exists, then 
\[
\DVTn(t-s,\rho_s,\rho_t) \stackrel{\eqref{def-psi-rig}}\leq  
\int_s^t \calR(\rho_r,\bj_r)\,\dd r,\qquad \text{for all }0\leq s<t\leq T,
\]
and therefore  $\VarW \rho0T<\pinfty$ and \eqref{calR-leq-VarW} holds.

To prove the other implication, assume that   $\VarW \rho0T<\pinfty$. 
Choose a sequence of partitions $P_n =(t_n^j)_{j=0}^{M_n}\in \mathfrak{P}_f([0,T])$ that becomes dense in the limit $n\to\infty$. For each $n\in\N$, construct a pair $(\rho^n,\bj^n)\in \CE 0T$ as follows:
On each time interval $[t_n^{j-1},t_n^j]$, let $(\rho^n,\bj^n)$  be given by Corollary~\ref{c:exist-minimizers} as the minimizer under the constraint $\rho^n(t_n^{j-1}) = \rho(t_n^{j-1})$ and $\rho^n(t_n^j) = \rho(t_n^j)$,
namely
\begin{equation}
\label{minimizer-cost}
\DVT{t_n^{j}{-}t_n^{j-1}}{\rho(t_n^{j-1})}{ \rho(t_n^{j})} = \int_{t_n^{j-1}}^{t_n^{j}}\calR(\rho_r^n,\bj_r^n) \,\dd r\,.
 \end{equation}
 By concatenating the minimizers on each of the intervals a pair $(\rho^n,\bj^n)\in \CE0T$ is obtained,  thanks to Lemma  \ref{l:concatenation&rescaling}.
 By construction we have the property
\begin{align}
\label{eq:VarW-rhon-calR}
&\VarW{\rho^n}0T =
  \int_0^T \calR(\rho^n_t,\bj^n_t)\,\dd t.
\end{align}
Also by optimality we have 
\[
\VarW{\rho^n}{t_n^{j-1}}{t_n^j} 
= \DVTn\bigl(t_n^j-t_n^{j-1},\rho(t_n^{j-1}),\rho(t_n^j)\bigr)
\leq \VarW\rho{t_n^{j-1}}{t_n^j},
\]
which implies by summing that
\begin{equation}
\label{ineq:VarW-rhon-rho}
\VarW{\rho^n}0T\leq \VarW\rho0T.
\end{equation}
By Lemma~\ref{l:convergence-of-interpolations} we then find that $\rho^n(t)\to \rho(t)$ 
 setwise \EEE
as $n\to\infty$ for each $t\in[0,T]$.

Applying Proposition~\ref{prop:compactness}, we find that $\bj^n(\dd
t\,\dd x\,\dd y):= \bj_t^n(\dd x\,\dd y)\,\dd t$
 setwise \EEE
 converges along a
subsequence to a limit $\bj$. The limit  $
\bj$ can be disintegrated as $\bj(\dd t\,\dd x\,\dd y)
= \lambda(\dd t) \, \bj_t(\dd x\,\dd y) $  \EEE for a measurable family $(\bj_t)_{t\in[0,T]}$, and the pair  $(\rho,\bj)$ 
\EEE
is an element of $\CE0T$. In addition
 we have the lower-semicontinuity property
\begin{equation}
\label{ineq:lsc:j-tilde-j}
\liminf_{n\to\infty} \int_0^T \calR(\rho^n_t,\bj^n_t)\,\dd t
\geq \int_0^T \calR(\rho_t,\bj_t)\EEE\,\dd t.
\end{equation}
We then have the series of inequalities
\begin{align*}
\VarW\rho 0T &\stackrel{\eqref{ineq:VarW-rhon-rho}} \geq \limsup_{n\to\infty} \VarW{\rho^n}0T
\stackrel{\eqref{eq:VarW-rhon-calR}} = \limsup_{n\to\infty} \int_0^T \calR(\rho^n_t,\bj^n_t)\,\dd t\\
 & \stackrel{\eqref{ineq:lsc:j-tilde-j}}{\geq} \int_0^T \calR(\rho_t, \bj_t) \EEE\,\dd t 
\stackrel{\eqref{calR-leq-VarW}}{\geq} \VarW \rho0T,
\end{align*}
which implies that $\int_0^T \calR(\rho_t, \bj_t) \EEE \,\dd t = \VarW \rho0T$.

Finally, the uniqueness of  $\bj$ \EEE is a consequence of the strict convexity of $\Upsilon(u_1,u_2,\cdot)$, 
 cf.\ Lemma~\ref{lem:Upsilon-properties}. \EEE Similarly, the skew-symmetry of $\bj$ follows from the strict convexity of $\Upsilon(u_1,u_2,\cdot)$, the symmetry of $\Upsilon(\cdot,\cdot,w)$,  and the invariance of the continuity equation~\eqref{eq:ct-eq-def} under the   `skew-symmetrization' $\bj  \mapsto \tj$,  
 cf.\ Remark \ref{rem:skew-symmetric}. \EEE
\end{proof}

\section{The Fisher information \texorpdfstring{$\Fish $}{D} and the definition of solutions}
\label{s:Fisherinformation}

With the definitions and the properties that we established  in the previous section we have given a rigorous meaning to the first term in the functional $\mathscr L$ in~\eqref{eq:def:mathscr-L}. In this section we continue with the second term in the integral, often called \emph{Fisher information}, after the canonical version in diffusion problems~\cite{Otto01}.   Section \ref{ss:5.2} is devoted to\
\begin{enumerate}[label=(\alph*)]
\item A rigorous definition of the Fisher information $\Fish(\rho)$
  (Definition~\ref{def:Fisher-information}).
\end{enumerate}
In several practical settings, such as  the proof of existence that we
give in Section~\ref{s:MM},
it is important to have lower semicontinuity of $\Fish$: this  is proved in Proposition \ref{PROP:lsc}. 
\par
 We are then in a position to give
\begin{enumerate}[resume,label=(\alph*)]
\item a rigorous definition of solutions  to the $(\calS, \calR, \calR^*)$ system (Definition~\ref{def:R-Rstar-balance}).
\end{enumerate}
In Section~\ref{ss:def-sol-intro}  we explained that the Energy-Dissipation balance approach to defining solutions is based on the fact  that  $\mathscr L(\rho,\bj ) \geq 0$ for all $(\rho,\bj )$
by the validity of a suitable chain-rule inequality.
\begin{enumerate}[resume,label=(\alph*)]
\item A rigorous proof of this chain-rule inequality, involving  $\calR$ and $\Fish$,   is given in Corollary~\ref{cor:CH3}, which is based on Theorem~\ref{th:chain-rule-bound}).
\end{enumerate}
This establishes the inequality $\mathscr L(\rho,\bj ) \geq 0$.
 Hence, 
 we can rigorously deduce that 
 the opposite inequality $\mathscr L(\rho,\bj ) \leq 0$  characterizes the property that $(\rho,\bj )$ is a solution to the $(\calS, \calR, \calR^*)$ system.
 Theorem \ref{thm:characterization} provides an additional characterization  of this solution concept.
%
%
%
\par
Finally, in Sections~\ref{subsec:main-properties} and \ref{ss:5.4}, \EEE 
\begin{enumerate}[resume,label=(\alph*)]
\item \GGGO we prove existence, uniqueness and stability
  of solutions under suitable convexity/l.s.c.~conditions on of
  $\Fish$
  (Theorems \ref{thm:existence-stability} and \ref{thm:uniqueness}).
  We  also discuss their asymptotic behaviour and the role of
  the invariant measures $\pi$.
\end{enumerate}
Throughout this section we adopt \textbf{Assumptions~\ref{ass:V-and-kappa}, \ref{ass:Psi}, and~\ref{ass:S}}.

\subsection{The Fisher information \texorpdfstring{$\Fish $}D} 
\label{ss:Fisher}

Formally, the Fisher information is the second term in~\eqref{eq:def:mathscr-L},  namely
\[
\Fish(\rho) = \calR^*\Bigl(\rho,-\thalf \ona \upphi(u)\Bigr)
=
\GGG \frac12\nc
\iint_\edg \Psi^*\bigl( -\thalf(\upphi'(u(y))-\upphi'(u(x))\bigr) \bnu_\rho(\dd x \, \dd y),\qquad \rho = u\pi\, .
\]
In order to give a precise meaning
to this formulation when $\upphi$ is not differentiable at $0$ 
(as, for instance, in the case of the Boltzmann entropy
function~\eqref{logarithmic-entropy}),
we use the function $\rmD_\upphi$ defined in \eqref{eq:182}.
\begin{definition}[The Fisher-information functional $\Fish $]
\label{def:Fisher-information}
The Fisher information $\Fish: \mathrm{D}(\calS)\to [0,\pinfty]$ is defined as
\begin{equation}
\label{eq:def:D}
	\Fish (\rho) :=
        \displaystyle \frac12\iint_\edg \OrmD_\upphi\bigl(u(x),u(y)\bigr)\,
        \tetapi(\dd x\,\dd y)
        \qquad   \text{for } 
		\rho = u\pi\,.
              \end{equation}
\end{definition}
\nc
\begin{example}[The Fisher information in the quadratic and in the  $\cosh$ case]
 \label{ex:D}
For illustration we recall the two expressions for $\OrmD_\upphi$ from  Example~\ref{ex:Dpm} for the linear equation~\eqref{eq:fokker-planck} with quadratic and cosh-type potentials $\Psi^*$ :    \begin{enumerate}
        \item If $\Psi^*(s)=s^2/2$ , then 
        \[
        \OrmD_\upphi(u,v) = \begin{cases}
                  \frac{1}{2}(\log(u)-\log(v))(u-v)  & \text{if }  u,\, v>0,
                  \\
                  0 & \text{if } u=v=0,
                  \\
                  \pinfty & \text{if } u=0 \text{ and } v \neq 0, \text{ or vice versa}.
                  \end{cases}
                \]
        \item If $\Psi^*(s)=4\bigl(\cosh(s/2)-1\bigr)$, then 
        \[
          \OrmD_\upphi(u,v) = 2\Bigl(\sqrt{u}-\sqrt{v}\Bigr)^2\qquad \forall\,  (u,v) \in [0,\pinfty) \times [0,\pinfty). 
        \]
      \end{enumerate}
       These two  examples of
      $\OrmD_\upphi$ are convex.      
  \end{example}
  Let us discuss the lower-semicontinuity properties of $\Fish$.
   In accordance with the Minimizing-Movement approach carried out in Section \ref{ss:MM}, we will just be interested in lower semicontinuity of $\Fish$ along sequences with bounded
  energy $\calS$. Now,  \EEE
  since sublevels of the energy~$\calS$ are
  relatively compact with respect to setwise convergence (by part~\ref{cond:setwise-compactness-superlinear} of Theorem~\ref{thm:L1-weak-compactness}),
  there is no difference between
  narrow and setwise lower semicontinuity  of $\Fish$. \EEE
  \begin{prop}[\textbf{Lower semicontinuity of $\Fish $}]
    \label{PROP:lsc}
    \upshape
    Assume either that $\pi$ is purely atomic
    or that the function $\rmD_\upphi$ is convex
    on 
    $\R_+^2$. Then
    $\Fish$ is (sequentially) lower semicontinuous 
    with respect to setwise convergence, i.e.,\
    for all $(\rho^n)_n,\, \rho \in 
     \mathrm{D}(\calS) $ 
    \begin{equation}
      \label{lscD}
      \rho^n \to \rho \text{ setwise in } \calM^+(V)\quad
      \Longrightarrow \quad \Fish(\rho) \leq \liminf_{n\to\infty} \Fish(\rho^n)\,.
\end{equation}
\end{prop}

\begin{proof}
  When $\pi$ is purely atomic, setwise convergence implies pointwise
  convergence
  $\pi$-a.e.~for the sequence of the densities, so that \eqref{lscD}
  follows by Fatou's Lemma.

  A standard argument, still based on Fatou's Lemma, shows that
  the functional
  \begin{equation}
    \label{eq:98}
    u\mapsto \iint_\edg \OrmD_\upphi(u(x),u(y))\,\tetapi(\dd x,\dd y)
  \end{equation}
  is lower semicontinuous with respect to ~the strong topology in $L^1(V,\pi)$:
  it is sufficient to check that $u_n\to u$ in $L^1(V,\pi)$
  implies $(u_n^-,u_n^+)\to (u^-,u^+)$ in $L^1(\edg,\tetapi)$.
  If $\OrmD_\upphi$ is convex on $\R_+^2$, then the functional \eqref{eq:98}
  is also lower semicontinuous with respect to the weak topology in
  $L^1(V,\pi)$.
  On the other hand, since $\rho_n$ and $\rho$ are absolutely
  continuous
  with respect to ~$\pi$, $\rho_n\to\rho$ setwise if and only if
  $\dd\rho_n/\dd\pi\weakto \dd\rho/\dd\pi$ weakly in $L^1(V,\pi)$
  (see Theorem~\ref{thm:equivalence-weak-compactness}). \EEE
\end{proof}

\nc

\subsection{The definition of solutions: \texorpdfstring{$\calR/\calR^*$}{R/R*} Energy-Dissipation balance }
\label{ss:5.2}

We are now in a position to formalize the concept of solution.

\begin{definition}[$(\calS,\calR,\calR^*)$ Energy-Dissipation balance]
\label{def:R-Rstar-balance}
We say that a 
curve $\rho: [0,T] \to \calM^+(V)$ is a  solution
of the $(\calS,\calR,\calR^*)$ evolution system,
if it
\nc satisfies the {\em $(\calS,\calR,\calR^*)$
  Energy-Dissipation balance}:
\begin{enumerate}
\item $\calS(\rho_0)<\pinfty$;
\item There exists a measurable family $(\bj_t)_{t\in [0,T]} \subset \calM(\edg)$
such that $(\rho,j)\in \CE0T$ with
\begin{equation}
\label{R-Rstar-balance}
\int_s^t \left( \calR(\rho_r, \bj_r) + \Fish(\rho_r) \right) \dd r+ \calS(\rho_t)   = \calS(\rho_s)   \qquad \text{for all } 0 \leq s \leq t \leq T.
\end{equation}
\end{enumerate}
\end{definition}

\begin{remark}\label{rem:properties}
  \begin{enumerate}
  \item Since $(\rho,\bj)\in \CE 0T$, the curve $\rho$ is absolutely
    continuous with respect to the total variation distance.
  \item  The Energy-Dissipation balance \eqref{R-Rstar-balance} written for $s=0$
    and $t=T$ implies that $(\rho,\bj)\in \CER 0T$ as well.
    Moreover, $t\mapsto \calS(\rho_t)$
    takes finite values and it is absolutely continuous in the
    interval $[0,T]$.
\item The chain-rule estimate \eqref{eq:CR2} implies the following important corollary:
\begin{cor}[Chain-rule estimate III]
\label{cor:CH3}
For any curve $(\rho,\bj)\in \CE 0T$, 
\begin{equation}
\label{eq:CR3}	
      \mathscr L_T(\rho,\bj):=
      \int_0^T \left( \calR(\rho_r, \bj_r) + \Fish(\rho_r) \right) \dd
      r+ \calS(\rho_T)
      -\calS(\rho_0)\geq 0  .
\end{equation}
\end{cor}	
\noindent 
It follows that  the Energy-Dissipation balance
    \eqref{R-Rstar-balance} is equivalent to the Energy-Dissipation
    Inequality
    \begin{equation}
      \label{EDineq}
      \mathscr L_T(\rho,\bj)\leq 0.   
    \end{equation}
  \end{enumerate}
\end{remark}

Let us give an equivalent characterization of solutions to the
$(\calS,\calR,\calR^*)$  evolution system.
Recalling the definition~\eqref{eq:184} of the map $\rmF$ in the interior of $\R_+^2$
and the definition~\eqref{eq:102} of $\rmA_\upphi$,
we first note that $\rmF$ can be extended to a function
defined in $\R_+^2$ with values in the extended real line
$[-\infty,+\infty]$ by 
\begin{equation}
  \label{eq:101}
  \rfield(u,v):=
  \begin{cases}
    \big(\Psi^*)'\big(\rmA_\upphi(u,v)\big)\upalpha(u,v)&\text{if }\upalpha(u,v)>0,\\
    0&\text{if }\upalpha(u,v)=0.
  \end{cases}
\end{equation}
where we set $(\Psi^*)'(\pm\infty):=\pm\infty$. The function
$\rfield$ is 
 skew-symmetric.
\begin{theorem}
  \label{thm:characterization}
  A curve $(\rho_t)_{t\in [0,T]}$ in $\Dom(V)$
   is a solution of the $(\calS,\calR,\calR^*)$  system \EEE iff
  \begin{enumerate}[label=(\arabic*)]
  \item\label{thm:characterization-first} $\rho_t=u_t\pi\ll\pi$ for every $t\in [0,T]$ and
    $t\mapsto u_t$ is an absolutely continuous a.e.~differentiable map
    with values in $L^1(V,\pi)$;
  \item $\calS(\rho_0)<\pinfty$;
  \item \label{thm:characterization-finite-F0}
  We have
    \begin{equation}
      \int_0^T \iint_\edg |\rfield(u_t(x),u_t(y))|\,\tetapi(\dd x,\dd y)\,\dd
t<\pinfty;\label{eq:190}
    \end{equation}
    and
    \begin{equation}
      \label{eq:191}
      \rmD_\upphi(u_t(x),u_t(y))=\rmD^-_\upphi(u_t(x),u_t(y))\quad
      \text{for $\Lebone\otimes\tetapi$-a.e.~$(t,x,y)\in [0,T]\times \edg$}.
    \end{equation}
    In particular the complement $U'$
    of the set 
    \begin{equation}
    U:=\{(t,x,y)\in [0,T]\times\edg: \rfield(u_t(x),u_t(y))\in
    \R\}\label{eq:192}
  \end{equation}
    is $(\Lebone \otimes \tetapi)$-negligible
    and $\rfield$ takes finite values 
    $(\Lebone\otimes\tetapi)$-a.e.~in $[0,T]\times \edg$;
  \item\label{thm:characterization-last} Setting
    \begin{equation}
      \label{eq:183}
      2\bj_t(\dd x,\dd y)=-\rfield(u_t(x),u_t(y))\,\tetapi(\dd
      x,\dd y),
    \end{equation}
  we have   $(\rho,\bj)\in \CE 0T$. In particular,
    \begin{equation}
      \label{eq:179}
      \dot u_t(x)=\int_V \rfield(u_t(x),u_t(y))\,\kappa(x,\dd y)
      \quad\text{for $(\lambda \otimes \pi)$-a.e.~}(t,x,y)\in [0,T]\times \edg.
    \end{equation}
  \end{enumerate}
\end{theorem}
\begin{proof}
  Let $\rho_t=u_t\pi$ be a solution of the
  $(\calS,\calR,\calR^*)$ system \EEE  
  with the corresponding flux $\bj_t$.
  By Corollary \ref{cor:propagation-AC}
  we can find a skew-symmetric measurable map $\xi:(0,T)\times \edg\to
  \R$
  such that $\bj_\Lebone=\xi\upalpha(u^-,u^+)\Lebone\otimes\tetapi$
  and \eqref{eq:42}, \eqref{eq:45} hold.
   Taking into account that  $\rmD^-_\upphi\le \rmD_\upphi$ and 
  applying the equality case of Corollary \ref{th:chain-rule-bound2},  
 we complete the
  proof
  of one implication.

  Suppose  now that $\rho_t$ satisfies all the above conditions \ref{thm:characterization-first}--\ref{thm:characterization-last};
we want to apply formula \eqref{eq:CR} of Theorem
\ref{th:chain-rule-bound}
for $\upbeta=\upphi$. For this we write the shorthand $u^-,u^+$ for $u_t(x),u_t(y)$  and set $w=-\rfield(u^-,u^+)$. We verify the equality conditions~\eqref{eq:109}  of Lemma~\ref{le:trivial-but-useful}:
\begin{itemize}
	\item At $(t,x,y)$ where  $\upalpha(u^-,u^+) = 0$, we have  by definition $w = -\rmF_0(u^-,u^+)=0$;
	\item At $(\lambda\otimes\teta)$--a.e.\ $(t,x,y)$ where $\upalpha(u^-,u^+) >0$,   $\rmF_0(u^-,u^+)$ is finite by condition~\ref{thm:characterization-finite-F0}, and by~\eqref{eq:101} it follows that $(\Psi^*)'\bigl(\rmA_\upphi(u^-,u^+)\bigr)$ is finite and therefore $\rmA_\upphi(u^-,u^+)$ is finite. The final condition $-w=(\Psi^*)'\big(\rmA_\upbeta(u,v)\big)\upalpha(u,v)$ then follows by the definition of $w$.
\end{itemize}
By Lemma~\ref{le:trivial-but-useful} therefore we have at  $(\lambda\otimes\teta)$--a.e.\ $(t,x,y)$
\begin{align*}
   -\rmB_\upphi(u^-,u^+,w) =
    \Upsilon(u^-,u^+,-w)+\rmD_\upphi^-(u^-,u^+)
    \stackrel{\eqref{eq:191}}= \Upsilon(u^-,u^+,-w)+\rmD_\upphi(u^-,u^+).
 \end{align*}
In particular $\rmB_\upphi$ is nonpositive, and the integrability condition \eqref{ass:th:CR} is trivially satisfied. Integrating~\eqref{eq:CR} in time we find~\eqref{R-Rstar-balance}.
\end{proof}

\begin{remark}
\label{rmk:why-interesting-1}
\upshape
By Theorem \ref{thm:characterization}(3), along  a solution  $\rho_t = u_t \pi$ of the $(\calS, \calR, \calR^*)$ system, the functions $ \rmD_\upphi$ and
$\rmD^-_\upphi$ coincide. Recall that, in general, we only have $ \rmD_\upphi^- \leq  \rmD_\upphi$, and the inequality can be strict, as in the examples of the linear equation \eqref{eq:fokker-planck} with the Boltzmann entropy and the quadratic and $\cosh$-dissipation potentials discussed in Ex.\ \ref{ex:Dpm}. There,
$ \rmD_\upphi$ and
$\rmD^-_\upphi$  differ on the boundary of $\R^2$. Therefore, \eqref{eq:191} encompasses the  information that the pair $(u_t(x),u_t(y))$ stays in the interior of $\R^2$
$(\lambda{\otimes}\tetapi)$-a.e.\ in $[0,T]\times \edg$. 
\end{remark} \EEE

\subsection{Existence and uniqueness of solutions of the
  $(\calS,\calR,\calR^*)$ system}
\label{subsec:main-properties}
Let us now collect a few basic structural properties of solutions
of the  $(\calS,\calR,\calR^*)$ Energy-Dissipation balance.
Recall that we will always adopt
\textbf{Assumptions~\ref{ass:V-and-kappa}, \ref{ass:Psi},
  and~\ref{ass:S}}.

Following an argument by Gigli~\cite{Gigli10} we first use convexity of
$\Fish$
to deduce uniqueness.

\begin{theorem}[Uniqueness]
 \label{thm:uniqueness}
Suppose that $\Fish$ is convex and the energy density $\upphi$ is strictly convex.
Suppose that $\rho^1,\, \rho^2$ satisfy   the $(\calS,\calR,\calR^*)$
Energy-Dissipation balance \eqref{R-Rstar-balance} and
are identical at time zero. Then $\rho_t^1 = \rho_t^2$ for every $t\in [0,T]$.  
 \end{theorem}
 \begin{proof}
   Let $\bj^i\in \calM((0,T)\times \edg)$ satisfy
   $\mathscr L_t(\rho^i,\bj^i)=0$
   and let us set
   \begin{displaymath}
     \rho_t:=\frac 12(\rho_t^1+\rho_t^2),\quad
     \bj:=\frac12(\bj^1+\bj^2).
   \end{displaymath}
   By the linearity of the continuity equation 
    we have that 
   $(\rho,\bj)\in \CE 0T$
   with $\rho_0=\rho^1_0=\rho^2_0$, so that
   by convexity 
   \begin{align*}
     \calS(\rho_t)
     &\ge\calS(\rho_0)-
       \int_0^t \left( \calR(\rho_r, \bj_r) + \Fish(\rho_r) \right) \dd
       r
       \\&\ge
     \calS(\rho_0)-
     \frac12\int_0^t \left( \calR(\rho^1_r, \bj^1_r) + \Fish(\rho^1_r) \right) \dd
     r
     -
     \frac12\int_0^t \left( \calR(\rho^2_r, \bj^2_r) + \Fish(\rho^2_r) \right) \dd
     r
     \\&
     =\frac 12\calS(\rho^1_t)+\frac12\calS(\rho^2_t).
   \end{align*}
   Since $\calS$ is strictly convex we deduce $\rho^1_t=\rho^2_t$.
\end{proof}

\begin{theorem}[Existence and stability]
  \label{thm:existence-stability}
  Let us suppose that the Fisher information functional $\Fish$
  is lower semicontinuous with respect to ~setwise convergence
  (e.g.~if $\pi$ is purely atomic, or $\rmD_\upphi$ is convex, see
  Proposition \ref{PROP:lsc}).
  \begin{enumerate}[label=(\arabic*)]
  \item \label{thm:existence-stability:p1} For every $\rho_0\in \calM^+(V)$ with
    $\calS(\rho_0)<\pinfty$
    there exists a solution $\rho:[0,T]\to \calM^+(V)$ of the
    $(\calS,\calR,\calR^*)$ evolution system starting from $\rho_0$.
  \item \label{thm:existence-stability:p2}
    Every sequence $(\rho^n_t)_{t\in [0,T]}$ of solutions to the
    $(\calS,\calR,\calR^*)$ evolution system such that
    \begin{equation}
      \label{eq:195}
      \sup_{n\in \N}      \calS(\rho^n_0)<\pinfty
    \end{equation}
    has a subsequence setwise converging to a limit $(\rho_t)_{t\in
      [0,T]}$
    for every $t\in [0,T]$.
  \item \label{thm:existence-stability:p3}
    Let $(\rho^n_t)_{t\in [0,T]}$ is a sequence of solutions, with corresponding fluxes $(\bj^n_t)_{t\in[0,T]}$. 
    Let $\rho^n_t$ converge setwise  to $\rho_t$ for every $t\in
    [0,T]$,  and assume that 
    \begin{equation}
    \lim_{n\to\infty}\calS(\rho^n_0)=\calS(\rho_0).\label{eq:194}
  \end{equation}
  Then $\rho$  is a solution
  as well, with flux $\bj$, and the following additional convergence properties hold:
  \begin{subequations}
  	    \label{eq:196}
  \begin{align}
  \label{eq:196a}
      \lim_{n\to\infty}\int_0^T\calR(\rho_t^n,\bj_t^n)\,\dd t&=
      \lim_{n\to\infty}\int_0^T\calR(\rho_t,\bj_t)\,\dd t,\\
      \label{eq:196b}
      \lim_{n\to\infty}\int_0^T\Fish(\rho_t^n)\,\dd t&=
      \lim_{n\to\infty}\int_0^T\Fish(\rho_t,\bj_t)\,\dd t,\\
      \label{eq:196c}
      \lim_{n\to\infty}\calS(\rho^n_t)&=\calS(\rho_t)\quad
      \text{for every }t\in [0,T].
    \end{align}
  \end{subequations}
  If moreover $\calS$ is strictly convex then $\rho^n$
  converges uniformly in $[0,T]$
  with respect to the total variation distance.
  \end{enumerate}
\end{theorem}

\begin{proof}
Part \textit{\ref{thm:existence-stability:p2}} follows immediately
  from Proposition \ref{prop:compactness}.

For part \textit{\ref{thm:existence-stability:p3}}, 
the three statements of~\eqref{eq:196} \emph{as inequalities} \RICKYNEW $\leq$ \EEE follow from earlier results: for~\eqref{eq:196a} this follows again from Proposition~\ref{prop:compactness}, for~\eqref{eq:196b} from Proposition~\ref{PROP:lsc}, and for~\eqref{eq:196c} from Lemma~\ref{l:lsc-general}.
Using these inequalities to pass to the limit in the equation $\mathscr L_T(\rho^n,\bj^n)=0$ we  obtain that  $\mathscr L_T(\rho,\bj)\le 0$.
  On the other hand, since $\mathscr L_T(\rho,\bj)\geq0$
   by the chain-rule estimate~\eqref{eq:CR3}, \RICKYNEW standard arguments yield   \EEE  the equalities in~\eqref{eq:196}.

   When  $\calS$  is strictly convex, we obtain
   the convergence in $L^1(V,\pi)$ of the densities $u^n_t=\dd
   \rho^n_t/\dd\pi$ for every $t\in [0,T]$.
   We then use the equicontinuity estimate \eqref{eq:65} of
   Proposition
   \ref{prop:compactness}  to conclude uniform convergence of the sequence $(\rho_n)_n$ with respect to  the total variation distance. 

For part \textit{\ref{thm:existence-stability:p1}}, when the density $u_0$ of $\rho_0$ takes value in a compact
   interval $[a,b]$ with $0<a<b<\infty$, the existence of a solution
   follows by Theorem \ref{thm:sg-sol-is-var-sol} below.
   The general case follows by a standard approximation
   of $u_0$ by truncation and applying the stability properties of
   parts \textit{\ref{thm:existence-stability:p2}}
and \textit{\ref{thm:existence-stability:p3}}.
\end{proof}

\subsection{Stationary states and attraction}
\label{ss:5.4}
Let us finally make a few comments on stationary measures and
on the asymptotic behaviour of solutions of the $(\calS,\calR,\calR^*)$ system.
The definition of invariant measures was already given in Section~\ref{subsub:kernels}, and we recall it for convenience. 
\begin{definition}[Invariant and stationary measures]
Let $\rho=u\pi\in D(\calS)$ be given. 
\begin{enumerate}
\item We say that $\rho$ is \emph{invariant} if $\kernel\kappa\rho(\dd x\dd y )= \rho(\dd x)\kappa(x,\dd y)$ has equal marginals, i.e.\ $\sfx_\# \kernel\kappa\rho = \sfy_\# \kernel\kappa\rho$.
	\item 
  We say that $\rho$ is \emph{stationary}
  if the constant curve
  $\rho_t\equiv \rho$ is a solution of the
  $(\calS,\calR,\calR^*)$ system.
\end{enumerate}
\end{definition}

Note that we always assume that $\pi$ is invariant (see Assumption~\ref{ass:V-and-kappa}).
It is immediate to check that
\begin{align}
  \label{eq:197}
  \rho\text{ is stationary}\quad
  \Longleftrightarrow\quad
  \Fish(\rho)=0
  \quad
  &\Longleftrightarrow\quad
  \rmD_\upphi(u(x),u(y))=0\quad\text{$\tetapi$-a.e.}
\end{align}
If a measure $\rho$ is invariant, then $u=\dd \rho/\dd\pi$ satisfies
\begin{equation}
  \label{eq:198}
  u(x)=u(y)\quad\text{for $\tetapi$-a.e.~$(x,y)\in \edg$},
\end{equation}
which implies~\eqref{eq:197}; therefore invariant measures are stationary. Depending on the system, the set of stationary measures might also contain non-invariant measures, as the next example shows.
%
%
%
%
%
%
%

\begin{example}
Consider the example of the cosh-type dissipation~\eqref{choice:cosh},
\[
\upalpha(u,v) := \sqrt{uv}, \quad\Psi^*(\xi) := 4\Bigl(\cosh\frac\xi2-1\Bigr), 
\]
but combine this with a Boltzmann entropy with an additional multiplicative constant $0<\gamma\leq 1$:
\[
\upphi(s) := \gamma(s\log s - s + 1).
\]
The case $\gamma=1$ corresponds to the example of~\eqref{choice:cosh}, and for general $0<\gamma\leq 1$ we find that 
\[
\rmF(u,v) = u^{\frac{1-\gamma}2}v^{\frac{1+\gamma}2}
- u^{\frac{1+\gamma}2}v^{\frac{1-\gamma}2},
\]
resulting in the evolution equation (see~\eqref{eq:180})
\[
\partial_t u(x) = \int_{y\in V} \Bigl[u(x)^{\frac{1-\gamma}2}u(y)^{\frac{1+\gamma}2}
- u(x)^{\frac{1+\gamma}2}u(y)^{\frac{1-\gamma}2}\Bigr]\, \kappa(x,\dd y).
\]
When $0<\gamma<1$, any function of the form  $u(x) = \One\{x\in A\}$ for $A\subset V$ is a stationary point of this equation, and equivalently any measure $\pi \mres A$ is a stationary solution of the $(\calS,\calR,\calR^*)$ system. For $0<\gamma<1$ therefore the set of stationary measures is much larger than just invariant measures.
\end{example}

As in the case of linear evolutions,
$(\calS,\calR,\calR^*)$ systems 
behave well with respect to decomposition of $\pi$ into mutually
singular invariant measures.

\begin{theorem}[Decomposition]
\label{thm:decomposition}
  Let us suppose that $\pi=\pi^1+\pi^2$ with $\pi^1,\pi^2\in
  \calM^+(V)$
  mutually singular and invariant.
  Let $\rho:[0,T]\to\calM^+(V)$ be a curve with $\rho_t=u_t\pi\ll\pi$
  and let $\rho^i_t:=u_t\pi^i$ be the decomposition of $\rho_t$
  with respect to ~$\pi^1$ and $\pi^2$.
  Then $\rho$ is a solution of the $(\calS,\calR,\calR^*)$ system
  if and only if each curve $\rho^i_t$, $i=1,2$, is a solution of
  the $(\calS^i,\calR^i,(\calR^i)^*)$ system, where
  $\calS^i(\mu):=\calF_\upphi(\mu|\pi^i)$
  is the relative entropy with respect to the measures $\pi^i$
  and
  and $\calR^i,(\calR^i)^*$ are induced by $\pi^i$.
\end{theorem}
\begin{remark}
  It is worth noting that
  when $\upalpha$ is $1$-homogeneous then $\calR^i=\calR$ and
  $(\calR^i)^*=\calR^*$ do not depend on $\pi^i$, cf.\ Corollary \ref{cor:decomposition}. 
  The decomposition is thus driven just by the splitting of the
  entropy $\calS$.
\end{remark}
\begin{proof}[Proof of Theorem~\ref{thm:decomposition}]
Note that the assumptions of invariance and mutual singularity of $\pi^1$ and $\pi^2$ imply that~$\tetapi$ has a singular decomposition $\teta = \teta^1 + \teta^2 := \kernel\kappa{\pi^1} + \kernel\kappa{\pi^2}$, where the $\kernel\kappa{\pi^i}$ are symmetric.
It then follows that   $\calS(\rho_t)=\calS^1(\rho^1_t)+\calS^2(\rho^2_t)$ and
  $\Fish(\rho_t)=\Fish^1(\rho^1_t)+\Fish^2(\rho^2_t)$,
  where
  \begin{displaymath}
    \Fish^i(\rho^i)=\frac12\iint_\edg
    \rmD_\upphi(u(x),u(y))\,\tetapi^i(\dd x,\dd y).
  \end{displaymath}
  Finally, Corollary~\ref{cor:decomposition} shows that
  decomposing $\bj$ as the sum $\bj^1+\bj^2$ where
  $\bj^i\ll\tetapi^i$, the pairs
  $(\rho^i,\bj^i)$ belong to $\CE 0T$ and
  $\calR(\rho_t,\bj_t)=
  \calR^1(\rho^1_t,\bj^1_t)+
  \calR^2(\rho^2_t,\bj^2_t)$.
\end{proof}

\begin{theorem}[Asymptotic behaviour]
  Let us suppose that the only stationary measures are multiples of $\pi$, and that $\Fish$ 
  is lower semicontinuous with respect to setwise convergence.
  Then every solution $\rho:[0,\infty)\to \calM^+(V)$ 
  of the $(\calS,\calR,\calR^*)$ evolution system
  converges setwise to $c\pi$, where 
  $c:=\rho_0(V)/\pi(V)$.
\end{theorem}
\begin{proof}
  Let us fix a vanishing sequence $\tau_n\downarrow0$
  such that $\sum_n\tau_n=\pinfty$.
  Let $\rho_\infty$ be any limit point
  with respect to ~setwise convergence of the curve $\rho_t$
  along a diverging sequence of times $t_n\uparrow\pinfty$.
  Such a point exists since
  the curve $\rho$ is contained in a sublevel set of $\calS$.
  Up to extracting a further subsequence, it is not restrictive to
  assume that $t_{n+1}\ge t_n+\tau_n$.

  Since
  \begin{align*}
    \sum_{n\in \N}\int_{t_n}^{t_n+\tau_n}\Big(\calR(\rho_t,\bj_t)+\Fish(\rho_t)\Big)\,\dd t
    \le \int_0^{\pinfty}\Big(\calR(\rho_t,\bj_t)+\Fish(\rho_t)\Big)\,\dd t\le \calS(\rho_0)<\infty    
  \end{align*}
  and the series of $\tau_n$ diverges,
  we find
  \begin{displaymath}
    \liminf_{n\to\pinfty}\frac1{\tau_n}\int_{t_n}^{t_n+\tau_n}\Fish(\rho_t)\,\dd
    t=0,\quad
    \lim_{n\to\infty}\int_{t_n}^{t_n+\tau_n}\calR(\rho_t,\bj_t)\,\dd t=0.
  \end{displaymath}
  Up to extracting a further subsequence, we can suppose that
  the above $\liminf$ is a limit and we can select
  $t'_n\in [t_n,t_n+\tau_n]$ such that
  \begin{displaymath}
    \lim_{n\to\infty}\Fish(\rho_{t_n'})=0,\quad
    \lim_{n\to\infty}\int_{t_n}^{t_n'}\calR(\rho_t,\bj_t)\,\dd t=0.
  \end{displaymath}
  Recalling the definition
  \eqref{def-psi-rig} of the Dynamical-Variational Transport cost and the monotonicity with respect to $\tau$,  we also get
  $\lim_{n\to\infty} \mathscr W(\tau_n,\rho_{t_n},\rho_{t_n'})=0$,
  so that Theorem \ref{thm:props-cost}(5) and the relative compactness
  of the sequence $(\rho_{t_n'})_n$ yield
  $\rho_{t_n'}\to \rho_\infty$ setwise. 

  The lower semicontinuity of $\Fish$ yields $\Fish(\rho_\infty)=0$
  so that $\rho_\infty=c\pi$ thanks to the uniqueness assumption
  and to the conservation of the total mass.
  Since we  have uniquely identified the limit point, we conclude that
  the \emph{whole} curve $\rho_t$ converges setwise to $\rho_\infty$  as $t\to\pinfty$.
\end{proof}


\section{Dissipative evolutions in $L^1(V,\pi)$}
\label{s:ex-sg}

In this section we construct solutions of the
$(\calS,\calR,\calR^*)$ formulation by studying their equivalent characterization as
abstract evolution equations in  $L^1(V,\pi)$. Throughout this section we adopt Assumption~\ref{ass:V-and-kappa}.

\subsection{Integro-differential equations in $L^1$}
Let $J\subset \R$ be a closed interval (not necessarily bounded)
and let us first consider a map
$\rmG:\edg\times J^2\to \R$ with the following properties:
\begin{subequations}
  \label{subeq:G}
  \begin{enumerate}
  \item measurability with respect to ~$(x,y)\in E$:
    \begin{equation}
      \label{eq:113}
      \text{for every $u,v\in J$ the map
      }
      (x,y)\mapsto \rmG(x,y;u,v)\text{ is measurable};
    \end{equation}
  \item continuity with respect to ~$u,v$ and linear growth:
    there exists a constant $M>0$ such that
    \begin{equation}
      \label{eq:114}
      \begin{gathered}
        \text{for every }(x,y)\in \edg\quad
        (u,v)\mapsto \rmG(x,y;u,v)\text{ is continuous and } \\
        |\rmG(x,y;u,v)|\le M(1+|u|+|v|)
        \quad \text{for every
        }u,v\in J,
      \end{gathered}
    \end{equation}
  \item skew-symmetry:
    \begin{equation}
    \label{eq:129}
    \rmG(x,y;u,v)=-\rmG(y,x;v,u),\quad
    \text{for every } (x,y)\in E,\ u,v\in J,
  \end{equation}
\item $\ell$-dissipativity:
  there exists a constant $\ell\ge0$
  such that for every $(x,y)\in E$, $u,u',v\in J$:
  \begin{equation}
    \label{eq:130}
    u\le u'\quad\Rightarrow\quad
    \rmG(x,y;u',v)-
    \rmG(x,y;u,v)\le \ell(u'-u).
  \end{equation}
  \end{enumerate}
\end{subequations}
\begin{remark}
  \label{rem:spoiler}
  Note that \eqref{eq:130} is surely
  satisfied if $\rmG$ is $\ell$-Lipschitz
  in $(u,v)$, uniformly with respect to ~$(x,y)$.
  The `one-sided Lipschitz condition'~\eqref{eq:130} however is weaker than the standard Lipschitz condition; this type of condition is common in the study of ordinary differential equations, since it is still strong enough to guarantee uniqueness and non-blowup of the solutions (see e.g.~\cite[Ch.~IV.12]{HairerWanner96}).
  
  Let us also remark that \eqref{eq:129} and \eqref{eq:130}
  imply the reverse monotonicity property of $\rmG$ with respect to ~$v$,
  \begin{equation}
    \label{eq:130bis}
    v\ge v'\quad\Rightarrow\quad
    \rmG(x,y;u,v')-
    \rmG(x,y;u,v)\le \ell(v-v')\,,
  \end{equation}
  and the joint estimate
  \begin{equation}
    \label{eq:120}
    u\le u',\ v\ge v'\quad\Rightarrow\quad
    \rmG(x,y;u',v')-
    \rmG(x,y;u,v)\le \ell\big[(u'-u)+(v-v')\big].
  \end{equation}
\end{remark}
Let us set 
$L^1(V,\pi;J):=\{u\in L^1(V,\pi):u(x)\in J\
\text{for $\pi$-a.e.~$x\in
  V$}\}$.
\begin{lemma}
  \label{le:tedious}
  Let $u:V\to J$ be a measurable $\pi$-integrable function.
  \begin{enumerate}
  \item We have
    \begin{equation}
      \label{eq:140}
      \int_V\big|\rmG(x,y;u(x),u(y))\big|\,\kappa(x,\dd y)<\pinfty
      \quad\text{for $\pi$-a.e.~$x\in V$},
    \end{equation}
    and the formula
    \begin{equation}
      \label{eq:115}
      \Gop[u](x):=
      \int_V \rmG(x,y;u(x),u(y))\,\kappa(x,\dd y)
    \end{equation}
    defines a function $\Gop[u]$ in $L^1(V,\pi)$
    that only depends on the Lebesgue equivalence class of $u$ in
    $L^1(V,\pi)$.
  \item The map $\Gop:L^1(V,\pi;J)\to L^1(V,\pi)$ is continuous.
  \item The map $\Gop$ is $(\ell\, \|\kappa_V\|_\infty \EEE)$-dissipative, in the sense that \OLI for all $h>0$,\EEE
    \begin{equation}
      \label{eq:141}
      \big\|(u_1- u_2)-h (\Gop[u_1]-\Gop[u_2])\big\|_{L^1(V,\pi)}\ge
      (1-\OLI 2 \EEE \ell  \|\kappa_V|_\infty \EEE \,h)\|u_1-u_2\|_{L^1(V,\pi)}
    \end{equation}
    for every $u_1,u_2\in L^1(V,\pi;J)$.
  \item If $a\in J$ satisfies
    \begin{equation}
      \label{eq:123a}
      0=\rmG(x,y;a,a)\le \rmG(x,y;a,v)\quad
      \text{for every }(x,y)\in \edg,\ v\ge a\,,
    \end{equation}
    then for every function $u\in  L^1(V,\pi;J)$
    we have
    \begin{equation}
      \label{eq:142}
      \begin{aligned}
        u\ge a\text{ $\pi$-a.e.}\quad
        &\Rightarrow\quad
        \lim_{h\downarrow0}\frac1h \int_V
        \Big(a-(u+h\Gop[u])\Big)_+\,\dd \pi=0\,.
      \end{aligned}
    \end{equation}
    If  $b\in J$ satisfies
    \begin{equation}
      \label{eq:123b}
      0=\rmG(x,y;b,b)\ge \rmG(x,y;b,v)\quad
      \text{for every }(x,y)\in \edg,\ v\le b,
    \end{equation}
    then for every function $u\in  L^1(V,\pi;J)$
    we have
    \begin{equation}
      u\le b\text{ $\pi$-a.e.}\quad
    \Rightarrow\quad
    \lim_{h\downarrow0}\frac 1h\int_V
    \Big(u+h\Gop[u]-b\Big)_+\,\dd \pi=0\,.\label{eq:145}
        \end{equation}
  \end{enumerate}
\end{lemma}

\begin{proof}
  \textit{(1)}
  Since $\rmG$ is a Carath\'eodory function,
  for every measurable $u$ and every $(x,y)\in E$ the map
  $(x,y)\mapsto \rmG(x,y;u(x),u(y))$ is measurable.
  Since
  \begin{equation}
    \label{eq:143}
    \begin{aligned}
      \iint_\edg |\rmG(x,y;u(x),u(y)|\,\kappa(x,\dd y)\pi(\dd x)&=
      \iint_\edg |\rmG(x,y;u(x),u(y)|\,\tetapi(\dd x,\dd y)
      \\&\le
      M  \|\kappa_V\|_\infty \EEE \bigg(1+2\int_V |u|\,\dd\pi\bigg)\,,
    \end{aligned}
  \end{equation}
  the first claim follows by Fubini's Theorem
  \cite[II, 14]{Dellacherie-Meyer78}.

\medskip

  \noindent
  \textit{(2)}
  Let $(u_n)_{n\in \N}$ be a sequence of functions strongly converging
  to $u$ in $L^1(V,\pi;J)$. Up to extracting a further subsequence,
  it is not restrictive to assume that $u_n$
  also converges to $u$ pointwise $\pi$-a.e.
  We have
  \begin{equation}
    \label{eq:144}
    \big\|\Gop[u_n]-\Gop[u]\big\|_{L^1(V,\pi)}=
    \iint_\edg \Big|\rmG(x,y;u_n(x),u_n(y))-
    \rmG(x,y;u(x),u(y))\Big|\,\tetapi(\dd x,\dd y)\, .
  \end{equation}
  Since the integrand $g_n$ in \eqref{eq:144} vanishes $\tetapi$-a.e.\ in
  $\edg$ as $n\to\infty$,
  by the generalized Dominated Convergence Theorem
  (see for instance \cite[Thm.\ 4, page 21]{Evans-Gariepy}
  it is sufficient to show that
  there exist positive functions $h_n$ pointwise converging to $h$
  such that
  \begin{displaymath}
    g_n\le h_n\ \tetapi\text{-a.e.~in $\edg$},\qquad
    \lim_{n\to\infty}\iint_\edg h_n\,\dd\tetapi=\iint_\edg h\,\dd\tetapi.
  \end{displaymath}
   We select
  \RICKYNEW $h_n(x,y):=M(2+|u_n(x)|+|u_n(y)|+|u(x)|+|u(y)|)$ and
  $ h(x,y):=2M(1+|u(x)|+|u(y)|)$. \EEE This proves the result.

\medskip

  \noindent
  \textit{(3)}
  Let us set
  \begin{displaymath}
    \frs(r):=
    \begin{cases}
      1&\text{if }r>0\,,\\
      -1&\text{if }r\le 0\,,
    \end{cases}
  \end{displaymath}
  \OLI and observe that the left-hand side of \eqref{eq:141} may be estimated from below by
  \begin{align*}
  	\big\|(u_1- u_2)-h (\Gop[u_1]-\Gop[u_2])\big\|_{L^1(V,\pi)} &\ge \|u_1-u_2\|_{L^1(V,\pi)} \\
  	&\hspace{2em}- h\int_V \frs(u_1-u_2)\big(\Gop[u_1]-\Gop[u_2]\big) \,\dd\pi
  \end{align*}
  for all $h>0$. Therefore, \EEE estimate \eqref{eq:141} follows if we prove that  
  \begin{equation}
    \label{eq:132}
    \delta:=\int_V \frs(u_1-u_2)\big(\Gop[u_1]-\Gop[u_2]\big) \,\dd\pi
    \le 2\ell  \|\kappa_V\|_\infty \EEE \, \|u_1-u_2\|_{L^1(V,\pi)}.
  \end{equation}
  Let us set
  \begin{displaymath}
    \Delta_\rmG(x,y):=
    \rmG(x,y;u_{1}(x),u_{1}(y))-
    \rmG(x,y;u_{2}(x),u_{2}(y)),
  \end{displaymath}
  and
  \begin{equation}
    \label{eq:133}
    \Delta_\frs(x,y):=\frs(u_{1}(x)-u_{2}(x))-\frs(u_{1}(y)-u_{2}(y)).
  \end{equation}
  Since $
    \Delta_\rmG(x,y)=-\Delta_\rmG(y,x)$, using \eqref{eq:129} we have
  \begin{align*}
   \delta= \int_V
      \frs\big(u_{1}-u_{2})\,\big(\Gop[u_{1}]-\Gop[u_{2}]\big)\,\dd\pi
    &=\iint_\edg
    \frs(u_{1}(x)-u_{2}(x))\Delta_\rmG(x,y)
      \,\tetapi(\dd
      x,\dd y)
    \\&=
    \frac12\iint_\edg
    \Delta_\frs(x,y)
    \Delta_\rmG(x,y)
      \,\tetapi(\dd
    x,\dd y)  .
  \end{align*}
  Setting $\Delta(x):=u_{1}(x)-u_{2}(x)$ we observe
  that by \eqref{eq:120} 
  \begin{align*}
    \Delta(x)>0,\ \Delta(y)>0\quad&\Rightarrow\quad
                                          \Delta_\frs (x,y)=0,\\
    \Delta(x)\le 0,\ \Delta(y)\le 0\quad&\Rightarrow\quad
                                                  \Delta_\frs (x,y)=0,\\
    \Delta(x)\le0,\ \Delta(y)>0
    \quad&\Rightarrow\quad
           \Delta_\frs (x,y)=-2,\
           \Delta_G(x,y)\ge-\ell\big(\Delta(y)-\Delta(x)\big)\\
    \Delta(x)>0,\ \Delta(y)\le 0\quad&\Rightarrow\quad
                                       \Delta_\frs (x,y)=2,\
                                           \Delta_G(x,y)\le \ell\big(\Delta(x)-\Delta(y)\big).
  \end{align*}
  We deduce that
  \begin{displaymath}
    \delta\le
    \ell \iint_\edg \Big[|u_{1}(x)-u_{2}(x)|+
    |u_{1}(y)-u_{2}(y)|\Big]\,\tetapi(\dd x,\dd y)\le
    2\ell  \|\kappa_V\|_\infty \EEE \,\|u_1-u_2\|_{L^1(V,\pi)}.
  \end{displaymath}
  
  \medskip
  
  \noindent
  \textit{(4)}
 We will only address the proof  of property  \eqref{eq:142},
as the argument for \eqref{eq:145} is completely analogous.  \EEE
  \OLI Suppose that $u\ge a$ $\pi$-a.e. \EEE Let us first observe that if $u(x)=a$, \OLI then from \eqref{eq:123a},\EEE
  \begin{displaymath}
    \Gop[u](x)=
    \int_V \rmG(x,y;a,u(y))\,\kappa(x,\dd y)\ge 0\,.
  \end{displaymath}
  We set $f_h(x):=h^{-1}(a-u(x))-\Gop[u](x)$,
  observing that $f_h(x)$ is monotonically decreasing to $-\infty$
  if $u(x)>a$ and $f_h(x)=-\Gop[u](x)\le 0$ if $u(x)=a$,
  so that $\lim_{h\downarrow0}\big(f_h(x)\big)_+=0$.
  Since $\big(f_h\big)_+\le \big(\!-\!\Gop[u]\big)_+$
  we can apply the Dominated Convergence Theorem \OLI to obtain \EEE
  \begin{displaymath}
    \lim_{h\downarrow0} \int_V \big(f_h(x)\big)_+\,\pi(\dd x)=0\,,
  \end{displaymath}
 \OLI thereby concluding the proof.\EEE
\end{proof}
%
 In what follows, we shall address the Cauchy problem \EEE
\begin{subequations}
\label{eq:119-Cauchy}
\begin{align}
  \label{eq:119}
  \dot u_t&=\Gop[u_t]\quad\text{in $L^1(V,\pi)$ for every }t\ge0,\\
  \label{eq:119-0}
  u\restr{t=0}&=u_0.
\end{align}
\end{subequations}
\begin{lemma}[Comparison principles]
  \label{le:positivity}
  Let us suppose that the map $\rmG$ satisfies
  {\rm (\ref{subeq:G}a,b,c)} with $J=\R$.
  \begin{enumerate}
  \item If $\bar u\in \R$ satisfies
    \begin{equation}
      \label{eq:123abis}
      0=\rmG(x,y;\bar u,\bar u)\le \rmG(x,y;\bar u,v)\quad
      \text{for every }(x,y)\in \edg,\ v\ge \bar u,
    \end{equation}
    then for every initial datum $u_0\ge\bar u$ the solution $u$ of
    \eqref{eq:119-Cauchy} satisfies $u_t\ge \bar u$ $\pi$-a.e.~for every $t\ge0$.
    \item If $\bar u\in \R$ satisfies
    \begin{equation}
      \label{eq:123bbis}
      0=\rmG(x,y;\bar u,\bar u)\ge \rmG(x,y;\bar u,v)\quad
      \text{for every }(x,y)\in \edg,\ v\le \bar u,
    \end{equation}
    then for every initial datum $u_0\le\bar u$ the solution $u$ of
    \eqref{eq:119-Cauchy} satisfies $u_t\le \bar u$ $\pi$-a.e.~for every $t\ge0$.
  \end{enumerate}
\end{lemma}

\begin{proof}
  \textit{(1)}
  Let us first consider the case $\bar u=0$.
  We define a new map $\overline \rmG$ by symmetry:
  \begin{equation}
    \label{eq:126}
    \overline \rmG(x,y;u,v):=\rmG(x,y;u,|v|)
  \end{equation}
  which satisfies the same structural properties
  (\ref{subeq:G}a,b,c), and moreover
  \begin{equation}
    \label{eq:127}
    0=\overline\rmG(x,y;0,0) \OLI \le \EEE \overline\rmG(x,y;0,v)\quad
    \text{for every } x,y\in V,\ v\in \R.
  \end{equation}
  We call $\overline\Gop$ the operator induced by $\overline\rmG$,
  and $\bar u$ the   solution curve \EEE of the corresponding
  Cauchy problem starting from the same
  (nonnegative) initial datum $u_0$.
  If we prove that $\bar u_t\ge0$ for every $t\ge0$,
  then $\bar u_t$ is also the unique solution of
  the original Cauchy problem \eqref{eq:119-Cauchy} induced by $\rmG$,
  so that we obtain the positivity of $u_t$.
  
  Note that \eqref{eq:127} \OLI and property \eqref{eq:130} \EEE yield
  \begin{equation}
    \label{eq:124}
    \overline\rmG(x,y;u,v)\ge
    \overline\rmG(x,y;u,v)-\overline\rmG(x,y;0,v)\ge \OLI \ell\,u\qquad\text{for $u\le 0$}\,.\EEE
  \end{equation}
  \OLI We set $\upbeta(r):=r_-=\max(0,-r)$ and $P_t:=\{x\in V:\bar u_t(x)<0\}$ for each $t\ge 0$. Due to the Lipschitz continuity of $\upbeta$, the map $t\mapsto b(t):=\int_V \upbeta(\bar u_t)\,\dd\pi$ is absolutely continuous. Hence, the chain-rule formula applies, which, together with \eqref{eq:124} gives
  \begin{align*}
  	\frac{\dd}{\dd t} b(t) &= -\int_{P_t} \overline\Gop[\bar u_t](x)\,\pi(\dd x) = -\iint_{P_t\times V}
               \overline\rmG(x,y;\bar u_t(x),\bar u_t(y))\,\tetapi(\dd x,\dd y) \\
               &\le \ell \iint_{P_t\times V} (-\bar u_t(x))\,\tetapi(\dd x,\dd y) = \ell \iint_E \upbeta(\bar u_t(x))\,\tetapi(\dd x,\dd y) \le \ell  \|\kappa_V\|_\infty \EEE b(t)\,.
  \end{align*}
\EEE
  Since $b$ is nonnegative and $b(0)=0$, we conclude, \OLI by Gronwall's inequality, \EEE that
  $b(t)=0$
  for every $t\ge0$ and therefore $\bar u_t\ge0$.
  In order to prove the the statement for a general $\bar u\in \R$
  it is sufficient to consider
  the new operator
  $\widetilde \rmG(x,y;u,v):=\rmG(x,y;u+\bar u,v+\OLI\bar u\GGG)$,
  and to consider the curve 
  $\widetilde u_t:=u_t-\bar u$ starting from the nonnegative initial datum
  $\widetilde u_0:=u_0-\bar u$.

\medskip

  \noindent
  \textit{(2)}
  It suffices to apply the transformation
  $\widetilde\rmG(x,y;u,v):=-\rmG(x,y;-u,-v)$
  and  set \EEE $\widetilde u_t:=-u_t$. We then apply the previous claim,
  yielding \EEE
  the lower bound $-\bar u$.
\end{proof}
 We can now state
our main result concerning
the well-posedness
of the Cauchy problem~\eqref{eq:119-Cauchy}. \EEE

\begin{theorem}
  \label{thm:localization-G}
  Let $J\subset \R$ be a closed interval of $\R$ and 
  let $G:\edg\times J^2\to\R$
  be a map satisfying conditions
  {\rm(\ref{subeq:G})}.
  Let us also suppose that, if $a=\inf J>-\infty$ then
  \eqref{eq:123a} holds, and that,
  if $b=\sup J<+\infty$  then \EEE \eqref{eq:123b} holds.
  \begin{enumerate}
  \item For every $u_0\in L^1(V,\pi;J)$ there
    exists a unique curve $u\in \rmC^1([0,\infty);L^1(V,\pi; J))$
    solving the Cauchy problem \eqref{eq:119-Cauchy}.
  \item $\int_V u_t\,\dd\pi=\int_V u_0\,\dd \pi$ for every $t\ge0$.
  \item If $u,v$ are two solutions with initial data $u_0,v_0\in L^1(V,\pi;J)$
    respectively,
    then
    \begin{equation}
      \label{eq:146}
      \|u_t-v_t\|_{L^1(V,\pi)}\le \rme^{2  \|\kappa_V\|_\infty \EEE \ell\,
        t}\|u_0-v_0\|_{L^1(V,\pi)}\quad
      \text{for every }t\ge0.
    \end{equation}
  \item If $\bar a\in J$ satisfies condition \eqref{eq:123a}
    and $u_0\ge \bar a$, then $u_t\ge \bar a$ for every $t\ge0$.
    Similarly, if $\bar b\in J$ satisfies condition \eqref{eq:123b}
    and $u_0\le \bar b$, then $u_t\le \bar b$ for every $t\ge0$.
  \item If $\ell=0$, then the evolution is order preserving:
    if $u,v$ are two solutions with initial data $u_0,v_0$
    then
    \begin{equation}
      \label{eq:128}
      u_0\le v_0\quad\Rightarrow\quad u_t\le v_t\quad\text{for every }t\ge0.
    \end{equation}
  \end{enumerate}
\end{theorem} \EEE
\begin{proof}
  Claims \textit{(1), (3), (4)}
  follow by the abstract generation result
  of \cite[\S 6.6, Theorem 6.1]{Martin76}
  applied to the operator $\Gop$ defined in the closed convex subset
  $D:=L^1(V,\pi;J)$ of the Banach space $L^1(V,\pi)$.
  \OLI For the theorem to apply, \EEE one has to check the continuity of $\Gop\OLI :D\to L^1(V,\pi)\GGG$ (Lemma
  \ref{le:tedious}(2)),
  its dissipativity \eqref{eq:141}, and the property
  \begin{displaymath}
    \liminf_{h\downarrow0} h^{-1}
    \inf_{v\in D} \|u+h\Gop[u]-v\|_{L^1(V,\pi)}=0
    \quad\text{for every }u\in D\,.
  \end{displaymath}
  When $J=\R$, the inner infimum always is zero; if $J$ is a bounded interval $[a,b]$ then the property above  follows from the estimates of Lemma
  \ref{le:tedious}(4), \OLI since for any $u\in D$,
\[
	\inf_{v\in D}\int_V |u + h \Gop[u]-v|\,\dd\pi \le \int_V \Bigl(a- (u+h\Gop[u])\Bigr)_+\dd\pi + \int_V \Bigl(u+h\Gop[u]-b\Bigr)_+\dd\pi\,.
\]
\EEE
When $J=[a,\infty)$ or $J = (-\infty,b]$ a similar reasoning applies.

Claim  \textit{(2)} is an immediate consequence of \eqref{eq:129}.
  Finally, when $\ell=0$, claim \textit{(5)}
  follows from the Crandall-Tartar Theorem \cite{Crandall-Tartar80},
  stating that a non-expansive map in $L^1$  (cf.\ \eqref{eq:146}) \EEE that satisfies claim \textit{(2)} is also order preserving. 
\end{proof}

\subsection{Applications to dissipative evolutions}
Let us now consider the map
$\fw:  (0,+\infty)^2 \to \R$  \EEE induced by the system $(\Psi^*,\upphi,\upalpha)$, first introduced in~\eqref{eq:184},
\begin{equation}
\label{eq:appA-w-Psip-alpha}	
\fw(u,v) := (\Psi^*)'\bigl( \upphi'(v)-\upphi'(u)\bigr)\,
\upalpha(u,v)
\quad\text{for every }u,v>0\,,
\end{equation}
with the corresponding integral operator:
\begin{equation}
  \label{eq:162}
  \Fop[u](x):=\int_V \fw(u(x),u(y))\,\kappa(x,\dd y)\,.
\end{equation}
Since $\Psi^*$, $\upphi$ are $\rmC^1$ \OLI convex \EEE functions on
$(0,\pinfty)$
and $\upalpha$ is locally Lipschitz in $(0,\pinfty)^2$
it is easy to check that $\fw$ satisfies properties
(\ref{subeq:G}a,b,c,d) in every compact subset $J\subset (0,\pinfty)$
and conditions \eqref{eq:123a}, \eqref{eq:123b}
at every point $a,b\in J$.
In order to focus on the structural properties of the \RICKYNEW {associated evolution
problem, cf.\ \eqref{eq:159} below,} \EEE we will mostly confine our analysis to
the regular case, according to the following: \EEE
\begin{Assumptions}{$\rmF$}
  \label{ass:F}
 \EEE The map $\fw$
  defined by \eqref{eq:appA-w-Psip-alpha} satisfies \OLI the following properties:\EEE
  \begin{gather}
    \label{eq:157}
    \fw\text{ admits a continuous extension to $[0,\infty)$, } \intertext{and for every $R>0$ there exists
      $\ell_R\ge 0$ such that }
    \label{eq:158}
    v\le v'\quad\Rightarrow\quad \fw(u,v)-\fw(u,v')\le \ell_R\,
    (v'-v) \quad\text{for every }u,v,v'\in [0,R].
  \end{gather}
  If moreover 
  \eqref{eq:158} is satisfied in $[0,\pinfty)$ for some constant
  $\ell_\infty\ge 0$ and
  there exists a constant $M$ such that
    \begin{equation}
      \label{eq:161}
      |\fw(u,v)|\le M(1+u+v)\quad\text{for every }u,v\ge0\,,
    \end{equation}
    we say that $(\rmF_\infty)$ holds.
\end{Assumptions}
Note that \eqref{eq:157} is always satisfied if $\upphi$ is
differentiable at $0$. Estimate \eqref{eq:158} is also true if in addition $\upalpha$ is
Lipschitz. However, as we have shown in Section
\ref{subsec:examples-intro},
there are important examples in which  $\upphi'(0)=-\infty$,
but \eqref{eq:157} and \eqref{eq:158} hold nonetheless.

\medskip

Theorem \ref{thm:localization-G} yields the following general result:
\begin{theorem}
  \label{thm:ODE-well-posedness}
  Consider
  the Cauchy problem
  \begin{equation}
    \label{eq:159}
    \dot u_t=\Fop[u_t] \quad t\ge0,\quad
    u\restr{t=0}=u_0.
  \end{equation}
  for a given nonnegative $u_0\in L^1(V,\pi)$.
  \begin{enumerate}[label=(\arabic*)]
  \item \label{thm:ODE-well-posedness-ex}
    For every $u_0\in L^1(V,\pi;J)$ with $J$ a compact subinterval of
    $(0,\pinfty)$ 
    there exists a unique bounded and nonnegative solution $u\in
    \rmC^1([0,\infty);L^1(V,\pi;J))$ of \eqref{eq:159}.
    We will denote by $(\sfS_t)_{t\ge0}$ the
    corresponding $\rmC^1$-semigroup \OLI of nonlinear operators, \EEE mapping $u_0$ to the value
    $u_t=\sfS_t[u_0]$ at time $t$ of the solution $u$.
  \item
    $\int_V u_t\,\dd\pi=\int_V u_0\,\dd\pi$ for every $t\ge0$.
  \item
    If $a\le u_0\le b$ $\pi$-a.e.~in $V$, then $a\le u_t\le b$
    $\pi$-a.e.~for every $t\ge0$.
  \item\label{thm:ODE-well-posedness-Lip}
    The solution satisfies the
    Lipschitz estimate \eqref{eq:146} (with $\ell=\ell_R$) and the
    order preserving property if $\ell_R=0$.
    \item
      If Assumption \ref{ass:F} holds,  then
      $(\sfS_t)_{t\ge0}$ can be extended
    to a semigroup defined on
     every essentially bounded nonnegative
     $u_0\in L^1(V,\pi)$ and satisfying the same properties \ref{thm:ODE-well-posedness-ex}--\ref{thm:ODE-well-posedness-Lip} above.
  \item
    If additionally $(\rmF_\infty)$ holds, 
    then $(\sfS_t)_{t\ge0}$ can be extended
    to a semigroup defined on
     every nonnegative
    $u_0\in L^1(V,\pi)$ and satisfying the same properties \ref{thm:ODE-well-posedness-ex}--\ref{thm:ODE-well-posedness-Lip} above.
  \end{enumerate}
\end{theorem}

We now show that the solution $u$
given by Theorem~\ref{thm:ODE-well-posedness} is also
a solution in the sense of the $(\calS,\calR,\calR^*)$ Energy-Dissipation balance.

\begin{theorem}
\label{thm:sg-sol-is-var-sol}
Assume~\ref{ass:V-and-kappa}, \ref{ass:Psi}, \ref{ass:S}.
Let $u_0 \in L^1(V;\pi)$ be nonnegative and $\pi$-essentially valued in
a compact interval $J$ of $(0,\infty)$
and let $u=\sfS[u_0]  \in  \mathrm{C}^1 ([0,\pinfty);L^1(V,\pi;J))$
be the solution to \eqref{eq:159}
given by Theorem~\ref{thm:ODE-well-posedness}.
\nc
Then the pair $(\rho,\bj )$ given by
 \begin{align*}
 \rho_t (\dd x)&: = u_t(x) \pi(\dd x)\,,\\
   \OLI 2\EEE\bj_t (\dd x\, \dd y)&:= w_t(x,y)\, \tetapi(\dd x\, \dd y)\,,\qquad
                          w_t(x,y):=-\rmF(u_t(x),u_t(y))\,,
 \end{align*}
 is an element of $\CE0\pinfty$ and  satisfies the $(\calS,\calR$,$\calR^*)$ 
 Energy-Dissipation balance \eqref{R-Rstar-balance}.

 If\/ $\rmF$ satisfies the stronger assumption\/ \ref{ass:F},
 then the same result holds for every essentially bounded and nonnegative initial datum.
 Finally, if also $(\rmF_\infty)$ holds,
 the above result is valid for every nonnegative $u_0\in L^1(V,\pi)$ with
 $\rho_0=u_0\pi\in D(\calS)$.
 \end{theorem} 
 \begin{proof}
   Let us first consider the case when $u_0$ satisfies
   $0<a\le u_0\le b<\pinfty$ $\pi$-a.e..
   Then, the solution $u=\sfS[u_0]$
   satisfies the same bounds,
   the map $w_t$ is uniformly bounded
   and $
   \upalpha(u_t(x),u_t(y))\ge \upalpha(a,a)>0$, 
   so that $(\rho,\bj)\in \CER 0T.$
   We can thus apply Theorem \ref{thm:characterization}, obtaining
   the Energy-Dissipation balance
   \begin{equation}
     \label{eq:164L}
     \calS(\rho_0)-\calS(\rho_T)=
     \int_0^T \calR(\rho_t,\bj_t)\,\dd t+
     \int_0^T \Fish(\rho_t)\,\dd t,
     \qquad\text{or equivalently}\quad
     \mathscr L(\rho,\bj)=0.
   \end{equation}

   In the case  $0\leq u_0\leq b$ 
   we can argue by approximation, setting $u_{0}^a:=\max\{u_0, a\}$, $a>0$,
   and considering the solution $u_t^a:=\sfS_t[u_0^a]$
   with divergence field $\OLI 2\EEE \bj_t^a(\dd x,\dd
   y)=-\rmF(u_t^{a}(x),u_t^{a}(y))\tetapi(\dd x,\dd y)$.
   Theorem \ref{thm:ODE-well-posedness}(4) shows that
   $u_t^a\to u_t$ strongly in $L^1(V,\pi)$ as $a\downarrow0$, \OLI and consequently also 
    $\bj_\lambda^a\to \bj_\lambda$ \EEE setwise. Hence,
    we can pass to the limit in
   \eqref{eq:164L} (written for $(\rho^a,\bj^a)$
   thanks to Proposition \ref{prop:compactness} and Proposition
   \ref{PROP:lsc}),
   obtaining $\mathscr L(\rho,\bj)\le 0$, which is still sufficient
   to conclude that $(\rho,\bj)$ is a solution thanks to Remark
   \ref{rem:properties}(3).

   Finally, if $(\rmF_\infty)$ holds,
   we obtain the general result by a completely analogous argument,
   approximating $u_0$ by the sequence $u_0^b:=\min\{u_0, b\}$
   and letting $b\uparrow\pinfty$.   
\end{proof}
\nc

\section{Existence via Minimizing Movements}
\label{s:MM}

In this section we construct solutions to the $(\calS,\calR,\calR^*)$ formulation
via the \emph{Minimizing Movement} approach. The method uses only fairly general properties of $\DVTn$, $\calS$, and the underlying space, and it may well have broader applicability than the \RICKYNEW measure-space \EEE setting that we consider here (see Remark \ref{rmk:generaliz-topol}).   Therefore we formulate the results in a slightly more general setup. 

We consider a topological space  
 \begin{equation}
\label{ambient-topological}
(X,\sigma) = \calM^+(V) \text{ endowed with the \GGGO setwise topology}.
\end{equation}
For consistency with the above definition, in this section we will use
use the abstract notation $\weaksigmatoabs$  to denote  setwise
convergence  in $X= \calM^+(V)$. 
Although throughout this paper we adopt the Assumptions~\ref{ass:V-and-kappa}, \ref{ass:Psi}, and~\ref{ass:S}, in this chapter we will base the discussion only on the following properties:
\begin{Assumptions}{Abs}
\label{ass:abstract}
\begin{enumerate}
\item the Dynamical-Variational Transport (DVT) cost $\DVTn$ enjoys properties \eqref{assW};
\item the driving functional $\calS$ enjoys the typical lower-semicontinuity and coercivity properties underlying the variational approach to gradient flows:
\begin{subequations}\label{conditions-on-S}
\begin{align}
\label{e:phi1}
&\calS \geq 0 \quad  \text{and}  \quad  \calS 
  \ \text{is $\sigma$-sequentially lower semicontinuous};\\
&\exists \rho^{*} \in X \quad\text{such that}\quad \forall\, \tau>0, \notag\\
	\label{e:phipsi1}
&\qquad \text{the map $\rho \mapsto \DVTn(\tau,\rho^{*}, \rho) + \calS(\rho)$ has
$\sigma$-sequentially compact sublevels.}
	\end{align}
\end{subequations}
\end{enumerate}
\end{Assumptions}

Assumption~\ref{ass:abstract} is implied by Assumptions~\ref{ass:V-and-kappa}, \ref{ass:Psi}, and~\ref{ass:S}. The properties~\eqref{assW} are the content of Theorem~\ref{thm:props-cost};  condition \eqref{e:phi1} follows from Assumption~\ref{ass:S} and Lemma~\ref{PROP:lsc};  condition~\eqref{e:phipsi1} follows from the superlinearity of $\upphi$ at infinity and Prokhorov's characterization of compactness in the space of finite measures~\cite[Th.~8.6.2]{Bogachev07}.

\subsection{The Minimizing Movement scheme and the convergence result}
\label{ss:MM}
The classical `Minimizing Movement' scheme for metric-space gradient flows~\cite{DeGiorgiMarinoTosques80,AmbrosioGigliSavare08} starts by defining approximate solutions through incremental minimization, 
\[
\rho^n \in \argmin_{\rho} \left( \frac1{2\tau} d(\rho^{n-1},\rho)^2 + \calS(\rho)\right).
\]
In the context of this paper the natural generalization of the expression to be minimized is $\DVT\tau{\rho^{n-1}}\rho + \calS(\rho)$. This can be understood by remarking that if $\calR(\rho,\cdot)$ is quadratic, then it formally generates a metric 
\begin{align*}
\frac12 d(\mu,\nu)^2 
&= \inf\left\{ \int_0^1 \calR(\rho_t,\bj_t)\,\dd t \, : \, \partial_t \rho_t + \odiv \bj_t = 0, \ \rho_0 = \mu, \text{ and }\rho_
  1 = \nu\right\}\\
&= \tau \inf\left\{ \int_0^\tau \calR(\rho_t,\bj_t)\,\dd t \, : \, \partial_t \rho_t + \odiv \bj_t = 0, \ \rho_0 = \mu, \text{ and }\rho_
  \tau = \nu\right\}\\
  &= \tau \DVT\tau{\mu}\nu.
\end{align*}
In this section we set up the approximation scheme  featuring the cost $\DVTn$.  

We 
consider  a partition
$
 \{\ttau^0 =0< \ttau^1< \ldots<\ttau^n < \ldots< \ttau^{N_\tau-1}<T\le
\ttau^{N_\tau}\} $, 
with fineness $\tau : = \max_{i=n,\ldots, N_\tau} (\ttau^{n} {-} \ttau^{n-1})$,
of the time interval $[0,T]$.
The sequence of approximations $(\rho_\tau^n)_n$ is defined by the \RICKYNEW following \EEE recursive minimization scheme. Fix $\rho^\circ\in X$.
\begin{problem}
\label{pr:time-incremental} Given $\Utau^0:=\rho^\circ,$ find $\Utau^1,
\ldots, \Utau^{N_\tau} \in X $ fulfilling
\begin{equation}
\label{eq:time-incremental} \Utau^n \in \argmin_{v \in X} \Bigl\{
\DVTn(\ttau^n -\ttau^{n-1}, \Utau^{n-1}, v) +\calS(v)\Bigr\} \quad
\text{for $n=1, \ldots, {N_\tau}.$}
\end{equation}
\end{problem}
\begin{lemma}
\label{lemma:exist-probl-increme} Under assumption \ref{ass:abstract},
  for any $\tau
>0$ Problem \ref{pr:time-incremental} admits a solution
$\{\Utau^n\}_{n=1}^{{N_\tau}}\subset X$.
\end{lemma}

\par
We denote by   $\pwC \rho \tau$  and  $\upwC \rho \tau$ the
left-continuous and right-continuous piecewise constant interpolants
of the values $\{\Utau^n\}_{n=1}^{{N_\tau}}$ on the nodes of the partition,
fulfilling $\pwC \rho \tau(\ttau^n)=\upwC \rho
\tau(\ttau^n)=\Utau^n$ for all $n=1,\ldots, {N_\tau}$, i.e.,
\begin{equation}
\label{pwc-interp}
 \pwC \rho \tau(t)=\Utau^n \quad \forall t \in (\ttau^{n-1},\ttau^n],
\quad \quad \upwC \rho \tau(t)=\Utau^{n-1} \quad \forall t \in
[\ttau^{n-1},\ttau^n), \quad n=1,\ldots, {N_\tau}.
\end{equation}
Likewise, we  denote by $\pwC {\sft}{\tau}$ and $\upwC {\sft}{\tau}$  the piecewise constant interpolants  $\pwC {\sft}{\tau}(0): = \upwC {\sft}{\tau}(0): =0$, $ \pwC {\sft}{\tau}(T): = \upwC {\sft}{\tau}(T): =T$,  and 
\begin{equation}
\label{nodes-interpolants}
 \pwC {\mathsf{t}} \tau(t)=\ttau^n \quad \forall t \in (\ttau^{n-1},\ttau^n],
\quad \quad \upwC {\mathsf{t}} \tau(t)=\ttau^{n-1} \quad \forall t \in
[\ttau^{n-1},\ttau^n)\,.
\end{equation}
\par
We also introduce another  notion  of interpolant of  the discrete values $\{\Utau^n\}_{n=0}^{N_\tau}$ 
 introduced by De Giorgi, namely the  \emph{variational interpolant} 
$\pwM\rho\tau : [0,T]\to X$, which
 is defined in the following way: the map 
 $t\mapsto \pwM \rho\tau(t)$ 
 is Lebesgue measurable
in $(0, T )$ and satisfies
\begin{equation}
  \label{interpmin}
  \begin{cases}\quad
      \pwM \rho\tau(0)=\rho^\circ, \quad \text{and, for }
       t=\ttau^{n-1} + r \in (\ttau^{n-1}, \ttau^{n}],
       \medskip
      \\\quad
      \pwM \rho\tau(t)
     \in
\displaystyle \argmin_{\mu \in X} \left\{ \DVTn(r, \Utau^{n-1}, \mu) +\calS(\mu)\right\}
\end{cases}
\end{equation}
The existence of a measurable selection is guaranteed by 
\cite[Cor. III.3, Thm. III.6]{Castaing-Valadier77}. 
\par
 It is  natural to introduce the following extension of the notion of \emph{(Generalized) Minimizing Movement}, which is typically given in a metric setting \cite{Ambr95MM,AmbrosioGigliSavare08}. For simplicity, we will continue to use the classical terminology.
\begin{definition}
\label{def:GMM}
We say that a curve $\rho: [0,T] \to X$ is a \emph{Generalized Minimizing Movement} for the energy functional $\calS$ starting from the initial datum $\rho^\circ\in \mathrm{D}(\calS)$,
if there exist a sequence of partitions with fineness $(\tau_k)_k$, $\tau_k\downarrow 0$ as $k\to\infty$, and, correspondingly, a sequence of discrete solutions $(\pwC \rho {\tau_k})_k$
such that, as $k\to\infty$,
\begin{equation}
\label{sigma-conv-GMM}
\pwC \rho {\tau_k}(t) \weaksigmatoabs \rho(t) \qquad \text{for all } t \in [0,T].
\end{equation}
We shall denote by $\GMM{\calS,\DVTn}{\rho^\circ}$ the collection of all Generalized Minimizing Movements for $\calS$ starting from $\rho^\circ$. 
\end{definition} 

We can now state the main result of this section.
\begin{theorem}
\label{thm:construction-MM}
Under \textbf{Assumptions~\ref{ass:V-and-kappa}, \ref{ass:Psi},
  and~\ref{ass:S}},
let the lower-semicontinuity Property \eqref{lscD} be satisfied.

Then $\GMMT{\calS,\DVTn}{0}{T}{\rho^\circ} \neq \emptyset$ and  every $\rho \in \GMMT{\calS,\DVTn}{0}{T}{\rho^\circ}$ 
 satisfies the $(\calS,\calR,\calR^*)$ Energy-Dissipation balance (Definition~\ref{def:R-Rstar-balance}). 
\end{theorem}

 Throughout Sections \ref{ss:aprio}--\ref{ss:compactness} we will first prove an abstract version of this theorem as Theorem~\ref{thm:abstract-GMM} below,  under \textbf{Assumption~\ref{ass:abstract}}.
 Indeed, therein we could `move away' from the context of the `concrete' gradient structure for the Markov processes, and carry out our analysis in a general topological setup (cf.\ Remark 
 \ref{rmk:generaliz-topol} ahead). In  Section~\ref{ss:pf-of-existence} we will `return' to the problem  under consideration and deduce the proof of Theorem\ \ref{thm:construction-MM} from Theorem\ \ref{thm:abstract-GMM}. 

%

\subsection{Moreau-Yosida approximation and  generalized slope}
\label{ss:aprio}
Preliminarily, let us observe some straightforward consequences of the properties of the transport cost:
\begin{enumerate}
\item
 the `generalized triangle inequality' from  \eqref{e:psi3} entails that for all $m \in \N$,  for
all $(m+1)$-ples $(t, t_1, \ldots, t_m) \in (0,\pinfty)^{m+1}$, and all 
$(\rho_0, \rho_1, \ldots, \rho_m) \in X^{m+1}$, we have 
\begin{equation}
\label{eq1} 
\DVTn(t,\rho_0,\rho_{m}) \leq \sum_{k=1}^{m}
\DVTn(t_k,\rho_{k-1},\rho_{k}) \qquad \text{if\, $t=\sum_{k=1}^m t_k$.}
\end{equation}
\item Combining \eqref{e:psi2} and \eqref{e:psi3} we deduce that
\begin{equation}
\label{monotonia} 
\DVTn(t,\rho,\mu) \leq \DVTn(s,\rho,\mu) \quad \text{ for all } 0<s<t
\text{ and for all } \rho, \mu \in X.
\end{equation}
\end{enumerate}

In the context of  metric gradient-flow theory, the `Moreau-Yosida approximation' (see e.g.~\cite[Ch.~7]{Brezis11} or~\cite[Def.~3.1.1]{AmbrosioGigliSavare08})  provides an approximation of the driving functional that is finite and sub-differentiable everywhere, and can be used to define a generalized slope. We now construct the analogous objects in the situation at hand. 

Given $r>0$ and $\rho \in X$, we define the subset $J_r(\rho)\subset X$ by 
$$
J_r(\rho) := \argmin_{\mu \in X} \Bigl\{ \DVTn(r,\rho,\mu) + \calS(\mu)\Bigr\}
$$
(by Lemma \ref{lemma:exist-probl-increme}, this set is non-empty) and
define 
\begin{equation}
\label{def:gen}
 \gen(r,\rho):= \inf_{\mu \in X} \left\{ \DVTn(r,\rho,\mu) +
\calS(\mu)\right\}= \DVTn(r,\rho, \rho_r) + \calS(\rho_r) \quad \forall\, \rho_r \in
J_r(\rho).
\end{equation}
In addition, for all $\rho \in \rmD(\calS)$, we define the \emph{generalized slope}\begin{equation}
\label{def:nuovo}
\nuovo(\rho):= \limsup_{r \downarrow 0}\frac{\calS(\rho) -\gen(r,\rho) }{r} =
\limsup_{r \downarrow 0}\frac{\sup_{\mu \in X} \left\{ \calS(\rho) -\DVTn(r,\rho,\mu)
-\calS(\mu)\right\} }{r}\,. 
\end{equation}
Recalling the \emph{duality formula} for the local slope (cf.\ \cite[Lemma 3.15]{AmbrosioGigliSavare08}) and the fact that 
$\DVTn(\tau,\cdot, \cdot)$ is a proxy for $\frac1{2\tau}d^2(\cdot, \cdot)$, it is immediate to recognize that the generalized slope is a surrogate of the local slope. Furthermore, as we will see that its definition is somehow tailored to  the validity of Lemma \ref{lemma-my-1} \RICKYNEW ahead. \EEE  
Heuristically, the generalized slope $\nuovo(\rho)$ coincides with the Fisher information $\Fish(\rho) = \calR^*(\rho,-\rmD\calS(\rho))$. This can be recognized, again heuristically, by fixing a point $\rho_0$ and considering curves $\rho_t := \rho_0 -t\odiv \bj $, for a class of fluxes $\bj $. We  then calculate
\begin{align*}
\calR^*(\rho_0,-\rmD\calS(\rho_0))
&= \sup_{\bj }\, \bigl\{ -\rmD\calS(\rho_0)\cdot \bj  - \calR(\rho_0,\bj )\bigr\}\\
&= \sup_\bj  \lim_{r\to0} \frac1r \biggl\{ \calS(\rho_0) - \calS(\rho_r) - \int_0^r \calR(\rho_t,\bj )\, \dd t\biggr\}.
\end{align*}
In Theorem~\ref{th:slope-Fish} below we  rigorously prove that  $\nuovo\geq \Fish$ using this approach.

\bigskip

The following result collects some properties of $\genn r$ and $\nuovo$.
\begin{lemma}
    \label{lemma-my-1} 
For all $\rho \in \rmD(\calS)$ and for every selection $ \rho_r \in
J_r(\rho)$
\begin{align}
& \label{gen1} \gen(r_2, \rho) \leq \gen(r_1,\rho) \leq \calS(\rho) \quad
\text{for all } 0 <r_1<r_2;
\\
& \label{gen2} \rho_r \weaksigmatoabs\rho \ \text{as $r \downarrow 0$,} \quad
\calS(\rho)= \lim_{r \downarrow 0} \gen(r,\rho);
\\
& \label{gen3}
 \frac{\rm d}{{\rm d}r} \gen(r,\rho) \leq - \nuovo(\rho_r) \quad \foraa\
 r>0.
\end{align}
In particular, for all $\rho \in \rmD(\calS)$
\begin{align}
& \label{nuovopos}
 \nuovo(\rho) \geq 0 \quad \ \text{and}
\\
&
    \label{ineqenerg}
   \DVTn(r_0, \rho, \rho_{r_0}) +
 \int_{0}^{r_0} \nuovo(\rho_r) \, {\rm d}r \leq \calS(\rho) -\calS(\rho_{r_0})
    \end{align}
    for every $ r_{0}>0$ and $
   \rho_{r_0} \in J_{r_0}(\rho)$.
\end{lemma}
\begin{proof}
Let $r>0$, $\rho\in \rmD(\calS)$, and $\rho_r\in J_r(\rho)$.
It follows from  \eqref{def:gen}  and 
\eqref{e:psi2} that
\begin{equation}
\label{eq3} \gen(r,\rho)= \DVTn(r,\rho,\rho_r) + \calS(\rho_r)  \leq \DVTn(r,\rho,\rho) +
\calS(\rho) = \calS(\rho) \quad \forall \, r>0, \rho \in X;
\end{equation}
in the same way, one checks that for all $\rho\in X$ and $0 <r_1<r_2$,
\[
\gen(r_2,\rho) -\gen(r_1, \rho) \leq \DVTn(r_2, \rho_{r_1}, \rho) + \calS(\rho_{r_1})
-\DVTn(r_1, \rho_{r_1}, \rho) - \calS(\rho_{r_1}) \stackrel{\eqref{monotonia}}\leq 0,
\]
which implies \eqref{gen1}. Thus, the map $r \mapsto \gen(r,\rho)$ is  non-increasing on $(0,\pinfty)$, and hence almost everywhere
differentiable. Let us fix a point of differentiability $r>0$. For $h>0$ and $\rho_r \in J_r (\rho)$ we then have
\begin{align*}
\frac{\gen(r+h,\rho)-\gen(r,\rho)}{h} 
& = \frac1h\, {\inf_{v \in X} \Bigl\{\DVTn(r+h,
\rho,v) +\calS(v) -\DVTn(r, \rho,\rho_r) -\calS(\rho_r) \Bigr\}}
\\
& \leq \frac1h \, {\inf_{v \in X} \Bigl\{\DVTn(h, \rho_r,v) +\calS(v) -\calS(\rho_r)
\Bigr\}},
\end{align*}
the latter inequality due to  \eqref{e:psi3}, so that
$$
\begin{aligned}
 \frac{\rm d}{{\rm d}r} \gen(r,\rho) &
 \leq \liminf_{h \downarrow 0} \,\frac1h \,{\inf_{v \in X} \Bigl\{\DVTn(h, \rho_r,v) +\calS(v) -\calS(\rho_r)
\Bigr\}}
\\
& = - \limsup_{h \downarrow 0}\, \frac1h \, {\sup_{v \in X} \Bigl\{-\DVTn(h, \rho_r,v)
-\calS(v) +\calS(\rho_r) \Bigr\}}, \end{aligned}
$$
whence \eqref{gen3}. Finally, \eqref{eq3} yields that, for any $\rho
\in \rmD(\calS)$ and any selection $\rho_r \in J_r(\rho)$, one has $
\sup_{r>0} \DVTn(r,\rho,\rho_r) <+ \infty .$ Therefore, \eqref{e:psi4}
entails the first convergence in   \eqref{gen2}. Furthermore, we
have
$$
\calS(\rho) \geq \limsup_{r \downarrow 0} \gen(r,\rho)\geq \liminf_{r \downarrow 0}
\left(\DVTn(r,\rho,\rho_r) + \calS(\rho_r)\right) \geq \liminf_{r \downarrow
0}\calS(\rho_r)\geq \calS(\rho),
$$
where the first inequality again follows from \eqref{eq3}, and the
last one from the $\sigma$-lower semicontinuity of $\calS$.
This  implies the second statement of~\eqref{gen2}.
\end{proof}

\subsection{\emph{A priori} estimates}
Our next result collects the basic  estimates on the discrete solutions. In order to properly state it, we need to introduce
   the `density  of
dissipated energy' associated with the interpolant $\pwC \rho{\tau}$, namely the 
piecewise constant
function
$\pwC {\mathsf{W}}{\tau}:[0,T] \to [0,\pinfty)$ defined by
\begin{align}
\notag
\pwC {\mathsf{W}}{\tau}(t)&:= \frac{\DVTn(\ttau^{n}-\ttau^{n-1},
\Utau^{n-1}, \Utau^{n})}{\ttau^{n}-\ttau^{n-1}} \quad t\in
(\ttau^{n-1}, \ttau^n], \quad n=1,\ldots, {N_\tau}, 
\\
   \text{so that}\quad \int_{\ttau^{j-1}}^{\ttau^n}\pwC {\mathsf{W}}{\tau}(t)\, {\rm d}t&=
\sum_{k=j}^n \DVTn(\ttau^{k}-\ttau^{k-1}, \Utau^{k-1}, \Utau^{k})
\quad \text{for all } 1 \leq j < n \leq {N_\tau}.
\label{density-W}
\end{align}
\begin{prop}[Discrete energy-dissipation inequality and \emph{a priori} estimates]
We have
\begin{align}
\label{discr-enineq-var}
&
   \DVTn(t-\upwC \sft{\tau}(t), \upwC \rho{\tau}(t), \pwM {\rho}{\tau}(t)) +
 \int_{\upwC \sft{\tau}(t)}^{t} \nuovo(\pwM {\rho}{\tau}(r)) \, {\rm d}r +\calS(\pwM
 {\rho}{\tau}(t))\leq \calS(\upwC \rho{\tau}(t)) \quad \text{for all } 0 \leq t \leq T\,,
\\
&
\label{discr-enineq}
 \int_{\upwC \sft{\tau}(s)}^{\pwC \sft{\tau}(t)}   \pwC {\mathsf{W}}{\tau}(r)\, {\rm d}r  + 
   \int_{\upwC \sft{\tau}(s)}^{\pwC \sft{\tau}(t)} \nuovo(\pwM {\rho}{\tau}(r)) \, {\rm d}r  +\calS(\pwC \rho\tau(t))
 \leq \calS(\upwC \rho\tau(s)) \qquad \text{for all } 0\leq s \leq t \leq T\,,
\end{align}
 and there exists a constant $C>0$ such that  for all $\tau>0$
\begin{gather}
\label{est-diss}
\int_0^T \pwC{\mathsf{W}}\tau (t)\, \dd t \leq C, \qquad \int_0^T  \nuovo(\pwM {\rho}{\tau}(t)) \, {\rm d}t  \leq C.
\end{gather}
Finally, there exists a $\sigma$-sequentially compact subset $K\subset X$ such that 
\begin{equation}\label{aprio1}
\pwC \rho{\tau}(t),\, \upwC \rho{\tau}(t),\,  \pwM \rho{\tau}(t)\, \in K \quad \text{$\forall\, t \in
 [0,T]$ and $\tau >0$}.
\end{equation}
\end{prop}
\begin{proof}
From \eqref{ineqenerg} we directly deduce, for $t \in (\ttau^{j-1},
\ttau^j]$,
\begin{equation}
\label{eq4}
   \DVTn(t-\ttau^{j-1}, \Utau^{j-1}, \pwM {\rho}{\tau}(t)) +
 \int_{\ttau^{j-1}}^{t} \nuovo(\pwM {\rho}{\tau}(r)) \, {\rm d}r +\calS(\pwM
 {\rho}{\tau}(t))\leq \calS(\Utau^{j-1}),
\end{equation}
which implies  
\eqref{discr-enineq-var};
in particular, for $t= \ttau^j$ one has
\begin{equation}
\label{eq4bis}
   \int_{\ttau^{j-1}}^{\ttau^j}\pwC {\mathsf{W}}{\tau}(t)\, {\rm d}t +
 \int_{\ttau^{j-1}}^{\ttau^j} \nuovo(\pwM {\rho}{\tau}(t)) \, {\rm d}t +\calS(\Utau^{j})\leq
 \calS(\Utau^{j-1}).
\end{equation}
The estimate~\eqref{discr-enineq}
 follows upon summing \eqref{eq4bis} over the index $j$. Furthermore,  applying \eqref{eq1}--\eqref{monotonia}
one deduces  for all $1 \leq n \leq N_\tau$ that
\begin{equation}
\label{est-basis}
\DVTn(n\tau,\rho_0,\Utau^{n}) +\calS(\Utau^{n})
\leq   \int_0^{\ttau^n}\pwC {\mathsf{W}}{\tau}(r)\, {\rm d}r +
   \int_0^{\ttau^n} \nuovo(\pwM {\rho}{\tau}(r)) \, {\rm d}r  +\calS(\Utau^{n})
 \leq \calS(\rho_0).
\end{equation}
In particular, \eqref{est-diss} follows, as well as $\sup_{n=0,\ldots,N_\tau} \calS(\Utau^{n}) \leq C$. 
Then, \eqref{eq4} also yields $\sup_{t\in [0,T]} \calS(\pwM
 {\rho}{\tau}(t)) \leq C$.

Next we show the two estimates
\begin{align}
\label{est1}
 \DVTn(2T, \rho^*,\pwC \rho{\tau}(t)) +\calS(\pwC \rho{\tau}(t))
\leq C,
\\
\label{est2}
\DVTn(2T, \rho^*, \pwM {\rho}{\tau}(t)) + \calS(\pwM {\rho}{\tau}(t)) \leq C  \,.
\end{align}
Recall that $\rho^*$ is introduced in Assumption~\ref{ass:abstract}.

To deduce \eqref{est1}, we  use the triangle inequality for $\DVTn$.
Preliminarily, we observe that 
$ \DVTn(t, \rho^*, \rho_0) <\pinfty$ for all $t>0$. In particular, let us fix an arbitrary $  m \in \{1,\ldots, N_\tau\}$ and let $C^*: =  \DVTn(\ttau^{ m}, \rho^*, \rho_0) $. 
We have for any $n$,
\begin{align*}
\DVTn(2T, \rho^*,\Utau^n) &\leq \DVTn(2T-\ttau^n, \rho^*, \rho_0)   +\DVTn(\ttau^n, \rho_0,\Utau^n) \stackrel{(1)}{\leq}   \DVTn(\ttau^{ m}, \rho^*, \rho_0)    +\DVTn(\ttau^n, \rho_0,\Utau^n) \\ & \leq  C^*   +\DVTn(\ttau^n, \rho_0,\Utau^n)  \quad \text{for all } n \in \{1, \ldots,N_\tau\},
\end{align*}
where for (1) we have used that $  \DVTn(2T-\ttau^n, \rho^*, \rho_0) \leq   \DVTn(\ttau^{ m}, \rho^*, \rho_0) $ since $2T- \ttau^n \geq \ttau^{ m} $. 
Thus, 
in view of  \eqref{est-basis} we 
 we deduce
\begin{align}
\notag 
\DVTn(2T, \rho^*,\pwC \rho{\tau}(t))   +\calS(\pwC \rho{\tau}(t))
 &\leq  C^*   +\DVTn(\pwC \sft \tau(t), \rho_0, \pwC \rho{\tau}(t) ) +\calS(\pwC \rho{\tau}(t)) \\
&\leq C^*
 +\calS(\rho_0) \leq C    \quad \text{for all } t \in [0, T]\,,
\label{eq5}
\end{align}
i.e.\  the desired  \eqref{est1}.
\par

Likewise, adding \eqref{eq4} and \eqref{eq4bis} one has
$
\DVTn(t, \rho_0, \pwM {\rho}{\tau}(t)) + \calS(\pwM {\rho}{\tau}(t)) \leq
\calS(\rho_0) $, whence \eqref{est2} with arguments similar to those in the previous lines. 
\end{proof}

\subsection{Compactness result}
\label{ss:compactness}
The main result of this section, Theorem \ref{thm:abstract-GMM}  below, states that $\GMMT{\calS,\DVTn}0T{\rho^\circ} $ is non-empty, and that any curve $\rho \in \GMMT{\calS,\DVTn}0T{\rho^\circ}$ fulfills an `abstract' version \eqref{limit-enineq} of the
 $(\calS,\calR,\calR^*)$  Energy-Dissipation estimate~\eqref{EDineq}, obtained by passing to the limit in the discrete inequality~\eqref{discr-enineq}. 
 
We recall the $\DVTn$-action of a curve $\rho:[0,T]\to X$, defined in~\eqref{def-tot-var} as 
\begin{equation*}
\VarW \rho ab: = \sup \left \{ \sum_{j=1}^M  
\DVT{t^j - t^{j-1}}{\rho(t^{j-1})}{\rho(t^j)} \, : \ (t^j)_{j=0}^M \in \mathfrak{P}_f([a,b])   \right\}
\end{equation*}
for all $[a,b]\subset [0,T]$, where $\mathfrak{P}_f([a,b])$  is the set of all finite partitions of the interval $[a,b]$.
We also introduce the \emph{relaxed generalized slope} $\nuovorel: \rmD (\calS) \to [0,\pinfty]$ of the driving energy functional $\calS$, namely
 the relaxation of the generalized slope $\nuovo$ along sequences with bounded energy:
\begin{equation}
\label{relaxed-nuovo}
 \nuovorel(\rho) : = \inf\biggl\{ \liminf_{n\to\infty} \nuovo(\rho_n) \, :   \ \rho_n\weaksigmatoabs \rho, \ \sup_{n\in \N} \calS(\rho_n) <\pinfty\biggr\}\,.
\end{equation}

We are now in a position to state and prove the `abstract version' of Theorem \ref{thm:construction-MM}.
\begin{theorem}
\label{thm:abstract-GMM} 
Under \textbf{Assumption~\ref{ass:abstract}}, let $\rho^\circ \in \mathrm{D}(\calS)$. Then, for every vanishing sequence $(\tau_k)_k$
there exist a (not relabeled) subsequence and a $\sigma$-continuous curve
  $\rho : [0,T]\to X$
   such that 
 $\rho(0) =  \rho^\circ$, and 
 \begin{equation}
 \label{convergences-interpolants}
 \pwC \rho{\tau_k}(t),\,  \upwC \rho{\tau_k}(t),\, \pwM \rho{\tau_k}(t)  \weaksigmatoabs\rho(t) \qquad \text{for all } t \in [0,T],
 \end{equation}
and $\rho$ satisfies the Energy-Dissipation estimate
\begin{equation}
\label{limit-enineq}
\VarW \rho0t + \int_0^t \nuovorel(\rho(r)) \dd r +\calS(\rho(t)) \leq \calS(\rho_0) \qquad \text{for all } t \in [0,T].
\end{equation}
\end{theorem}
\begin{remark}
\label{rmk:generaliz-topol} 
\upshape
Theorem \ref{thm:abstract-GMM} could be extended to a topological space where the cost $\DVTn$ and the energy functional $\calS$ satisfy the properties 
listed at the beginning of the section.
\end{remark}

\begin{proof}
Consider a sequence $\tau_k \downarrow 0$ as $k\to\infty$.  

\emph{Step 1: Construct the limit curve $\ol\rho$. }
We first define the limit curve $\ol\rho$ on the set $A: = \{0\} \cup 
  N$, with $N$ a countable dense subset of $(0,T]$.  Indeed, in view of \eqref{aprio1}, 
  with a diagonalization procedure we find a function $\overline\rho : A \to X$ and a (not relabeled) subsequence such that 
\begin{equation}
  \label{step1-constr}
  \pwC \rho{\tau_k}(t) \weaksigmatoabs\overline\rho(t) \quad \text{for all } t \in A \quad \text{and}  \quad  \overline\rho(t) \in K  \text{ for all } t \in A .  
\end{equation}
In particular, $ \overline\rho(0)=\rho^\circ$. 

We next show that $\ol\rho$ can be uniquely extended to a $\sigma$-continuous curve $\ol\rho:[0,T]\to X$. Let $s,t\in A$ with $s<t$. By the lower-semicontinuity property~\eqref{lower-semicont} we have 
\begin{align*}
\DVT {t-s} {\ol\rho(s)} {\ol\rho(t)}
&\leq\liminf_{k\to\infty} 
\DVT {t-s} {\pwC\rho{\tau_k}(s)} {\pwC\rho{\tau_k}(t)}
\stackrel{\eqref{density-W}}\leq\liminf_{k\to\infty} 
\int_{\upwC {\sft}{\tau_{k}} (s)}^{\pwC {\sft}{\tau_{k}} (t)} \pwC {\mathsf{W}}{\tau_{k}} (r) \,\dd r\\
& \stackrel{(1)}{\leq} \liminf_{k\to\infty} \calS(\pwC \rho{\tau_{k}} (t_1) ) \stackrel{(2)}{\leq}  \calS(\rho_0),
\end{align*}
where {(1)} follows from \eqref{discr-enineq}
 (using the lower bound on $\calS$),
  and  {(2)} is due to the fact that  $t\mapsto \calS(\pwC \rho{\tau_{k}}(t))$ is nonincreasing. 

By the property~\eqref{e:psi6} of $\DVTn$, this estimate is a form of uniform continuity of $\ol\rho$, and we now use this to extend $\ol\rho$. Fix $t\in [0,T]\setminus A$, and choose a sequence $t_m\in A$, $t_m\to t$, with the property that $\ol\rho(t_m)$ $\sigma$-converges to some $\tilde\rho$. For any sequence $s_m\in A$, $s_m\to t$, we then have
\[
\sup_{m} \DVT {|t_m-s_m|}{\ol\rho(s_m)}{\ol\rho(t_m)} < \pinfty,
\] 
and since $|t_m-s_m|\to0$, property~\eqref{e:psi6} implies that ${\ol\rho(s_m)}\weaksigmatoabs \tilde \rho$. This implies that along any converging sequence $t_m\in A$, $t_m\to t$ the sequence  $\ol\rho(t_m)$ has the same limit; therefore there is a unique extension of $\ol\rho$ to $[0,T]$, that we again indicate by $\ol\rho$. By again applying the lower-semicontinuity property~\eqref{lower-semicont} we find that 
\[
\DVT {|t-s|}{\ol\rho(s)}{\ol\rho(t)} \leq \calS(\rho_0)
\qquad \text{for all }t,s\in [0,T], \ s\not=t,
\]
and therefore the curve $[0,T]\ni t\mapsto \ol\rho(t)$ is $\sigma$-continuous.

\bigskip

\emph{Step 2: Show convergence at all $t\in [0,T]$.}
Now fix $t\in [0,T]$; we show that $\pwC \rho{\tau_k}(t)$, $\upwC\rho{\tau_k}(t)$, and $\pwM\rho{\tau_k}(t)$ each converge to $\ol\rho(t)$.
Since $\pwC \rho{\tau_k}(t)\in K$, there exists  a convergent subsequence $\pwC\rho{\tau_{k_j}}(t)\weaksigmatoabs\tilde \rho$. Take any $s\in A$ with $s\not=t$. Then 
\begin{align*}
\DVT {|t-s|}{\tilde \rho}{\ol\rho(s)}
\leq \liminf_{j\to\infty} 
\DVT {|t-s|}{\pwC\rho{\tau_{k_j}}(t)}{\pwC\rho{\tau_{k_j}}(s)}
\leq \calS(\rho_0)\leq C,
\end{align*}
by the same argument as above. Taking the limit $s\to t$,  property~\eqref{e:psi6} and the continuity of $\ol\rho$ imply $\tilde\rho= \ol\rho(t)$. Therefore $\pwC\rho{\tau_{k_j}}(t)\weaksigmatoabs \ol\rho(t)$ along each subsequence $\tau_{k_j}$, and consequently also along the whole sequence $\tau_k$.

Estimates \eqref{discr-enineq-var} \&  \eqref{discr-enineq} also give at each $t\in (0,T]$
\[
 \limsup_{k\to\infty}    \DVTn(t-\upwC \sft{\tau_k}(t), \upwC \rho{\tau_k}(t), \pwC {\rho}{\tau_k}(t))   \leq   \calS(\rho_0),
\qquad  \limsup_{k\to\infty}    \DVTn(t-\upwC \sft{\tau_k}(t), \upwC \rho{\tau_k}(t), \pwM {\rho}{\tau_k}(t))   \leq   \calS(\rho_0),
\]
so that,  again using  the compactness information provided by  \eqref{aprio1} and property 
 \eqref{e:psi6}  of the cost $\DVTn$, it is immediate to conclude \eqref{convergences-interpolants}.

\bigskip

\emph{Step 3: Derive the energy-dissipation estimate.}
Finally, let us observe  that  
\begin{equation}
 \label{liminf-var}
 \liminf_{k\to\infty} \int_0^{\pwC {\sft}{\tau_k}(t)} \pwC {\mathsf{W}}{\tau_k}(r) \dd r \geq \VarW \rho0t \quad \text{for all } t  \in [0,T].
 \end{equation}
 Indeed, for any partition $\{ 0=t^0<\ldots <t^j <\ldots<t^M = t\}$ of $[0,t]$ we find that 
 \begin{align*}
 \sum_{j=1}^{M} \DVT{t^j-t^{j-1}}{\rho(t^{j-1})}{\rho(t^j)} &\stackrel{(1)}{\leq} \liminf_{k\to\infty}  \sum_{j=1}^{M} \DVT{\pwC {\sft}{\tau_k}(t^j)-{\pwC {\sft}{\tau_k}(t^{j-1})}}{\pwC \rho{\tau_k}(t^{j-1})}{\pwC \rho{\tau_k}(t^j)} \\
 &= \liminf_{k\to\infty} \int_0^{\pwC {\sft}{\tau_k}(t)} \pwC {\mathsf{W}}{\tau_k}(r) \,\dd r,
 \end{align*}
 with (1) due to 
 \eqref{lower-semicont}.
  Then \eqref{liminf-var} follows by taking the supremum over all partitions.
 On the other hand, by  Fatou's Lemma we find that 
 \[
  \liminf_{k\to\infty} \int_0^{\pwC {\sft}{\tau_k}(t)} \nuovo (\pwM \rho{\tau_k}(r)) \,\dd  r \geq \int_0^t \nuovorel(\rho(r)) \dd r,
 \]
 while the lower semicontinuity of $\calS$ gives 
 \[
   \liminf_{k\to\infty} \calS(\pwC \rho{\tau_k}(t))  \geq \calS(\rho(t))
 \]
 so that \eqref{limit-enineq} follows from taking  the $\liminf_{k\to\infty}$ in \eqref{discr-enineq} for $s=0$.  
\end{proof}

\subsection{Proof of Theorem~\ref{thm:construction-MM}}
\label{ss:pf-of-existence}

Having established the abstract compactness result of Theorem~\ref{thm:abstract-GMM}, we now apply this to the proof of Theorem~\ref{thm:construction-MM}. As described above, 
under \textbf{Assumptions~\ref{ass:V-and-kappa}, \ref{ass:Psi}, and~\ref{ass:S}} 
the conditions of Theorem~\ref{thm:abstract-GMM} are fulfilled, and Theorem~\ref{thm:abstract-GMM} provides us with a curve $\rho:[0,T]\to\calM^+(V)$ \OLI that is continuous with respect to setwise convergence \EEE such that 
\begin{equation}
\label{ineq:pf-abs-to-concrete}
\VarW \rho 0t+ \int_0^t \nuovorel(\rho(r)) \dd r +\calS(\rho(t)) \leq \calS(\rho_0) \qquad \text{for all } t \in [0,T].
\end{equation}
To conclude the proof of Theorem~\ref{thm:construction-MM}, we now show
that
\GGG
the Energy-Dissipation inequality \eqref{EDineq}
can be derived from \eqref{ineq:pf-abs-to-concrete}.
\nc

We first note that  Corollary~\ref{c:exist-minimizers} implies the  existence of a flux $\bj$ such that $(\rho,\bj)\in \CE 0T$ and $\VarW \rho 0T = \int_0^T \calR(\rho_t,\bj_t)\,\dd t$. 
Then from Corollary~\ref{cor:cor-crucial} below, we find that
$\nuovorel(\rho(r))\geq \Fish(\rho(r))$ for all  \RICKYNEW $r \in [0,T]$. \EEE  Combining these
results with~\eqref{ineq:pf-abs-to-concrete} we find the required
estimate~\eqref{EDineq}.

\bigskip

It remains to prove the inequality $\nuovorel\geq \Fish$, which follows from the corresponding inequality $ \nuovo\geq \Fish$ for the non-relaxed slope (Theorem~\ref{th:slope-Fish}) with the 
lower semicontinuity of $\Fish$ that is assumed in Theorem~\ref{thm:construction-MM}.
This is the topic of the next section.


\subsection{The generalized slope bounds the Fisher information}

We recall the definition of the generalized slope $\nuovo$ from~\eqref{def:nuovo}:
\[
\nuovo(\rho):= 
\limsup_{r \downarrow 0}\sup_{\mu \in X} \frac1r \Bigl\{ \calS(\rho) -\calS(\mu)-\DVTn(r,\rho,\mu)
\Bigr\} \,. 
\]
Given the structure of this definition, the proof of the inequality $\nuovo\geq \Fish$ naturally proceeds by constructing an admissible curve $(\rho, \bj)\in\CE0T$ such that $\rho\restr{t=0}=\rho$ and such that the  expression in braces can be related to $\Fish(\rho)$. 

For the systems of this paper, the construction of such a curve faces three technical difficulties: the first is that $\rho$ needs to remain nonnegative,  the second is that $\upphi'$ may be unbounded at zero, and the third is that the function $\rmD_\upphi(u,v)$ in~\eqref{eq:182} that defines $\Fish$ may be infinite when $u$ or $v$ is zero (see Example~\ref{ex:D}).


We first prove a  lower  bound for the generalized slope $\nuovo$  involving $\rmD_\upphi^-$, under the basic conditions on the $(\calS, \calR, \calR^*)$ system presented in 
Section \ref{s:assumptions}. \EEE
\begin{theorem}
\label{th:slope-Fish}
Assume \ref{ass:V-and-kappa}, \ref{ass:Psi}, and~\ref{ass:S}. Then 
\begin{equation}
 	\label{ineq:nuovo-FI1}
 	\nuovo (\rho) \geq
        \frac12\iint_\edg \rmD^-_\upphi(u(x),u(y))\,\tetapi(\dd x,\dd
        y)
        \quad \text{for all } \rho=u\pi \in  \rmD(\calS).
      \end{equation}
    \end{theorem}
 \begin{proof}
   Let us fix $\rho_0=u_0\pi\in \rmD(\calS)$, 
   a bounded measurable skew-symmetric map
   \begin{displaymath}
     \xi:\edg\to \R
     \quad\text{with }
     \xi(y,x)=-\xi(x,y),\quad
     |\xi(x,y)|\le \Xi<\infty
     \quad\text{for every $(x,y)\in \edg$,}
 \end{displaymath}
 the Lipschitz functions
 \OLI $q(r):=\min(r, 2(r-1/2)_+)$ \GGGO (approximating the identity far from $0$) and
 $h(r):=\max(0,\min(2-r, 1))$ (cutoff for $r\ge 2$), and the Lipschitz regularization of
 $\upalpha$
%
%
 \begin{displaymath}
   \upalpha_\eps(u,v):=\eps q(\upalpha(u,v)/\eps).
 \end{displaymath}
 We introduce the field $\rmG_\eps:\edg\times\R_+^2\to \R$ 
   \begin{equation}
     \label{eq:166}
     \rmG_\eps(x,y;u,v):=\xi(x,y)g_\eps(u,v)\,,
   \end{equation}
   where
   \[
   	g_\eps(u,v):=\upalpha_\eps(u,v)\,h(\eps \max(u, v))q(\min(1,\min(u,v)/\epsilon))\,,
   \]
   which vanishes if \OLI $\upalpha(u,v)<\eps/2$ or $\min(u,v)<\eps/2$ or $\max(u, v)\ge 2/\eps$, \GGGO
   and coincides with $\upalpha$ if \OLI $\upalpha\ge \eps$, $\min(u,v)\ge \eps$, and
   $\max(u, v)\le 1/\eps$. \GGGO
   Since $g_\eps$ is Lipschitz,
   it is easy to check that $\rmG_\eps$ satisfies all the assumptions
   (\ref{subeq:G}a,b,c,d)  and also \eqref{eq:123a} for $a=0$,
   since \OLI $0=g_\eps(0,0)\le g_\eps(0,v)$ for every $v\ge 0$ and every $(x,y)\in E$. \GGGO

   It follows that for every nonnegative $u_0\in L^1(X,\pi)$ there exists a unique
   nonnegative solution $u^\eps\in \rmC^1([0,\infty);L^1(V,\pi))$ of the Cauchy problem
   \eqref{eq:119-Cauchy} \OLI induced by $\rmG_\eps$ \GGGO with initial datum $u_0$ and 
   the same total mass. \OLI Henceforth, we set $\rho_t^\eps = u_t^\eps\pi$ for all $t\ge 0$. \GGGO

   Setting $\OLI 2 \GGGO\bj_t^\eps(\dd x,\dd y):=w_t^\eps(x,y) \tetapi(\dd x,\dd y) $,
   where $w_t^\eps(x,y):=\rmG_\eps(x,y;u_t(x),u_t(y))$,
   it is also easy to check that
   $(\rho^\eps,\bj^\eps)\in \CER 0T$,
   since $g_\eps(u,v)\le \upalpha(u,v)$ and \OLI
   \begin{displaymath}
     |w_t^\eps(x,y)|\le |\xi| \upalpha(u_t^\eps(x),u_t^\eps(y))\nchi_{U_\eps(t)}(x,y)\qquad \text{for $(x,y)\in E$}\,,
   \end{displaymath}
   where $U_\eps(t):=\{(x,y)\in \edg: g_\epsilon(u_t^\eps(x),u_t^\eps(y))>0\}$, thereby yielding
   \[
   		\Upsilon(u_t^\eps(x),u_t^\eps(y),w_t^\eps(x,y))\le \Psi(\Xi)\upalpha(2/\eps,2/\eps)\,.
   \]\GGGO
   Finally, 
   recalling \eqref{eq:102} and \eqref{eq:105},
   we get
   \begin{displaymath}
     |\RICKYNEW \rmB_\upphi \EEE(u_t^\eps(x),u_t^\eps(y),w_t^\eps(x,y))|
     \le \Xi  \big(\upphi'(2/\eps)-\OLI \upphi'(\eps/2) \GGGO \big) \upalpha(2/\eps,2/\eps).
   \end{displaymath}
Thus, we can apply Theorem \ref{th:chain-rule-bound}
   obtaining
   \begin{equation}
     \label{eq:167}
     \calS(\rho_0)-\calS(\rho_\tau^\eps)=
     -\OLI\frac{1}{2}\GGGO\int_0^\tau \iint_\edg\rmB_\upphi(u_t^\eps(x),u_t^\eps(y),w_t^\eps(x,y))\,\tetapi(\dd
     x,\dd y)\,\dd t,
   \end{equation}
   and consequently
   \begin{equation}
     \label{eq:168}
     \begin{aligned}
       \nuovo(\rho_0)&\ge \limsup_{\tau\downarrow0}\tau^{-1}
       \Big(\calS(\rho_0)-\calS(\rho_\tau^\eps)-
       \int_0^\tau \calR(\rho_t^\eps,\bj_t^\eps)\,\dd t\Big)
       \\&=
       \OLI\frac{1}{2}\GGGO\iint_\edg
       \Big(\rmB_\upphi(u_0(x),u_0(y),w_0^\eps(x,y))-
       \Upsilon(u_0(x),u_0(y), w_0^\eps(x,y))\Big)\,\tetapi(\dd x,\dd
       y).
     \end{aligned}
   \end{equation}
   
   Let us now set $\Delta_k$ to be the truncation of $\upphi'(u_0(x))-\upphi'(u_0(y))$ to $[-k,k]$, i.e. \OLI
   \[
   	\Delta_k(x,y):=\max\Bigl\{-k,\min\bigl[k,
   \upphi'(u_0(x))-\upphi'(u_0(y))\bigr]\Bigr\}\,,
   \]
   and $\xi_k(x,y):= (\Psi^*)'(\Delta_k(x,y))$ for each $k\in\N$. Notice that $\xi_k$ is a bounded measurable skew-symmetric map satisfying $|\xi_k(x,y)|\le k$ for every $(x,y)\in E$ and $k\in\N$. Therefore, inequality \eqref{eq:168} holds for $w_0^\eps(x,y) = \xi_k(x,y)\,g_\eps(u_0(x),u_0(y))$, $(x,y)\in E$. We then observe from Lemma~\ref{le:trivial-but-useful}\ref{le:trivial-but-useful:ineq} that
   \begin{equation}
     \label{eq:171}
     \begin{aligned}
     (\upphi'(u_0(x))-\upphi'(u_0(y)))\cdot \xi_k(x,y)&\ge
         \Delta_k(x,y)\xi_k(x,y) \\
         &= \Psi(\xi_k(x,y))+\Psi^*(\Delta_k(x,y))\,,
     \end{aligned}
   \end{equation}
   and from $g_\epsilon(u,v)\le \upalpha(u,v)$ that \GGGO
   \begin{equation}
     \label{eq:172}
     \begin{aligned}
     \Upsilon(u_0(x),u_0(y), w_0^\eps(x,y))
     &=\Psi\left(\frac{\xi_k(x,y) g_\eps(u_0(x),u_0(y))}{\upalpha(u_0(x),u_0(y))}\right)
     \upalpha(u_0(x),u_0(y)) \\
     &\le \Psi(\xi_k(x,y))\upalpha(u_0(x),u_0(y))\,.
     \end{aligned}
   \end{equation}
   Substituting these bounds in \eqref{eq:168} and passing to the
   limit
   as $\eps\downarrow0$ we obtain
   \begin{equation}
     \label{eq:173}
     \nuovo(\rho)\ge \OLI\frac{1}{2}\GGGO\iint_\edg \Psi^*(\Delta_k(x,y))\upalpha(u_0(x),u_0(y))\,
     \tetapi(\dd x,\dd y)\,.
   \end{equation}
   We eventually let $k\uparrow\infty$ and obtain
    \eqref{ineq:nuovo-FI1}. \EEE
   \end{proof}
 In the next proposition we finally bound  $\nuovo$ from below by the Fisher information,  by relying on the existence of a solution to 
 the $(\calS,\calR,\calR^*)$ system, as shown in Section~\ref{s:ex-sg}. \EEE
\begin{prop}
\label{p:slope-geq-Fish}
Let us suppose that for $\rho\in D(\calS)$
there exists a solution to the $(\calS,\calR,\calR^*)$ system.
Then the generalized slope bounds the Fisher information from above: 
\begin{equation}
	\label{ineq:nuovo-FI}
	\nuovo (\rho) \geq \Fish (\rho) \quad \text{for all } \rho \in  \rmD(\calS). 
	\end{equation}
\end{prop}
   \begin{proof}
     Let $\rho_t = u_t \pi$
     be a solution to the $(\calS,\calR,\calR^*)$ system
     with initial datum $\rho_0\in \rmD(\calS)$. Then, we can find \OLI a family $(\bj_t)_{t\ge 0}\in
     \calM(\edg)$ \EEE such that 
     $(\rho,\bj )\in \CE0\pinfty$ and 
 \[
 \calS(\rho_t) +\int_0^t \bigl[\calR(\rho_r,\bj_r)+\Fish(\rho_r)\bigr] \,\dd r = \calS(\rho_0)\qquad \text{for all }t\geq0.
 \]
 Therefore 
 \begin{align*}
 \nuovo(\rho_0) &\geq \liminf_{t\downarrow 0} 
 	\frac1t \Bigl[ \calS(\rho_0) - \calS(\rho_t) - \DVT t{\rho_0}{\rho_t}\Bigr]\\
 &\geq \liminf_{t\downarrow 0} 
 	\frac1t \Bigl[ \calS(\rho_0) - \calS(\rho_t) - \int_0^t \calR(\rho_r,\bj_r) \,\dd r\Bigr]
 = \liminf_{t\downarrow 0} 
 	\frac1t \int_0^t \Fish(\rho_r)\, \dd r\,.
 \end{align*}
 Since $u_t\to u_0$ in $L^1(V;\pi)$
 as $t\to0$ and since $\Fish$ is lower semicontinuous with respect to
 $L^1(V,\pi)$-convergence (\OLI see the proof of \EEE Proposition~\ref{PROP:lsc}),
  with a change of variables  we find
 \[
 \nuovo(\rho_0) \geq \liminf_{t\downarrow 0} \int_0^1 \Fish(\rho_{ts})\, \dd s
 \geq \Fish(\rho_0).
 \qedhere
 \]
 \end{proof}
 
 We then easily get   the desired lower bound for $\nuovorel$ in terms of $\Fish$, under the condition that the latter functional is lower semicontinuous (recall that Proposition \ref{PROP:lsc} provides
 sufficient conditions for the lower semicontinuity of $\Fish$): \EEE
 \begin{cor}
 \label{cor:cor-crucial}
 Let us suppose that \textbf{Assumptions~\ref{ass:V-and-kappa},
   \ref{ass:Psi}, \ref{ass:S}}
 hold and that $\Fish$ is lower semicontinuous
 with respect to ~setwise convergence. Then
 	\begin{equation}
 	\label{desired-rel}
 	\nuovorel (\rho) \geq  \Fish (\rho)
        \quad \text{for all } \rho \in \mathrm{D}(\calS). 	\end{equation}
 \end{cor}
  \noindent
 
 \begin{remark}
 \label{rmk:Mark}
 The combination of Theorem \ref{th:slope-Fish}, Proposition \ref{p:slope-geq-Fish}, and Corollary~\ref{cor:cor-crucial} illustrates why we introduced both $\rmD_\upphi$ and $\rmD^-_\upphi$. For the duration of this remark, consider both the functional $\Fish$ that is defined in~\eqref{eq:def:D} in terms of $\rmD_\upphi$, and a corresponding functional ~$\Fish^-$ defined in terms of the function $\rmD_\upphi^-$:
 \[
 	\Fish^- (\rho) :=
        \displaystyle \frac12\iint_\edg \OrmD^-_\upphi\bigl(u(x),u(y)\bigr)\,
        \tetapi(\dd x\,\dd y)
        \qquad   \text{for } 
		\rho = u\pi\,.
\]
In the two guiding cases of Example~\ref{ex:Dpm}, $\rmD_\upphi$ is convex and lower semicontinuous, but $\rmD_\upphi^-$ is only lower semicontinuous. As a result, $\Fish$ is lower semicontinuous with respect to setwise convergence, but $\Fish^-$ is not  (indeed, \EEE consider e.g.\ a sequence $\rho_n$ converging setwise to $\rho$, with $\dd\rho_n/\dd\pi$ given by characteristic functions of some sets $A_n$, where the sets $A_n$ are chosen such that for the limit the density $\dd\rho/\dd\pi$ is strictly positive and non-constant; then $\Fish^-(\rho_n)=0$ for all $n$ while $\Fish^-(\rho)>0$).
Setwise lower semicontinuity of $\Fish$ is important for two reasons: first, this is required for stability of solutions of the Energy-Dissipation balance under convergence in some parameter (evolutionary $\Gamma$-convergence), which is a hallmark of a good variational formulation; and secondly, the proof of existence using the   Minimizing-Movement approach \EEE  requires the bound~\eqref{desired-rel}, for which $\Fish$ also needs to be lower semicontinuous. This explains the importance of $\rmD_\upphi$, and it also explains why we  defined the Fisher information $\Fish$ in terms of  $\rmD_\upphi$ and not in terms of $\rmD_\upphi^-$.

On the other hand, $\rmD_\upphi^-$ is  straightforward to determine, and in addition the weaker control of $\rmD_\upphi^-$ is still sufficient for the chain rule: it is $\rmD_\upphi^-$ that appears on the right-hand side of~\eqref{eq:CR2}. Note that if $\rmD_\upphi^-$ itself is convex, then it coincides with $\rmD_\upphi$.
 \end{remark}

\appendix

\section{Continuity equation}
\label{appendix:proofs}
 In this Section  we
complete the analysis of the continuity equation by  carrying out the proofs of Lemma \ref{l:cont-repr} and Corollary
\ref{c:narrow-ct}. 
\begin{proof}[Proof of Lemma \ref{l:cont-repr}]
 The distributional identity \EEE  \eqref{eq:90} yields that
  for every $\zeta\in \Cb(V,\tau)$ 
the map
\[
  t\mapsto \rho_t(\zeta): =  \int_V \zeta(x) \rho_t (\dd x )\quad\text{belongs to $W^{1,1}(a,b)$},
\]
with distributional derivative 
\begin{equation}
  \frac{\dd}{\dd t}\rho_t(\zeta) =  \iint_{\edg} \dnabla\zeta\,\dd \bj_t=
  -\int_V \zeta\,\dd\odiv \bj_t
\quad\text{for almost all $ t \in [a,b]$.}\label{eq:13}
\end{equation}
Hence, setting $\mathfrak d_t:=|\odiv \bj_t|\in \calM^+(V)$, we have
\begin{equation}
\label{distributional-derivative}
 \left| \frac{\dd}{\dd t}\rho_t(\zeta)\right| \leq
\int_V |\zeta| \,\dd \mathfrak d_t\le 
\|\zeta\|_{\Cb(V)}|\odiv \bj_t|(V)\le
2\|\zeta\|_{\Cb(V)}| \bj_t|(\edg),
\end{equation}
where we used the fact that
\begin{equation*}
  \label{eq:53}
  \mathfrak d_t=|\sfx_\sharp (\bj_t-\sfs_\sharp \bj_t)|
  =
  |\sfx_\sharp \bj_t-\sfy_\sharp \bj_t|\le
  |  \sfx_\sharp \bj_t|+
  |\sfy_\sharp\bj_t|
\end{equation*}
which implies
\[
\mathfrak d_t (V) \le 
  2|\bj_t|(E).
\]
\nc
Hence, the set $L_\zeta$ of the Lebesgue points of $t\mapsto
\rho_t(\zeta)$ has full Lebesgue measure.
Choosing $\zeta\equiv 1$ one immediately recognizes that $\rho_t(V)$ is
(essentially) constant: it is not restrictive to normalize it to $1$ for convenience.
Let us now consider a countable set  \RICKYNEW $Z=\{\zeta_k\}_{k\in \N}$ \EEE of
uniformly bounded functions in $\Cb(V)$
such that
\begin{displaymath}
  |\zeta_k|\le 1,\quad
  \mathsf d(\mu,\nu):=\sum_{k=1}^\infty 2^{-k}\Big|\int_V \zeta_k\,\dd(\mu-\nu)\Big|
\end{displaymath}
is a distance inducing the weak topology in 
 $ \calM^+(V) $ \EEE
(see e.g.~\cite[\S\,5.1.1]{AmbrosioGigliSavare08}).
By introducing the set $L_Z := \bigcap_{\zeta \in Z} L_\zeta$,
it follows from \eqref{distributional-derivative}
that
\begin{equation}
  \label{eq:10}
  \mathsf d(\rho_s,\rho_t)\le 2 \int_s^t |\bj_r|(\edg)\, \dd r
\end{equation}
showing that the restriction of $\rho$ to $L_Z$ is
continuous in  $\Dom(V)$.  \EEE
Estimate~\eqref{distributional-derivative} also shows that for all $s,t\in L_Z$ with $s \leq t$ we have
\begin{equation}
  |\rho_t(\zeta) {-} \rho_s(\zeta) | \leq
  \int_s^t \int_V |\zeta|\,\dd\mathfrak d_r\,\dd r\le 2\| \zeta\|_{\Cb(V)}
  \int_s^t |\bj_r|(\edg)\, \dd r \qquad
 \text{for all } \zeta \in \Cb(V).
 \label{eq:11}
\end{equation}
Taking the supremum with respect to ~$\zeta$ we obtain
\begin{equation}
  \label{eq:12}
  \|\rho_t-\rho_s\|_{TV}\le 2
  \int_s^t |\bj_r|(\edg)\, \dd r \qquad
  \text{ and  all } s,t\in L_Z,\ s \leq t,
\end{equation}
 which shows that 
 the measures $(\rho_t)_{t\in L_Z}$ are uniformly continuous
 with respect to the total variation metric in  $\Dom(V)$ \EEE
 and thus can be extended to an absolutely continuous curve $\tilde\rho\in
 \mathrm{AC}(I;\Dom(V))$ \EEE
 satisfying
 \eqref{eq:12} for every $s,t\in I$.

 When $\varphi\in \Cb(V)$, \eqref{2ndfundthm}  immediately follows from
  \eqref{eq:13}. By a standard argument based on
  the functional monotone class Theorem \cite[\S 2.12]{Bogachev07}
  we can extend the validity of
  \eqref{2ndfundthm} to every bounded Borel function.

  If $\varphi\in \mathrm C^1([a,b];\Bb(V))$,
  combining \eqref{eq:13} and the fact that
  the map $t\mapsto \int_V \varphi(t,x)\,\tilde\rho_t(\dd x)$
  is absolutely continuous we easily get \eqref{maybe-useful}. 
\end{proof}

\begin{proof}[Proof of Corollary \ref{c:narrow-ct}]
Keeping the same notation of the previous proof,  
if we define
  \begin{displaymath}
    \gamma:=\rho_0+\int_0^T \mathfrak d_t \,\dd t
  \end{displaymath}
  then the estimate \eqref{distributional-derivative} shows that
  \begin{displaymath}
    \rho_t(B)\le \gamma(B)\quad\text{for every }B\in \frB,
  \end{displaymath}
  thus showing that $\rho_t=\tilde u_t\gamma$ for every $t\in [0,T]$
  and
  \begin{equation}
    \label{eq:54}
    \|\rho_t-\rho_s\|_{TV}=\int_V |\tilde u_t-\tilde
    u_s|\,\dd\gamma\le
    2\int_s^t |\bj_r|(\edg)\,\dd r
    \quad\text{for every }0\le s<t\le T.
  \end{equation}
\end{proof}
 We conclude with  a result on the decomposition of 
the measure   \RICKYNEW $  \bj -\symmap \bj = 2\tj $ \EEE into its positive and negative part. 
\begin{lemma}
  \label{le:A1}
  If $\bj \in \calM(\edg)$ and we set
  \begin{equation}
    \label{eq:14}
   \RICKYNEW    \bj^+:=(   \bj {-}\symmap \bj )_+,\quad
    \bj^-:=(   \bj {-}\symmap \bj )_-,  \EEE
  \end{equation}
  then we have
  \begin{equation}
    \label{eq:15}
    \bj^-=\symmap \bj^+,\quad
  \RICKYNEW  \odiv \bj^+= \odiv \bj. \EEE
  \end{equation}
  When $\bj$ is skew-symmetric, we also have
  \begin{equation}
    \label{eq:16}
   \RICKYNEW  \bj^+=2\bj_+,\quad \bj^-=-2\bj_-. \EEE
  \end{equation}
\end{lemma}
\begin{proof} 
  \RICKYNEW By definition, we have 
  $\bj^+=2\tj_+$,
  $\bj^-=2\tj_-$. 
Furthermore, $ \tj =-\symmap \tj =\symmap \tj _--\symmap \tj_+$,
where the first equality follows from the fact that $\tj$ is skew-symmetric. 
Since $\symmap \tj_-\perp \symmap \tj_+$ we deduce that
  $\symmap \tj_+=\tj_-,$ $\symmap \tj_-=\tj_+$ and $\tj=\tj_+-\symmap \tj_+$, so that 
$\odiv \bj = \odiv \tj =2\odiv \tj_+=\odiv \bj^+$. 
\end{proof}
\EEE

\section{Slowly increasing superlinear entropies}
 The main result of this Section is Lemma \ref{le:slowly-increasing-entropy} ahead, invoked in the proof of Proposition 
\ref{prop:compactness}. It provides the construction of a \emph{smooth} function estimating the entropy  density $\upphi$ from below
and such that 
the function $(r,s) \mapsto \Psi^*(A_\upomega(r,s)) \alpha(r,s)$ fulfills a suitable bound,
 cf.\ \eqref{eq:41} ahead. Prior to that, we prove the preliminary
 Lemmas \ref{le:alpha-behaviour} and \ref{le:sub} below. \EEE
\begin{lemma}
  \label{le:alpha-behaviour}
  Let us suppose that $\upalpha$ satisfies Assumptions
  \ref{ass:Psi}. Then for every $a\ge0$
  \begin{equation}
    \label{eq:40}
    \lim_{r\to\pinfty}\frac{\upalpha(r,a)}r=
    \lim_{r\to\pinfty}\frac{\upalpha(a,r)}r=0.
  \end{equation}
\end{lemma}
\begin{proof}
  Since $\upalpha$ is symmetric it is sufficient to prove the first
  limit.  
  Let us first observe that the concavity of $\upalpha$ yields
  the existence of the limit since the map $r\mapsto
  r^{-1}(\upalpha(r,a)-\upalpha(0,a))$ is decreasing, so that
  \begin{displaymath}
    \lim_{r\to\pinfty}\frac{\upalpha(r,a)}r=
    \lim_{r\to\pinfty}\frac{\upalpha(r,a)-\upalpha(0,a)}r=
    \inf_{r>0}\frac{\upalpha(r,a)-\upalpha(0,a)}r.
  \end{displaymath}
  Let us call $L(a)\in \R_+$ the above quantity. 
  The inequality (following by the concavity of $\upalpha$ and the fact
  that $\upalpha(0,0)\ge0$)
  \begin{equation}
    \label{eq:concave-elementary}
    \upalpha(r,a)\le \lambda\upalpha(r/\lambda,a/\lambda)\quad\text{for
      every }\lambda\ge1
  \end{equation}
  yields
  \begin{equation}
    \label{eq:32}
    L(a)=\lim_{r\to\pinfty}\frac{\upalpha(r,a)}r\le 
    \lim_{r\to\pinfty}\frac{\upalpha(r/\lambda,a/\lambda)}{r/\lambda}=
    L(a/\lambda) \quad\text{for
      every }\lambda\ge1.
  \end{equation}
  For every $b\in (0,a)$ and $r>0$, setting $\lambda:=a/b>1$, we thus obtain
  \begin{displaymath}
    L(a)\le L(b)\le \frac{\upalpha(r,b)-\upalpha(0,b)}r    
  \end{displaymath}
  Passing first to the limit as $b\downarrow0$ and using the
  continuity of $\upalpha$ we get
  \begin{displaymath}
    L(a)\le \frac{\upalpha(r,0)-\upalpha(0,0)}r    \quad\text{for every $r>0$}.
  \end{displaymath}
  Eventually, we pass to the limit as $r\uparrow\pinfty$ and we get
  $L(a)\le \upalpha^\infty(1,0)=0$ thanks to \eqref{alpha-0}.
\end{proof}
\begin{lemma}
  \label{le:sub}
  Let $f:\R_+\to \R_+$ be an increasing continuous function
  and $f_0\ge0$ with
  \begin{equation}
    \label{eq:43}
    \lim_{r\to\pinfty}f(r)=\sup f=\pinfty,\qquad
    \liminf_{r\downarrow0}\frac{f(r)-f_0}r\in (0,\pinfty].
  \end{equation}
  Then for every $g_0\in [0,f_0]$
  there exists a $\rmC^\infty$ concave function $g:\R_+\to\R_+$ such that
  \begin{equation}
    \label{eq:35}
    \forall\,r\in \R_+:g(r)\le f(r),\qquad
    g(0)=g_0,\qquad
    \lim_{r\to\pinfty}g(r)=\pinfty.
  \end{equation}
\end{lemma}
\begin{proof}
  By subtracting $f_0$ and $g_0$ from $f$ and $g$, respectively,
  it is not restrictive to assume $f_0=g_0=0$.
  We will use a recursive procedure to construct a concave piecewise-linear function $g$ satisfying
  \eqref{eq:35}; a standard regularization yields a
  $\rmC^\infty$ map.

  We set
  \begin{equation}
    \label{eq:44}
    a:=\frac 13\liminf_{r\downarrow0}\frac{f(r)}r,\quad
    x_1:=\sup\Big\{x\in (0,1]:f(r)\ge 2 ar\text{ for
      every }r\in (0,x]\Big\},
  \end{equation}
  and $\delta:=ax_1.$
  We consider a strictly increasing sequence $(x_n)_{n\in\N}$, $n\in \N$, defined by
  induction starting from $x_0=0$ and $x_1$ as in \eqref{eq:44},
  according to
  \begin{equation}
    \label{eq:25}
    x_{n+1}:=\min\Big\{x\ge 2x_n-x_{n-1}:
    f(x)\ge f(x_n)+\delta\Big\},\quad n\ge 1.
  \end{equation}
  Since $\lim_{r\to\pinfty}f(r)=\pinfty$, the minimizing set in
  \eqref{eq:25} is closed and not empty, so that the algorithm is well
  defined. It yields a sequence $x_n$ satisfying
  \begin{equation}
    \label{eq:33}
    x_{n+1}-x_n\ge x_{n}-x_{n-1},\quad
    x_{n+1}\ge x_n+\delta\quad\text{for every }n\ge0,
  \end{equation}
  so that $(x_n)_{n\in \N}$ is strictly increasing and unbounded,
  and induces a partition $\{0=x_0<x_1<x_1<\cdots<x_n<\cdots\}$ of $\R_+$.
  We can thus consider the piecewise linear function $g:\R_+\to\R_+$
  such that 
  \begin{equation}
    \label{eq:36}
    g(x_n):=n\delta,\quad g((1-t)x_n+t x_{n+1}):=(n+t)\delta\quad
    \text{for every }n\in \N,\ t\in[0,1].
  \end{equation}
We observe that   $g$ is increasing, $\lim_{r\to\pinfty}g(r)=\pinfty$ and it is concave since
  \begin{displaymath}
    \frac{g(x_{n+1})-g(x_n)}{x_{n+1}-x_n}=
    \frac{\delta}{x_{n+1}-x_n}\topref{eq:33}\le
    \frac\delta{x_{n}-x_{n-1}}=\frac{g(x_{n})-g(x_{n-1})}{x_{n}-x_{n-1}}.
  \end{displaymath}
 Furthermore,  $g$ is also dominated by $f$:
  in the interval $[x_0,x_1]$ this follows by \eqref{eq:44}.
  For $x\in [x_n,x_n+1]$ and $n\ge 1$, we observe that
  \eqref{eq:25} yields
  $f(x_{n+1})\ge f(x_n)+\delta$ so that by induction $f(x_n)\ge
  (n+1)\delta$; on the other hand
  \begin{displaymath}
    \text{for every $x\in [x_{n},x_{n+1}]$}:\quad
    g(x)\le g(x_{n+1})=(n+1) \delta\le f(x_n)\le f(x).\qedhere
\end{displaymath}  
\end{proof}
\begin{lemma}
  \label{le:slowly-increasing-entropy}
  Let $\Psi^*,\upalpha$ be satisfying
  Assumptions \ref{ass:Psi} and
  let $\upbeta:\R_+\to\R_+$ be a convex superlinear function with
  $\upbeta'(r)\ge \upbeta_0'>0$ for a.e.~$r\in \R_+$.
  Then, there exists a $\rmC^\infty$ convex superlinear function
  $\upomega:\R_+\to\R_+$ such that
  \begin{equation}
    \label{eq:41}
     \upomega(r)\le \upbeta(r),\qquad
    \Psi^*(\upomega'(s)-\upomega'(r))\upalpha(s,r)\le
    r+s\quad\text{for every }r,s\in \R_+.
  \end{equation}
\end{lemma}
\begin{proof}
  By a standard regularization, \GGG 
  we can always approximate $\upbeta$ by a smooth convex superlinear function
  $\tilde\upbeta\le \upbeta$ whose derivative is strictly positive, so that \EEE 
  it is not restrictive to assume that
  $\upbeta$ is of class $\rmC^2$. 
  Let us set 
  $r_0:=\inf\{r>0:\Psi^*(r)>0\}$ and   
  let $P:(0,\pinfty)\to (r_0,\pinfty)$ be  the inverse map of $\Psi^*$:
  $P$ is continuous, strictly increasing, and of class $\rmC^1$.
  
  Since $\upalpha$ is concave, the function $x \mapsto \upalpha(x,1)/x$ is
  nonincreasing in $(0,\pinfty)$; we can thus define the nondecreasing
  function $Q(x):=P(x/\upalpha(x,1))$
  and the function
  \begin{displaymath}
    \gamma(x):=2g_0+\int_1^x \min(\upbeta''(y),Q'(y))\,\dd y\quad
    \text{for every }x\ge 1,\quad
    g_0:=\frac12\min(\upbeta_0', Q(1))>0.
  \end{displaymath}
  By construction $\gamma(1)=2g_0= \min(\upbeta'_0,Q(1))\le \upbeta'(1)$
  so that $\gamma(x)\le \min(\upbeta'(x),Q(x))$ for every $x\ge1$.
  We eventually set
  \begin{displaymath}
    f(t):=\frac{\rme^t}{\gamma(\rme^t)}\quad t\ge0.
  \end{displaymath}
  Clearly, we have 
     $f(0)=2g_0$. Furthermore,  we combine the estimate
     $\gamma(\rme^t) \leq Q(\rme^t) = P(\rme^t/\upalpha(\rme^t,1))$ with the facts that 
     $\rme^t/\upalpha(\rme^t,1) \to +\infty$ as $t\to +\infty$, thanks to Lemma \ref{le:alpha-behaviour}, 
     and that $P$ has sublinear growth at infinity, being the inverse function of $\Psi^*$. 
     All in all, we conclude that 
    \[
    \lim_{t\to\pinfty}f(t)=\pinfty.
  \]
  Therefore, we are in a position to 
\EEE apply Lemma \ref{le:sub}, obtaining an increasing concave function
  $g:\R_+\to\R_+$ such that
  $g_0=g(0)\le g(t)\le f(t)$ and $\lim_{t\to\pinfty}g(t)=\pinfty$.
  Since $g(0)\ge0$, the concaveness of $g$ yields
  $g(t'')-g(t')\le g(t''-t')$ for every $0\le t'\le t''$, so that 
  the function $h(x):=g(\log (x\lor 1))$  satisfies
  $h(x)=g_0\le \upbeta'(x)$ for $x\in [0,1]$, and
  \begin{equation}
    \label{eq:48}
    h(z)\le \min(\upbeta'(z),Q(z))\quad\text{for every }z\ge 1,\quad
    h(y)-h(x)\le 
    h(y/x)\quad
    \text{for every }0< x\le y.
  \end{equation}
  In fact, if $x\le 1$ we get
  \begin{displaymath}
    h(y)-h(x)=h(y)-g_0\le h(y)\le h(y/x)
  \end{displaymath}
  and if $x\ge 1$ we get
  \begin{displaymath}
     h(y)-h(x)\le  g(\log y)-g(\log x)\le g(\log y-\log x)=g(\log(y/x))=
    h(y/x).
  \end{displaymath}
  Let us now define the convex function $\upomega(x):=\int_0^x h(y)\,\dd y$ with
  $\upomega(0)=0$ and $\upomega'=h$. In particular $\upomega(x)\le
  \upbeta(x)$
  for every $x\ge 0$.

  It remains to check the second inequality of \eqref{eq:41}.
  The case $r,s\le 1$ is trivial since
  $\upomega'(s)-\upomega'(r)=h(r)-h(s)=0$.
  We can also consider the case
  $\upomega'(r)\neq \upomega'(s)$ and $\upalpha(r,s)>0$;
  since \eqref{eq:41} is also symmetric, it is not restrictive to
  assume $r\le s$; by continuity, we can assume $r>0$.

  Recalling that $\upalpha(s,r)\le r\upalpha(s/r,1)$ if $0<r\le s$,
  and $(r+s)/r>s/r$,
  \eqref{eq:41} is surely satisfied if
  \begin{equation}
    \label{eq:49}
    \Psi^*(\upomega'(s)-\upomega'(r))\upalpha(s/r,1)\le s/r\quad
    \text{for every }0<r<s.
  \end{equation}
  Recalling that $\upomega'(s)-\upomega'(r)\le \upomega'(s/r)$ by
  \eqref{eq:48} and $\Psi^*$ is nondecreasing,
  \eqref{eq:49} is satisfied if
  \begin{equation}
    \label{eq:50}
    \Psi^*(\upomega'(s/r))\upalpha(s/r,1)\le s/r\quad
    \text{for every }0<r<s.
  \end{equation}
  After the substitution $t:=r/s$, \eqref{eq:50} corresponds to
  \begin{equation}
    \label{eq:51}
    \upomega'(t)\le P(t/\upalpha(t,1))=Q(t)\quad\text{for every }t\ge 1,
  \end{equation}
  which is a consequence of the first inequality of \eqref{eq:48}.
\end{proof}
\section{Connectivity by curves of finite action}
\label{s:app-2}
Preliminarily, with  the reference measure $\pi \in \calM_+(V)$ and with the    `jump equilibrium rate' $\tetapi$ from \eqref{nu-pi} we associate  the `graph divergence' operator
$\odiv_{\pi,\tetapi}:  L^p(\edg;\tetapi) \to L^p(V;\pi) $, $p\in [1,\pinfty]$, defined as the transposed of the `graph gradient' $\ona:L^q(V;\pi) \to L^{q}(\edg;\tetapi)$, with $q = p'$.
Namely
\[
\begin{gathered}
\text{for }  \zeta \EEE \in  L^p(\edg;\tetapi), \qquad  \xi = - \odivn_{\pi,\tetapi}( \zeta \EEE)\qquad \text{if and only if}
\\
 \int_{V} \xi(x) \omega(x) \pi(\dd x) = \int_{\edg}  \zeta \EEE(x,y) \ona \omega(x,y) \tetapi(\dd x, \dd y) \quad \text{for all } \omega \in  L^q(V;\pi) 
 \end{gathered}
\]
or, equivalently, 
\begin{equation}
\label{in-terms-of-measures}
\xi \pi = - \odiv( \zeta \EEE\tetapi) 
\end{equation}
 (with $\odiv$ the divergence operator from \eqref{eq:def:ona-div}) in the sense of measures. 
\par
We can now first address the connectivity problem in the very specific setup
\begin{equation}
\label{alpha-equiv-1}
\upalpha(u,v) \equiv 1 \qquad \text{for all } (u,v) \in [0,\pinfty) \times [0,\pinfty).
\end{equation}
Then, the action functional $\int \calR$ is translation-invariant. 
Let us consider two measures $\rho_0, \rho_1 \in \calM_+(V)$ such that for $i\in \{0,1\}$ there holds $\rho_i = u_i \pi$ with $u_i \in L_+^p(V;\pi)$ for some
$p\in (1,\pinfty)$. 
Thus, we look for curves $\rho \in  \ADM 0\tau{\rho_0}{\rho_1}$, with finite action, such that $\rho_t \ll \pi$,  with density $u_t$, for almost all $t\in (0,\tau)$. 
Consequently, any flux $(\bj_t)_{t\in (0,\tau)}$ shall satisfy
$\bj_t \ll \tetapi$ for a.a.\ $t\in (0,\tau)$ (cf.\ Lemma   \ref{l:alt-char-R}). 
 Taking into account \eqref{in-terms-of-measures},  the continuity equation reduces to 
\begin{equation}
\label{cont-eq-densities}
\dot{u}_t = - \odivn_{\pi,\tetapi}( \zeta_t \EEE) \qquad \foraa\, t \in (0,\tau)
\end{equation}
with $  \zeta_t =   \frac{\dd \bj_t}{\dd \tetapi}$. Furthermore, we look for a connecting curve $\rho_t = u_t \pi$ with $u_t = (1{-}t)u_0 +t u_1$, so that 
\eqref{cont-eq-densities} becomes 
$
- \odivn_{\pi,\tetapi}( \zeta_t \EEE)  \equiv u_1 -u_0$.  Hence, we can restrict to flux densities that are constant in time, i.e.\  $\zeta_t \equiv \zeta$ with $\zeta\in  L^p(\edg;\tetapi)$. \EEE
In this specific context, and if we further confine the discussion to the case $\Psi(r) = \frac1p |r|^p $ for $p\in (1,\pinfty)$, 
the minimal action problem becomes 
\begin{equation}
\label{minimal-action}
\inf\left \{ \frac1p \int_{\edg} |w|^p \tetapi(\dd x,\dd y) \, : \   w = 2\zeta \in  L^p(\edg;\tetapi), \ - \odivn_{\pi,\tetapi}(\zeta)  \equiv u_1 -u_0 \right\} \EEE
\end{equation}
Now, by a general duality result on linear operators, 
the operator $- \odiv_{\pi,\tetapi}: L^p(\edg;\tetapi) \to L^p(V;\pi)$ is surjective
 if and only if the graph gradient $\ona:L^q(V;\pi) \to L^{q}(\edg;\tetapi)$ fulfills 
the following property:
\[
\exists\, C>0 \ \ \forall\, \xi \in L^q(V;\pi)  \text{ with } \int_V \xi \pi(\dd x) =0 \text{ there holds } \| \xi\|_{L^q(V;\pi)} \leq C \| \ona \xi\|_{L^q(\edg;\tetapi)},
\]
namely the $q$-Poincar\'e inequality \eqref{q-Poinc}. We can thus conclude the following result.
\begin{lemma}
\label{l:intermediate-conn}
Suppose that $\upalpha \equiv 1$, that $\Psi$ has $p$-growth (cf.\ \eqref{psi-p-growth}), and that 
 the measures $(\pi,\tetapi) $ satisfy a  $q$-Poincar\'e inequality for $q=\tfrac p{p-1}$. 
Let $\rho_0, \rho_1  \in  \calM^+(V) $ be given by $\rho_i = u_i \pi$, with positive
$u_i \in L^p(V; \pi)$, for $i \in \{0,1\}$. Then, for every $\tau\in (0,1)$
we have $\DVT{\tau}{\rho_0}{\rho_1}<\pinfty$. If $\Psi(r) = \frac1p |r|^p$,  $q$-Poincar\'e inequality  is also necessary for
having $\DVT{\tau}{\rho_0}{\rho_1}<\pinfty$.
\end{lemma}
We are now in a position to carry out the 
\begin{proof}[Proof of Proposition \ref{prop:sufficient-for-connectivity}] Assume that $\rho_0(V) = \int_V u_0 (x) \pi(\dd x) = \pi(V)$. Hence, it is sufficient to provide a solution for the connectivity problem between $u_0$ and $u_1 \equiv 1$. We may also assume without loss of generality that
$\upalpha(u,v) \geq \upalpha_0(u,v)$ with $\upalpha_0(u,v) = c_0 \min(u,v,1)$ for some $c_0>0$,  so that 
\begin{equation}
\label{inequ}
\begin{aligned}
\Psi\left(\frac w{\upalpha(u,v)} \right)\upalpha(u,v) \leq \Psi\left(\frac w{\upalpha_0(u,v)} \right)\upalpha_0(u,v) &  \leq C_p \left( 1+ \left| \frac w{\upalpha_0(u,v)} \right|^p \right) \upalpha_0(u,v)  
\\
 & \leq
C_p c_0 + C_p |w|^p (\upalpha_0(u,v))^{1-p}\,,
\end{aligned}
\end{equation}
where the first estimate follows from the convexity of $\Psi$ and the fact that $\Psi(0)=0$, yielding that $\lambda \mapsto \lambda \Psi(w/\lambda)$ is non-increasing. 
It is therefore sufficient to consider the case in which $c_0=C_p=1$,  $\upalpha_0(u,v) = \min(u,v,1)$, and to solve the connectivity problem for $\tilde\Psi(r) = \frac1p |r|^p$. By Lemma \ref{l:intermediate-conn}, we may first find $w\in L^p(\edg;\tetapi)$
 solving the minimum problem \eqref{minimal-action} in the case $\upalpha \equiv 1$, so that  the flux density  $\zeta_t \equiv  \frac12 w$ \EEE is associated with the curve  $u_t = (1{-}t)u_0 +t u_1$, 
 $t\in [0,\tau]$.
  Then, we fix an exponent $\gamma>0$ and we consider the rescaled curve $\tilde{u}_t: = u_{t^\gamma}$, 
 that fulfills
 $\partial_t \tilde{u}_t = -  \odivn_{\pi,\tetapi}(\tilde{\zeta}_t) $
 with $\tilde{\zeta}_t  = \frac12 \tilde{w}_t = \frac12 \gamma t^{\gamma-1} w$. \EEE
 Moreover, 
 \begin{align*}
 \upalpha_0(\tilde{u}_t (x), \tilde{u}_t (y)) &= \min \{  (1{-}t^\gamma) u_0(x) + t^\gamma u_1(x), (1{-}t^\gamma) u_0(y) + t^\gamma u_1(y), 1\} \\
 &\geq \min(t^\gamma, 1) = t^\gamma
 \end{align*}
 since $u_1(x) = u_1(y) =1$.
 By \eqref{inequ} we thus get
 \[
 \begin{aligned}
 &
 \int_{\edg} \Psi \left(\frac{\tilde{w}_t(x,y)}{\upalpha(\tilde{u}_t (x), \tilde{u}_t (y))} \right) \upalpha(\tilde{u}_t (x), \tilde{u}_t (y)) \tetapi(\dd x,\dd y)  
 \\ &\quad \leq 
 C_p c_0 \tetapi(\edg)+
 \int_{\edg} \gamma^p t^{p(\gamma{-}1)} |w(x,y)|^p t^{\gamma(1{-}p)} \tetapi(\dd x,\dd y)  =   C_p c_0 \tetapi(\edg)+ \gamma^p t^{\gamma-p} \|w\|_{L^p(\edg;\tetapi)}^p\,.
 \end{aligned}
 \]
 Choosing $\gamma>p-1$ we conclude that 
 \[
 \int_0^\tau \int_{\edg} \Psi \left(\frac{\tilde{w}_t(x,y)}{\upalpha(\tilde{u}_t (x), \tilde{u}_t (y))} \right) \upalpha(\tilde{u}_t (x), \tilde{u}_t (y)) \tetapi(\dd x,\dd y)  <\pinfty\,
 \]
hence  $\ADM 0\tau{\rho_0}{\rho_1} \neq \emptyset$. 
 \end{proof}
 \EEE

\bibliographystyle{alpha}
\bibliography{ricky_lit}

\end{document}